\documentclass[a4paper,twoside,10pt]{article}
\usepackage[a4paper,left=3cm,right=3cm, top=3cm, bottom=3cm]{geometry}
\usepackage[latin1]{inputenc}
\usepackage{mathrsfs}
\usepackage{url}
\usepackage{graphicx}
\usepackage{epstopdf}
\usepackage{amsmath}
\usepackage{amsthm}
\usepackage{amssymb}
\usepackage{hyperref}
\usepackage{subfigure}
\usepackage{float}
\usepackage{cite}
\usepackage{color}
\usepackage[abs]{overpic}
\usepackage{verbatim}
\usepackage[font=footnotesize,labelfont=bf]{caption}

\restylefloat{table}
\theoremstyle{plain}
\newtheorem{thm}{Theorem}[section]
\newtheorem{cor}[thm]{Corollary}
\newtheorem{lem}[thm]{Lemma}

\theoremstyle{definition}

\theoremstyle{remark}
\newtheorem{remark}{Remark}

\newcommand{\tauh}{\mathcal T _{h}} %decomposizione tau_{h}
\newcommand{\tauhtilde}{\widetilde{\mathcal T} _{h}} %sottodecomposizione tau_{h} tilde in triangoli
\newcommand{\underh}{_{h}} %semplicemente mette pedice h
\newcommand{\Nu}{ \mathcal V }%nu maiuscola che non esiste
\newcommand{\Epsilon}{ \mathcal E } %epsilon maiuscola
\newcommand{\Vh}{ V_{h} }  %spazio virtuale dipendente da h
\newcommand{\VhK}{\Vh(K)}
\newcommand{\ah}{ a \underh } %forma bilineare discreta VEM dipendente da h
\newcommand{\fh}{ f \underh } %loading term discreto VEM dipendente da h
\newcommand{\uh}{ u \underh } %soluzione VEM dipendente da h
\newcommand{\Pinablap}{\Pi ^{\nabla}_p} %operatore Pi nabla con dipendenza esplicita da p
\newcommand{\Pinabla}{\Pi ^{\nabla}_p} %operatore Pi nabla con o senza dipendenza da p
\newcommand{\NVK}{N_{K}} %numero di vertici dell'elemento K
\newcommand{\bigfh}{ \mathcal{F} \underh } % piccola costante su stime loading term senza dipendenza da p
\newcommand{\upi}{u_{\pi}} % approssimante polinomiale a tratti di VEM VOLLEY
\newcommand{\phihp}{\varphi_p^h} %funzione polinomiale a tratti continua approssimante hp in BabuSuri
\newcommand{\boldalpha}{\boldsymbol \alpha} %alpha bold
\newcommand{\el}{\ell} %parametro per la convergenza in p, regolarità della funzione a sinistra
\newcommand{\diam}{\text{diam}} %diametro
\newcommand{\const}{C} %costante generica
\newcommand{\AhK}{A_h^K} %matrice di stiffness locale
\newcommand{\ahK}{a_h^K} %forma bilineare locale
\newcommand{\tildephi}{\widetilde \varphi} %phi tildato
\newcommand{\hatphi}{\widehat \varphi} %phi cappuccio
\newcommand{\Omegaext}{\Omega_{\text{ext}}} %Omega esteso serve per le stime su funzioni analitiche con il nostro trick
\newcommand{\uhat}{\widehat u} % u cappuccio
\newcommand{\Qhat}{\widehat Q} % Q cappuccio
\newcommand{\Babuska}{Babu\v{s}ka} %macro Babuska
\author{
L. Beir\~ao da Veiga
\thanks{Dipartimento di Matematica,  Universit\`a degli Studi di Milano, E-mail: {\tt lourenco.beirao@unimi.it}},
A. Chernov
\thanks{Institut f\"ur Mathematik, C. von Ossietzky Universit\"at Oldenburg, E-mail: {\tt alexey.chernov@uni-oldenburg.de}},
L. Mascotto
\thanks{Dipartimento di Matematica,  Universit\`a degli Studi di Milano, E-mail: {\tt lorenzo.mascotto@unimi.it}},
A. Russo
\thanks{Dipartimento di Matematica,  Universit\`a degli Studi di Milano-Bicocca, E-mail: {\tt alessandro.russo@unimib.it}}
}

\date{}

\title{Basic principles of hp Virtual Elements on quasiuniform meshes}

\begin{document}
%%%%%%%%%%%%%%%%%%%%%%%%%%%%%%%%%%%%%

\maketitle
\begin{abstract}
In the present paper we initiate the study of $hp$ Virtual Elements. We focus on the case with uniform polynomial degree across the mesh and derive theoretical convergence estimates that are explicit both in the mesh size $h$ and in the polynomial degree $p$
in the case of finite Sobolev regularity. Exponential convergence is proved in the case of analytic solutions.
The theoretical convergence results are validated in numerical experiments. Finally, an initial study on the possible choice of local basis functions is included.
\end{abstract}

% -----------------------------------------------------------------
\section {Introduction}
% -----------------------------------------------------------------

The Virtual Element Method (VEM) is a very recent generalization of the Finite Element Method, introduced in \cite{VEMvolley}, that responds to the increasing interest in using general polyhedral and polygonal meshes, also including non convex elements and hanging nodes.
The main idea of VEM is to use richer local approximation spaces that include (but are typically not restricted to) polynomial functions and, most importantly, avoid the explicit integration of the associated shape functions. Indeed, the operators and matrices appearing in the problem are evaluated by introducing an innovative construction that only requires an implicit knowledge of the local shape functions.
By following such developments, the VEM acquires very interesting properties and advantages with respect to more standard Galerkin methods, yet still keeping the same implementation complexity. For instance, in addition to allowing for polygonal and polyhedral meshes, it can handle approximation spaces of arbitrary $C^k$ global regularity on unstructured meshes.

Although the Virtual Element Method has been applied to a large range of problems (a non exhaustive list being
%%%%%
\cite{VEMvolley, VEMelasticity, Brezzi-Marini:2012, hitchhikersguideVEM, equivalentprojectorsforVEM, streamvirtualelementformulationstokesproblempolygonalmeshes, VEMchileans, BeiraodaVeiga-Manzini:hkp, Brezzi-Falk-Marini, Paulino-VEM, Berrone-VEM, Topology-VEM}),
%%%%%
all the present works on VEM are focused, both theoretically and numerically, on the $h-$behaviour of the method. In other words, the convergence properties of the schemes are investigated assuming that the polynomial degree $p$ is fixed and only the mesh is refined. On the other hand, looking at the Finite Element literature, a very successful approach in applications is to allow for a variable value of $p$ and to focus not only in the accuracy that can be obtained by reducing the mesh size $h$, but also by increasing $p$. As in the FEM literature, we here refer to such approach as $hp$ analysis; we mention for instance
\cite{SchwabpandhpFEM,Melenk_hpsingularperturbations, BabuSurioptimalconvergenceestimatepmethods, babuskasurihpversionFEMwithquasiuniformmesh,PolyDG-1}
as very short list of papers and books among the very large literature of $hp$ FEM.

The aim of the present paper is to initiate the study of $hp$ Virtual Element Methods. The first motivation of such study is to show that the powerful $hp$ methodology can be adopted also in the framework of Virtual Elements. The second, but not secondary,
motivation is that we believe that combining the huge mesh flexibility of VEM with the advantages of a full (possibly adaptive) $hp$ method can yield a very efficient and competitive methodology.

The present contribution focuses on the initial foundations of such ambitious plan, mainly in terms of convergence estimates.
We restrict our attention to a two dimensional scalar elliptic model problem (as in \cite{VEMvolley}) and assume a polynomial degree $p$ that is the same on all elements of the (quasi uniform) mesh.
First of all, we prove fundamental convergence results (and the associated interpolation estimates) that is explicit in both $h$ and $p$.
As a second result, we show that also Virtual Elements can attain exponential convergence when the target solution is analytic on a suitable (small) extension of the domain.
We then explore numerically the behaviour of Virtual Elements in terms of $p$, both in the case of solutions with finite Sobolev regularity and for analytic solutions. %%
In the Appendix, we start to explore another interesting issue of $hp$ elements, that is the choice of the basis and the condition number of the associated stiffness matrix.
Note that, since in this work we focus on scalar problems in a planar two-dimensional domain, direct solvers can generally be used and the condition number issue is not of primary importance.
Indeed, it is mainly the stability of the solver that determines the best attainable accuracy, as we show in the numerical tests.
Nevertheless, in order to answer to some natural questions (such as: how do Legendre bases cope on general polygons?) we decided to include an initial study related to the choice of the basis.

The paper is organized as follows. After presenting the continuous
model problem, in Section \ref{sectionVEMforthePoissonproblem} we
make a brief review of the Virtual Element Method. Afterwards, in
Section \ref{sectionapproximationresult} we present the
theoretical error estimates and in Section
\ref{sectionnumericalresults} we develop the associated numerical
tests. Finally, the Appendix follows. 

% -----------------------------------------------------------------
\section {The model problem}
\label {sectionthecontinuousproblem}
% -----------------------------------------------------------------

Let $\Omega$ be a simply connected polygonal domain and let
$\Gamma$ be its boundary. Let $H^l(\omega)$, with $l\in
\mathbb N _0$ and $\omega$ open measurable set, denote the usual
Sobolev space with square integrable weak derivatives of order
$l$; let $||\cdot||_{l,\Omega}$ and $|\cdot|_{l,\Omega}$ denote
the associated norm and seminorm, respectively (see
\cite{adams_fournier_sobolev_2003}). Let $f\in L^2(\Omega)$. We
consider the two dimensional Poisson problem with homogeneous
Dirichlet boundary conditions:
\begin {equation} \label{continuousproblemstrongformulation}
-\Delta u =f  \text{  in  } \Omega,\quad u=0 \text{  on  } \Gamma.
\end{equation}
We set $V:= H_0^1 (\Omega)$ and we consider the weak formulation of problem (\ref{continuousproblemstrongformulation})
\begin {equation} \label{continuousproblemweakformulation}
\text{find } u\in V \text{ such that } a(u,v)=(f,v)_{0,\Omega},\quad \forall v\in V,
\end{equation}
where $(\cdot,\cdot)_{0,\Omega}$ is the $L^2$ scalar product on $\Omega$ and $a(\cdot, \cdot):= (\nabla     \cdot , \nabla \cdot)_{0,\Omega}$.

It is well known that problem (\ref{continuousproblemweakformulation}) is well-posed (see for instance \cite{BrennerScott})
since the bilinear form $a$ is continuous and coercive (i.e. $a(v,v) \ge  \alpha ||v||^2_{1,\Omega}$, where $\alpha >0$) thanks to the Poincar\'e inequality.
Throughout this paper, $\const$ denotes a positive constant whose dependence on certain parameters will be made explicit where necessary.

% -----------------------------------------------------------------
\section {Virtual Elements for the Poisson problem}
\label{sectionVEMforthePoissonproblem}
% -----------------------------------------------------------------

In the present section, we introduce a Virtual Element Method for the Poisson problem (\ref{continuousproblemweakformulation}) based on polygonal meshes.
Let $\{ \tauh \} \underh$ be a sequence of decompositions of $\Omega$ into non-overlapping polygonal elements $K$ of diameter $h_K = \diam(K) := \sup\{|\mathbf{x} - \mathbf{y}|~:~\mathbf{x},\mathbf{y} \in K\}$. The characteristic mesh size is denoted by $h:=\max\{h_K~:~K \in \tauh\}$. Let $\Nu _h$ and $\Epsilon \underh$ be the sets of all vertices and edges in the mesh $\tauh$ respectively. Moreover, we denote by $\Nu_h^b := \Nu _h \cap \partial \Omega$ be the set of all boundary vertices and by $\Epsilon \underh ^K$ the set of edges $e$ of an element $K \in \tauh$.

Henceforth, we assume that there exist two positive real numbers $\gamma$ and $\widetilde \gamma$ such that the sequence of decompositions satisfies the following:
\begin {enumerate}
\item[$(\mathbf{D0})$] the decomposition $\tauh$ is made of a finite number of simple polygons of diameter $h_K$,
\item[$(\mathbf{D1})$] for all $K \in \tauh$, $K$ is star-shaped with respect to a ball of radius $\ge h_K \gamma$,
\item[$(\mathbf{D2})$]  for all $K \in \tauh$, the distance between any two vertices of $K$ is $\ge h_K \widetilde {\gamma}$.
\end {enumerate}
To every edge $e \in \Epsilon \underh $ we associate a tangential vector $\boldsymbol{\tau}_e$ and a normal unit vector $\mathbf{n}_e$ obtained by a counter-clockwise rotation of $\boldsymbol{\tau}_e$.

We split the bilinear form $a$ as a sum of local contributions
\[
a(u,v) := \sum _{K \in \tauh} a ^K (u,v),\quad \forall u,\, v \in V,
\]
with $a^K(u,v) := (\nabla u, \nabla v)_{0,K}$. 

It was shown in \cite{VEMvolley} that it is possible to build:
\begin {itemize}
\item $\Vh(K)$, a finite dimensional subspace of $H_0^1(\Omega)|_K$;
\item a symmetric local bilinear form $\ah^K: \Vh(K) \times \Vh(K) \rightarrow \mathbb R$;
\item $\Vh$ a finite dimensional subspace of $H_0^1(\Omega)$ such that $\Vh|_K = \Vh (K)$;
\item a symmetric bilinear form $\ah : \Vh \times \Vh \rightarrow \mathbb R$, of the form $\ah (u_h,v_h)= \sum _{K \in \tauh} \ah^K(u_h,v_h) $, $\forall u_h,v_h \in \Vh$ where $\ah^K : \Vh|_K \times \Vh|_K \rightarrow \mathbb R$ are local symmetric bilinear forms;
\item an element $\fh \in \Vh '$ and a duality pairing $\langle \cdot ; \cdot \rangle \underh$
\end{itemize}
in such a way that the resulting discrete problem
\begin {equation} \label{discreteproblem}
\text{find } \uh \in \Vh \quad \text{such that}\quad \ah (\uh, v_{h}) = \langle \fh ; v_{h} \rangle\underh,\quad \forall v_h \in \Vh
\end{equation}
has a unique solution $\uh \in \Vh$ which is close to the solution $u$ of the original problem \eqref{continuousproblemweakformulation}. More precisely, when $u \in H^k(\Omega)$ the error in the energy norm admits the upper bound
\begin{equation}\label{hconvergenceVEMvolley}
\text{if } u\in H^{k}(\Omega),\, k\ge 1,\quad  |u - \uh|_{1,\Omega} \le \const h^{k-1} |u|_{k,\Omega}
\end {equation}
where the constant $C=C(p)$ depends implicitly on the (fixed)
polynomial degree $p$ but not on the characteristic mesh size $h$. We now briefly review the local Virtual
Spaces introduced in \cite{VEMvolley}. Let $K \in \tauh$ and let
$p\in \mathbb N$, $p\ge 1$. Let $\mathbb P _p(e)$ and $\mathbb
P_{p-2}(K)$ be respectively the set of polynomials of degree $p$
over the edge $e$ and of degree $p-2$ over the polygon $K$,
with the convention $\mathbb P_{-1}(K) = \{ 0 \}$. With the space of contunuous piecewise polynomials over the boundary of each element $K$
\[
\mathbb B _p (\partial K) := \left\{  v_h\in \mathcal{C}^0(\partial K) \mid v_h|_e \in \mathbb P _ p (e),\, \forall e \in \Epsilon \underh ^K    \right\},
\]
we define the local Virtual Element spaces
\begin{equation} \label{virtualspace}
\Vh (K) := \left\{  v_h \in H^1(K) \mid \Delta v_h \in \mathbb P _{p-2} (K), v_h|_{\partial K} \in \mathbb B_p (\partial K)            \right\} .
\end{equation}
Observe that $\mathbb P _{p}(K)\subseteq \Vh (K)$ for any $K \in \tauh$. For any fixed function $v\in \VhK$ we identify the following set of local degrees of freedom:
\begin {itemize}
\item the values of $v$ at vertices of $K$,
\item the values of $v$ at $(p-1)$ internal nodes of each edge $e$ $\in \Epsilon \underh^K$ (for instance at the internal Gau\ss-Lobatto nodes, as done in \cite{VEMvolley} and  \cite{hitchhikersguideVEM}),
\item the internal moments $\frac{1}{|K|} \int _K q_{\mathbf{\boldalpha}} v_h dx$, where $\left\{ q_{\boldalpha}~:~ 0\leq  |\boldalpha| \leq  p(p-1)/2\right\}$ is a basis for $\mathbb P _{p-2}(K)$.
For instance, \cite{VEMvolley} and \cite{hitchhikersguideVEM} employed a basis of shifted and scaled monomials:
let $\mathbf{x_K}$ and $h_K$ be the barycenter and the diameter of $K$ respectively, then
$q_{\mathbf{\boldalpha}}(\mathbf{x}):=\left( \frac{\mathbf{x} - \mathbf{x_K}}{h_K}  \right)^{\mathbf{\boldalpha}}$ for any
$\mathbf{\boldalpha}= (\alpha _1, \alpha _2) \in \mathbb N_0 ^2$ such that $|\alpha|:=\alpha _1 + \alpha _2 \le p-2$.
\end{itemize}
\noindent The global Virtual Space is obtained by the continuous matching of the local spaces over the element boundaries
\[
\Vh := \left\{  v_h\in C^0(\overline \Omega) ~:~ v_h |_K \in \VhK, v_h|_{\partial \Omega} = 0\right\} \subset H^1_0(\Omega)
\]
with the natural definition of the global degrees of freedom from the local ones. It was shown in \cite{VEMvolley} that, given $K\in \tauh$, the bilinear forms $\ah ^K$ must satisfy the two following assumptions.
\begin{enumerate}
\item[(\textbf{A1})] \textbf{$p$-consistency}: $\forall K \in \tauh$ it holds that
\begin{equation} \label{pconsistencyequation}
a^K(q,v_h)=\ah ^K (q,v_h),\quad \forall v_h\in \VhK,\; \forall q\in \mathbb P _{p}(K);
\end{equation}
\item[(\textbf{A2})] \textbf{stability}: $\forall K \in \tauh$ there exist two constants $0< \alpha_{*} < \alpha^{*}<\infty$ such that
\begin{equation} \label{stabilityequation}
\alpha_{*} a^K(v_h,v_h) \le \ah ^K (v_h,v_h) \le \alpha ^{*} a^K(v_h,v_h), \quad \forall v_h \in \VhK.
\end{equation}
\end{enumerate}
Let $\varphi$ be a sufficiently regular function, e.g. $\varphi \in H^1(K)$. We introduce the local averaging operator:
\begin{equation} \label{overlinephi}
\overline \varphi := 
\begin{cases}
\frac{1}{|\partial K|} \int _{\partial K} \varphi (x)dx,& \text{if } p=1\\
\frac{1}{|{K}|} \int _{K} \varphi (x)dx,& \text{if } p>1\\
\end{cases},
\end{equation}

Having this, we introduce the projection operator 
$\Pinablap : \VhK \rightarrow \mathbb P_p(K)$ as follows: for any $v_h \in \VhK $ its projection $\Pinablap v_h \in \mathbb P_p(K)$ is the unique polynomial satisfying,
\begin{equation} \label{definitionPinabla}
\begin{split}
a^K(\Pinabla v_h -v_h , q)=0&\quad \forall q \in \mathbb P _p (K),\\
\overline {\Pinabla v_h-v_h} = 0& \quad\text{where the averaging operator is defined in } (\ref{overlinephi}).
\end{split}
\end{equation}
Then, a candidate bilinear form $\ah$ satisfying (\textbf{A1}) and (\textbf{A2}) can be sought in the form:
\[
\ah ^K (u_h,v_h) = a^K (\Pinabla u_h, \Pinabla v_h) + S^K(u_h - \Pinabla u_h,v_h - \Pinabla v_h),\quad \forall u_h, v_h \in \VhK,
\]
\noindent where $S^K$ is a positive definite bilinear form satisfying
\begin{equation} \label{stabilizingterm}
c_0 a^K(v_h,v_h)\le S^K(v_h,v_h) \le c_1 a^K(v_h,v_h),\quad  \forall v_h \in \VhK, \text{ such that } \Pinabla v_h=0,
\end{equation}
for some positive constants $c_0$ and $c_1$ independent on $h$, $p$ and $K$.\\
The global discrete bilinear form reads
\[
\ah (u_h,v_h) := \sum _{K\in \tauh} \ah ^K (u_h,v_h), \quad \forall u_h,v_h \in \Vh.
\]
Finally, we recall from \cite{VEMvolley} a possible choice for the loading term.
Let $P ^{0,K}_{p-2}$ and $P ^{0,K}_{0}$ be the $L^2$-projector on polynomials of degree $p-2$ and $0$ respectively over the polygon $K$ and let the averaging operator be defined in (\ref{overlinephi}).
Then, we may define
\begin {equation} \label{discreteloadingterm}
\langle \fh, v_h \rangle \underh :=
\begin{cases}
\displaystyle \sum \limits_{K\in \tauh} \int _ K \left[ P ^{0,K}_{p-2} f \right] v_h dx, \quad \forall v_h\in \Vh& \text{if } p\ge 2,\\
\displaystyle \sum \limits_{K\in \tauh} \int _ K \left[ P ^{0,K}_{0} f \right] \overline{v_h} dx, \quad \forall v_h\in \Vh& \text{if } p=1.\\
\end{cases}
\end{equation}
\noindent \\
Under the above choices for $\Vh$, $\ah$ and $\fh$,
\cite{VEMvolley} guarantees well-posedness and $h$-convergence
\eqref{hconvergenceVEMvolley}.
\begin{remark} \label{remark Pinabla operator} As shown in \cite{VEMvolley}, the projection operator $\Pinabla$ in \eqref{definitionPinabla} is computable using the degree of freedom values,
without the need of any further information on the virtual shape functions. We finally note that the definition in \eqref{overlinephi} is not the only possible one; other (computable) choices could be used instead.
\end{remark}
% -----------------------------------------------------------------
\section{Approximation results} \label{sectionapproximationresult}
% -----------------------------------------------------------------

In this section, we give a convergence result for the error of the Virtual Element Method measured in the energy norm in terms of both $h$ and $p$.
% -----------------------------------------------------------------
\subsection{Auxiliary approximation results} \label{subsectionsomeapproximationresults}
% -----------------------------------------------------------------
Let $u$ and $u_h$ be the solutions of (\ref{continuousproblemweakformulation}) and (\ref{discreteproblem}) respectively, and denote by $S_h^{p,-1} (\tauh)$ the space of the piecewise discontinuous polynomials of degree $p$ over the decomposition $\tauh$.
Given $u\in H^1(K)$, $\forall K\in \tauh$ , we define the broken $H^1$-seminorm as
\begin{equation} \label{H1brokenSobolevseminorm}
|u|_{h,1,\Omega} = \sum _{K\in \tauh} \left( |u|^2_{1,K} \right)^{\frac{1}{2}}.
\end{equation}
Let $\bigfh$ be the smallest constant satisfying
\begin{equation} \label{definitionbigfhp}
(f,v_h)_{0,\Omega} - \langle \fh , v_h \rangle \underh \le \bigfh |v_h|_{1,\Omega}, \quad \forall v_h \in V_h.
\end{equation}
Then the following best approximation estimate holds (see Theorem 3.1 in \cite{VEMvolley})
\begin{equation} \label{stimevecchieVEMvolley}
|u-\uh| _{1, \Omega} \le \const \left(  \inf_{\upi \in S_h ^{p,-1}(\tauh)} |u-\upi|_{1,h,\Omega} + \inf _{u_I \in \Vh} |u-u_I|_{1,\Omega} +  \bigfh     \right),
\end{equation}
where $\const$ is a constant depending only on $\alpha_{*}$ and $\alpha^{*}$ from assumption (\textbf{A2}).
In what follows, we shall derive estimates for the three terms in (\ref{stimevecchieVEMvolley}) that are explicit in both $h$ and $p$.\\
% -----------------------------------------------------------------
\subsubsection {Polynomial approximation term} \label{subsubsectionfirstterm}
% -----------------------------------------------------------------
We start by bounding the term $|u-\upi|_{h,1,\Omega}$.
In order to derive the bound, we need to prove a generalized-polygonal version of a classic result, namely Lemma 4.1 in \cite{babuskasurihpversionFEMwithquasiuniformmesh}.
In this lemma, it was shown the existence of a sequence of polynomials which approximate $H^k$ functions over the triangular and square reference elements.
We extend this result for generic polygons having the unit diameter. Thus, we are ready to show the following
\begin{lem}\label{lemmapestimatesonboundeddomain}
Let $\widehat{K} \subseteq \mathbb R ^2$ be a polygon with diam($\widehat{K}$)=1.
Moreover, assume that $\widehat K$ is star-shaped with respect to a ball of radius $\ge \gamma$ and the distance between any two vertices of $\widehat K$ is $\ge  \widetilde \gamma$,
$\gamma$ and $\widetilde \gamma$ being the constants introduced in assumptions (\textbf{D1}) and (\textbf{D2}) of Section \ref{sectionVEMforthePoissonproblem}.
Then, there exists a family of projection operators $\{ \widehat{\Pi}_{\widehat K, p}   \}$, $p=1,2,\dots$ with $\widehat{\Pi}_{\widehat K, p} : H^{k+1}(\widehat K) \rightarrow \mathbb P_p (\widehat K)$
such that, for any $0 \le \el \le k+1$, $\widehat u \in H^{k+1}(\widehat K)$, $k\in \mathbb N$, it holds
\begin{equation} \label{estimatelemma41}
|| \widehat u - \widehat{\Pi}_{\widehat K, p} \widehat u ||_{\el,\widehat K} \le \const  p^{-(k+1-\el)} || \widehat u ||_{k+1,\widehat K}
\end{equation}
\noindent with $\const$ a constant independent on $u$ and $p$.
\end{lem}
\begin{proof}
We assume without loss of generality that $\mathbf{x_{\widehat{K}}}$, the barycenter of $\widehat{K}$, coincides with the origin \textbf{0}.
For a given $r>0$, we define
\begin {equation} \label{definizionequadrator}
R(r):= \left\{  (x,y)\in \mathbb R ^2 \mid |x|<r, \, |y|<r     \right\}.
\end{equation}
Thanks to the fact that diam($\widehat K$)=1 and $\mathbf
x_{\widehat K}=\mathbf{0}$, we have $R(1)\supset \overline
{\widehat K}$.\\
Let $r_0 >1$. Then, it obviously holds $\overline{\widehat{K}} \subset R(r_0)$.
We note that $\partial \widehat K$ is Lipschitz; consequently, using \cite{steinsingularintegrals},
there exists $E: H^{k+1}(\widehat{K}) \rightarrow H^{k+1}(R(2r_0))$ extension operator such that $E\widehat u=0$ on $R(2r_0)\setminus R(\frac{3}{2}r_0)$ and $|| E \widehat u   ||_{k+1,R(2r_0)} \le \const || \widehat u  ||_{k+1,\widehat K}$.
A careful inspection of Theorem 5 in Chapter VI of \cite{steinsingularintegrals} shows that the constant $\const$ depends only on $k$ and on the ``worst angle'' value
$$
\theta_{\widehat{K}} = \min_{\theta \in {\cal A}_{\widehat{K}}}  \min{ \{ \theta, 2\pi-\theta \} },
$$
where ${\cal A}_{\widehat{K}}$ denotes the set of the (amplitude of) internal angles of $\widehat{K}$.
In particular, the constant $\const$ may explode when $\theta_{\widehat{K}} \rightarrow 0$.
It is easy to check that, under the regularity hypotheses on $K$,
the angle parameter $\theta_{\widehat{K}}$ is bounded from below by a constant depending only on $\gamma, \widetilde \gamma$. Thus the constant $\const$ can be bounded in terms of $k$ and $\gamma, \widetilde \gamma$.
Therefore, it holds $|| E \widehat u   ||_{k+1,R(2r_0)} \le \const( k, \gamma, \widetilde \gamma) || \widehat u  ||_{k+1,\widehat K}$.
The remaining part of the proof, that is based on the approximation of the extended function $E \widehat u $, follows exactly the same steps as in \cite{babuskasurihpversionFEMwithquasiuniformmesh}, Lemma 4.1, and is therefore not shown.
\end{proof}

\noindent Using this result, we are able to give a generalized-polygonal version of Lemma 4.5 of \cite{babuskasurihpversionFEMwithquasiuniformmesh},
which will play the role of local $hp$ estimate result on $|u-\upi|_{1,K}$, where $K$ is a polygon of the decomposition $\tauh$.
\begin{lem} \label {lemmaphBrambleHilbert}
Let $K\in \tauh$ satisfying assumptions (\textbf{D1}) and (\textbf{D2}) and $u\in H^{k+1}(K)$. Then there exists a sequence of projection operators $\{ \Pi_{K, p}^h   \}$, $p=1,2,\dots$, with $\Pi_{K, p}^h : H^{k+1}(K) \rightarrow \mathbb P_p (K)$
such that for any $0 \le \el \le k+1$, $k\in \mathbb N$
\begin{equation} \label{formulalemma42}
|u-\Pi ^h_{K,p}u|_{\el,K} \le \const \frac{h_K^{\mu + 1 -\el}}{p^{k+1-\el}} || u ||_{k+1,K},
\end{equation}
\noindent where $\mu = \min (p, k)$ and $\const$ is independent on $u$, $h$ and $p$.
\end{lem}
\begin{proof}
We consider the mapping $F(\mathbf{x})=
\frac{1}{h_K}(\mathbf{x}-\mathbf{x_K})$. Let $\widehat K = F(K)$,
where $h_K$ denotes the barycenter of $K$. Obviously
diam($\widehat K$)=1 and the barycenter of $\widehat K$ is in the origin, $\mathbf{x_{\widehat
K}}=\mathbf{0}$. Let $\widehat{\Pi}_{\widehat K, p} \widehat u$ be
the sequence of approximating polynomials of degree $p$,
introduced in Lemma \ref{lemmapestimatesonboundeddomain}. We set
$\Pi_{K, p}^hu$ be the push forward of the above sequence with
respect to the transformation $F$, i.e. $\Pi^h_{K,p}u =
(\widehat{\Pi}_{\widehat K, p}(\widehat u)) \circ F$, where
$\widehat \varphi = \varphi \circ F^{-1}$ for a sufficiently
regular function $\varphi$. Then, it is easy to check, by a simple
change of variables argument, that
\[
|u-\Pi_{K,p}^h u|_{\el,K} \le \const h_K ^ {1-\el} | \widehat u - \widehat {\Pi}_{\widehat K, p} \widehat u |_{\el,\widehat K},
\]
where $\const$ is a constant independent on $K$ (hence on $\widehat K$), $h$, $u$ and $p$;
besides, $\const$ is independent also on $\el$, $\gamma$ and $\widetilde \gamma$, thanks to the fact that $F$ is the composition of a translation with a dilatation.\\
We apply Lemma \ref{lemmapestimatesonboundeddomain} and we obtain, by adding and subtracting any $\widehat q \in \mathbb P_p(\widehat K)$,
\begin{equation} \label{espressionedipendentedaKcappuccio}
|u-\Pi_{K,p}^h u|_{\el,K} \le   \const h_K ^{1-\el} || (\widehat u - \widehat q) -  \widehat {\Pi}_{\widehat K, p} (\widehat u - \widehat q) ||_{\el,\widehat K}\le
                                                    \const(k,\gamma, \widetilde \gamma) \frac{h_K ^{1-\el}}{p^{k+1-\el}} ||\widehat u - \widehat q||_{k+1,\widehat K},
                    \quad \forall \widehat q \in \mathbb P _p (\widehat K),
\end{equation}
where $\const$ in the right hand side of  \eqref{espressionedipendentedaKcappuccio} is a constant depending on $k, \gamma, \widetilde \gamma$,
since clearly $\widehat{K}=F(K)$ still satisfies the shape regularity assumptions with the same constants $\gamma, \widetilde \gamma$.
Using the classical Scott-Dupont theory (see e.g. \cite{BrennerScott} or \cite{ScottDupontBrambleHilbert}) and a scaling argument, bound (\ref{espressionedipendentedaKcappuccio}) yields
\begin{equation} \label{stimeqkonK}
|u - \Pi ^h_{K,p}u|_{\el,K} \le \const \frac{h_K^{1-\el}}{p^{k+1-\el}} \left( \sum_{i=\mu}^{k+1} |\widehat u|^2_{i,\widehat K}  \right)^{\frac{1}{2}} \le \const\frac{h_K^{\mu + 1 - \el}}{p^{k+1-\el}} ||u||_{k+1,K}, \quad \mu = \min (p, k),
\end{equation}
where $\const$ is independent on $u$, $p$ and $h$.
\end{proof}
\begin{remark} \label{remarkdanormaaseminorma}
We note that if $k \le p$ then the classical Bramble-Hilbert Lemma allows to take the seminorm in the right hand side of \eqref{formulalemma42}, yielding
\[
|u-\Pi ^h_{K,p}u|_{\el,K} \le \const(k,\gamma, \widetilde \gamma) \frac{h_K^{k+1-\el}}{p^{k+1-\el}} | u |_{k+1,K}.
\]
\end{remark}
We are now able to give a global estimate on $|u-\upi|_{h,1,\Omega}$ in \eqref{stimevecchieVEMvolley}, where $\upi \in S_h^{p, -1}(\tauh)$, $S_h^{p, -1}(\tauh)$ being defined at the beginning of Section \ref{subsectionsomeapproximationresults}.
In fact, by choosing $\upi|_K= \Pi _{K,p}^h u$ for all $K\in \tauh$ and recalling the shape regularity properties (\textbf{D1})-(\textbf{D2}), we obtain:
\begin{equation} \label{stimaupi}
\begin{split}
&|u-\upi|_{h,1,\Omega} \le \const_1 \frac{h^{\mu}}{p^{k}} ||u||_{k+1,\Omega}, \quad \mu = \min (p,k), \\
&|u-\upi|_{h,1,\Omega} \le \const_2 \frac{h^{k}}{p^{k}} |u|_{k+1,\Omega}, \quad  \text{for } p \geq k\\
\end {split}
\end{equation}
where $\const_1$ and $\const_2$ are two constants independent on $u$, $p$ and $h$.\\

% -----------------------------------------------------------------
\subsubsection{Virtual Interpolation term} \label{subsectionsecondterm}
% -----------------------------------------------------------------

We turn now to the term $|u-u_I|_{1,\Omega}$ in (\ref{stimevecchieVEMvolley}).
Preliminarily, we observe that (\textbf{D1}) and (\textbf{D2}), defined in Section \ref{sectionVEMforthePoissonproblem}, imply that
there exists $\tauhtilde$, an auxiliary conformal triangular mesh that refines $\tauh$, obtained by connecting, for all $K\in \tauh$, the $\NVK$ vertices to the center of the ball that realizes assumption (\textbf{D1}) for $K$.
Moreover, it is easy to check that each triangle $T\in \tauhtilde$ is uniformly shape regular. 

Let $\widetilde S ^{p,0}_h(\widetilde \tauh)$ be the set of continuous piecewise polynomials of degree $p$ over the auxiliary triangular decomposition introduced above.
It is well known (see Theorem 4.6 in \cite{babuskasurihpversionFEMwithquasiuniformmesh}) that there exists $\phihp\in \widetilde S ^{p,0}_h(\widetilde \tauh)$ such that for any $u\in H^{k+1}(\Omega)$, $k\in \mathbb R$
\begin{equation} \label{stimaBabuSurimeshtriangolare}
\begin{split}
& || u-\phihp  ||_{1,\Omega} \le \const_1 \frac{h^{\mu }}{p^{k}} || u ||_{k+1,\Omega},\quad k>\frac{1}{2}, \\
& | u-\phihp  |_{1,\Omega} \le \const_2 \frac{h^{k }}{p^{k}} | u |_{k+1,\Omega},\quad k>\frac{1}{2},\; p\ge k \\
\end{split}
\end{equation}
\noindent where $\const_1$ and $\const_2$ are two constants independent on $u$, $p$ and $h$ and where $\mu =\min (p, k)$.

Now, we use $\phihp$ in (\ref{stimaBabuSurimeshtriangolare}) to construct an approximant $u_I \in \Vh$ of $u$.
For this purpose, we modify a particular technique introduced in \cite{VEMchileans}.
\begin{lem} \label{lemmastimahpsudominio}
Under (\textbf{D1}) and (\textbf{D2}), for all $u\in H^{k+1}(\Omega)$, $k\in \mathbb N$, there exists $u_I\in \Vh$ such that
\begin{equation} \label{stimaestensionecileni}
|u-u_I|_{1,\Omega} \le \const \frac{h^{\mu }}{p^{k}} ||u||_{k+1,\Omega},\quad \mu = \min (p , k),
\end{equation}
\noindent where $\const$ is independent on $h$, $p$ and $u$.
\end{lem}
\begin {proof}
Let $\upi$ be the function defined in (\ref{stimaupi}). Let $\phihp$ be the function described in (\ref{stimaBabuSurimeshtriangolare}).
For each $K\in \tauh$, we define $u_I|_K$ the solution of the following problem:
\begin{equation}    \label{definitionuI}
\left\{\begin{array}{ll}
-\Delta u_I = -\Delta \upi & \text{in } K\\
u_I=\phihp & \text{on } \partial K\\
\end{array} \right. .
\end{equation}
It is easy to check that $u_I|_K \in V_h|_K$. Moreover, since $u_I\in H^1(\Omega)$, it holds that $u_I\in \Vh$.\\
Using (\ref{definitionuI}), we can write
\[
\left\{\begin{array}{ll}
-\Delta (u_I-\upi) = 0 & \text{in } K\\
u_I-\upi=\phihp-\upi & \text{on } \partial K\\
\end{array} \right. .
\]
\noindent Therefore, since $(u_I - u_{\pi})$ is harmonic it holds
\begin{equation}    \label{stimadainternoabordo}
|u_I-\upi|_{1,K} = \inf \left\{  |z|_{1,K},\; z\in H^1(K) \mid z= \phihp-\upi \text{ on } \partial K     \right\} \le | \phihp - \upi |_{1,K}.
\end{equation}
\noindent Finally by (\ref{stimadainternoabordo}) we obtain
\begin{equation}\label{splittingchileans}
\begin{split}
|u-u_I|_{1,K} &\le |u-\upi|_{1,K}+|\upi-u_I|_{1,K}
\le |u-\upi|_{1,K}+|\upi-\phihp|_{1,K} \\
&\le 2|u-\upi|_{1,K} + |u-\phihp|_{1,K}.
\end{split}
\end{equation}
The proof is completed by summing on all the elements in \eqref{splittingchileans} and using (\ref{stimaupi}), (\ref{stimaBabuSurimeshtriangolare}).
\end{proof}
\begin{remark} \label{remarkdanormaaseminorma2}
We point out that if $k \le p$ and under the hypothesis of Lemma \ref{lemmastimahpsudominio}, the following holds:
\[
|u-u_I|_{1,\Omega} \le \const \frac{h^{k}}{p^{k}} |u|_{k+1,\Omega},
\]
\noindent $\const$ depending only on $k$, $\gamma$ and $\widetilde{\gamma}$ (introduced in (\textbf{D1}) and in (\textbf{D2})).
\end{remark}

% ------------------------------------------------------------------------------------
\subsubsection{Loading approximation term} \label{subsubsectionthirdterm}
% -----------------------------------------------------------------------------------
%%
It remains to estimate the term $\bigfh$ in (\ref{stimevecchieVEMvolley}). We have the following result.

\begin{lem} \label{lemmastima-loading}
Under assumptions (\textbf{D1}) and (\textbf{D2}), let the loading term $f \in H^{\widetilde{k}+1}(K)$ for all $K \in \tauh$, ${\tilde{k}} \in {\mathbb N}$. Then it holds
\begin{equation} \label{stima-loading}
\bigfh \le \const \frac{h^{\widetilde \mu}}{p ^{\widetilde k + 2}} \left( \sum_{K\in \tauh } ||f||^2_{\widetilde k, K}   \right)^{\frac{1}{2}},
\quad \widetilde \mu = \min (p , \widetilde k +2).
\end{equation}
where $\const$ is a constant independent on $h$, $p$ and $u$.
\end{lem}
\begin {proof}
Since the case $p=1$
has been already analysed in \cite{VEMvolley}, we only consider the case $p\ge 2$.
Let $v_h\in\Vh$. Let $P ^{0,K}_{p-2}$ be the $L^2$-projector on polynomials of degree $p-2$ over the polygon $K$, for all $K \in \tauh$. We get by \eqref{discreteloadingterm}
\[
\begin{split}
(f,v_h)_{0,\Omega} - \langle \fh, v_h \rangle _h&
= \sum_{K \in \tauh} (f-P_{p-2}^{0,K}f, v_h)_{0,K} = \sum_{K \in \tauh} (f-P_{p-2}^{0,K}f, v_h- P _{p-2}^{0,K}v_h)_{0,K} \\
            & \le \sum_{K \in \tauh} || f-P _{p-2}^{0,K}f  ||_{0,K} || v_h-P_{p-2}^{0,K}v_h ||_{0,K} \\
            &\le \sum_{K \in \tauh} || f- f^{\pi}_{p-2}|_K ||_{0,K} || v_h-v_{p-2}^{\pi}|_K ||_{0,K} ,
\end{split}
\]
where $f^{\pi}_{p-2}|_K$ and $v^{\pi}_{p-2}|_K$ are the piecewise polynomial functions of degree $p-2$ that realize the bound (\ref{stimeqkonK}) with $\el=0$ on each $K\in \tauh$.
An easy adaptation of Lemma \ref{lemmapestimatesonboundeddomain} (and so also of Lemma 4.1 in \cite{babuskasurihpversionFEMwithquasiuniformmesh} or Lemma 3.1 in \cite{BabuSurioptimalconvergenceestimatepmethods})
implies that, given $\widetilde p = \max (1,p-2)$,
\[
\begin{split}
(f,v_h)_{0,\Omega} - \langle \fh, v_h \rangle & \le \const \sum_{K \in \tauh} \frac{h_K^{\min ((p-2)+1, \widetilde k +1)}}{\widetilde p ^{\widetilde k + 1}} ||f||_{\widetilde k+1, K} \frac{h_K}{\widetilde p} |v_h|_{1,K} \\
            & \le \const \frac{h^{\min (p, \widetilde k +2 )}}{\widetilde p ^{\widetilde k + 2}}   \left( \sum_{K \in \tauh} ||f||^2_{\widetilde k+1, K}  \right)^{\frac{1}{2}} |v_h|_{1,K}.
\end{split}
\]
The final result follows by the definition of $\bigfh$ in \eqref{definitionbigfhp} and substituting ${\widetilde p}$ with $p$, up to a change of the constant $\const$.
\end{proof}
By observing that, if the solution $u$ of \eqref{continuousproblemweakformulation} is in $H^{k+1}(\Omega)$ then $f \in H^{k-1}(\Omega)$, Lemma \ref{lemmastima-loading} immediately gives also the following corollary.
\begin{cor} \label{corol-loading}
Under assumptions (\textbf{D1}) and (\textbf{D2}), let the solution $u$ of \eqref{continuousproblemweakformulation} be in $H^{k+1}(\Omega)$, $k\in \mathbb N$.
Then it holds
\begin{equation} \label{corol-loading-eq}
\bigfh \le \const(k, \gamma, \widetilde \gamma) \frac{h^{\mu}}{p ^{k}} || u ||_{k+1, \Omega} ,
\quad \mu = \min (p , k).
\end{equation}
\end{cor}

Finally, we note that an analogous observation as in Remark \ref{remarkdanormaaseminorma} and Remark \ref{remarkdanormaaseminorma2} holds also for Corollary \ref{corol-loading}, yielding
\begin{equation}
\bigfh \le \const (k, \gamma, \widetilde \gamma) \frac{h^{k}}{ p ^{k}}  |u|_{k+1,\Omega} \: , \quad   1\le k+1 \le p+1 .
\end{equation}

\begin{remark} \label{remarkenanchementloadingterm}
We stress the fact that using the same enhancing strategy introduced in \cite{equivalentprojectorsforVEM} it is possible to obtain a more accurate load approximation. Nevertheless, the global order of convergence of the method does not change due to the presence of the other terms in the error estimates.
\end{remark}

% ---------------------------------------------------------------------------------------------------------------------
\subsection{$hp$ estimate in the energy norm} \label{subsectionhpestimate}
% ---------------------------------------------------------------------------------------------------------------------
Finally, we are able to show the following convergence result.
\begin{thm} \label{theoremstimehpfinali}
Let $k\in \mathbb N$, $k>\frac{1}{2}$ and let the mesh assumptions (\textbf{D1}) and (\textbf{D2}) hold. Let $u$ and $u_h$ be respectively the solution of problems (\ref{continuousproblemweakformulation}) and (\ref{discreteproblem}), with $u\in H^{k+1} (\Omega)$.
Then, the following $hp$ estimates hold:
\begin{eqnarray} \label{stimehpfinalia}
&& |u-\uh|_{1,\Omega} \le \const_1 \frac{h^{\mu}}{p^{k}}||u||_{k+1,\Omega}, \quad \mu = \min (p, k), \\
&& |u-\uh|_{1,\Omega} \le \const_2 \frac{h^{k}}{p^{k}} |u|_{k+1,\Omega}, \quad \text{if } k \le p,
\label{stimehpfinalib}
\end{eqnarray}
where $\const_1$ and $\const_2$ are two constants independent on $h$, $p$ and $u$.
\end{thm}
\begin{proof}
It suffices to combine (\ref{stimevecchieVEMvolley}), (\ref{definitionbigfhp}), (\ref{stimaupi}), (\ref{stimaestensionecileni}) and (\ref{corol-loading-eq}).
\end{proof}
\begin{remark} \label{remarkstimehpL2}
Let the domain $\Omega$ be convex. Following the argument shown in \cite{VEMelasticity} (and, if $p=1,2$ suitably changing the definition of the discrete loading term \eqref{discreteloadingterm})
and applying approximation results similar to those shown above, one can also easily derive $L^2$ estimates of the form:
\begin{equation} \label{stimeL2hpfinalia}
||u-\uh||_{0,\Omega} \le \const(k,\gamma, \widetilde \gamma) \frac{h^{\mu + 1}}{p^{k+1}}||u||_{k+1,\Omega}, \quad \mu = \min (p, k),
\end{equation}
with the usual modification for the case $k \le p$ and where $\const$ is a constant independent on $h$ and $p$.
\end{remark}

% ---------------------------------------------------------------------------------------------------------------------
\section{Exponential convergence for analytic solutions} \label{sectionexponentialconvergenceforanalyticfunction}
% ---------------------------------------------------------------------------------------------------------------------
In this section, we derive an exponential convergence result for analytic solutions, under a further regularity assumption on the decomposition.
We recall that we are given a polygonal decomposition $\tauh$ and a triangular auxiliary subdecomposition $\tauhtilde$ (described in Section \ref{subsectionsecondterm}).
Given $K$ polygon in $\tauh$, we define $Q=Q(K)$ as any of the smallest square containing $K$;
besides, given $\widetilde K$ triangle in $\tauhtilde$, we define $\widetilde Q= \widetilde Q(\widetilde  K)$
the parallelogram given by $\overline{\widetilde Q} = \overline{\widetilde K \cup \widetilde K ^*}$, where $\widetilde K^*$ is the reflection of $\widetilde K$ with respect to a midpoint of anyone of its edges.
We point up that there are three possible $\widetilde Q (\widetilde K)$; we fix arbitrarily one of them. Next, we define:
\begin{equation} \label{equationOmegaext}
\Omegaext= \Omegaext (h):= \Omega \cup \left( \bigcup_{K\in \tauh} Q(K) \right)  \cup \left( \bigcup_{\widetilde K \in \tauhtilde} \widetilde Q(\widetilde K) \right).
\end{equation}
We observe that dist$(x,\Omega) \le d(h)$, $\forall x\in \Omegaext$, $d(\cdot)$ being a non-decreasing function in $h$.
Therefore, $\forall h\le \overline h$ one has $d(h)\le d(\overline h)$ and thus $\Omegaext$ is an uniformly bounded domain in terms of $h$, if $h$ is bounded.
We demand for the following regularity assumption on the mesh:
\begin {enumerate}
\item[$(\mathbf{D3})$] there exists $N\in \mathbb N$ independent on $h$ such that there are at most $N$ overlapping squares in the collection $\{Q(K)\}$ and $N$ parallelograms in the collection $\{\widetilde Q (\widetilde K)\}$, i.e.
                                      for all $Q(K)$ in $\{Q(K)\}$ and for all $\widetilde Q (\widetilde K)$ in $\{\widetilde Q (\widetilde K)\}$,
                  given $I_{K'}:=\{Q(K)\mid Q(K) \cap Q(K') \ne \emptyset\}$ and $\widetilde I_{\widetilde K'}:=\{\widetilde Q(\widetilde K)\mid \widetilde Q(\widetilde K) \cap \widetilde Q(\widetilde K') \ne \emptyset\}$,
                  it holds that card($I_{K'}$),\, card($\widetilde I_{\widetilde K'}$)$\le N$, $\forall K\in \tauh$ and $\forall \widetilde K \in \tauhtilde$.
\end {enumerate}
We note that, given $u \in H^{k+1} (\Omega)$, $k\in \mathbb N$, under assumption (\textbf{D3}), the following holds:
\[
\sum _{K\in \tauh} || u ||^2_{k+1,Q(K)} \le N ||u||^2_{k+1,\Omegaext},\quad \sum _{\widetilde K\in \tauhtilde} || u ||^2_{k+1,\widetilde Q (\widetilde K)} \le N ||u||_{k+1,\Omegaext},
\]
with $\Omegaext$ defined in \eqref{equationOmegaext}.
In order to obtain exponential convergence estimates for analytic functions, we must show bounds analogous to \eqref{stimehpfinalia} and \eqref{stimehpfinalib} by expliciting the dependence of the constants $\const_1$ and $\const_2$ on $k$, i.e. on the Sobolev regularity of the solution $u$.
For this reason, we split this section in two parts.
In Section \ref{subsectionhpestimateusinganoverlappingsquaremethod}, we derive estimates of type \eqref{stimehpfinalia} and \eqref{stimehpfinalib} with the dependence on $k$ explicited;
in Section \ref{subsectionexponentialconvergence}, we derive an exponential convergence estimate.
% ---------------------------------------------------------------------------------------------------------------------
\subsection{$hp$ estimate using an overlapping square method} \label{subsectionhpestimateusinganoverlappingsquaremethod}
% ---------------------------------------------------------------------------------------------------------------------
In this Section, we use an overlapping square technique which allows us, under assumption (\textbf{D3}), to explicit the dependence on the Sobolev regularity in the estimate proven in Lemma \ref{lemmapestimatesonboundeddomain}
(consequently also on Lemmata \ref{lemmaphBrambleHilbert} and \ref{lemmastima-loading}) and on Lemma \ref{lemmastimahpsudominio}.
Finally, we restate Theorem \ref{theoremstimehpfinali} on a proper extended domain, under assumption (\textbf{D3}).
We note that the polynomial approximation which allows to have an estimate in $p$ will be different from that discussed in Section \ref{sectionapproximationresult};
such a polynomial, introduced by \Babuska\text{ }and Suri in \cite{babuskasurihpversionFEMwithquasiuniformmesh} and \cite{BabuSurioptimalconvergenceestimatepmethods}, is a Fourier-type approximation.
We decide to use here a different choice, by choosing an approximant of Legendre type whose properties are studied for instance in \cite{SchwabpandhpFEM}.
The reason for this change is discussed in Remark \ref{remarkBabuvsSchwab}.
% ---------------------------------------------------------------------------------------------------------------------
\subsubsection{A first local estimate} \label{subsubsectionafirstlocalestimate}
% ---------------------------------------------------------------------------------------------------------------------
Here, we give an explicit representation of the constant $\const$
in \eqref{estimatelemma41} in terms of $k$, $k$ being the Sobolev
regularity of the target function. We start by showing the
counterpart of Lemma \ref{lemmapestimatesonboundeddomain}. As a
minor note, we point out that the estimate of Lemma
\ref{lemmacounterpartlemma41} does not require explicitely a shape
regularity condition on the polygons, differently from Lemma
\ref{lemmapestimatesonboundeddomain}.
\begin{lem}\label{lemmacounterpartlemma41}
Let $\Qhat$ be the square $[-1,1]^2$.
Let $\widehat{K} \subseteq \widehat Q$ be a polygon with barycenter $\mathbf x_{\widehat K}=\mathbf{0}$. Moreover, assume that $p\ge 2k$, with $k\in \mathbb N$.
Then, there exists a family of projection operators $\{ \widehat{\Pi}_{\widehat Q, p}   \}$, $p=1,2,\dots$ with $\widehat{\Pi}_{\widehat Q, p} : H^{2}(\Qhat) \rightarrow \mathbb P_p (\Qhat)$
such that, for any $\widehat u \in H^{k+1} (\Qhat)$, it holds
\begin{equation} \label{estimatecounterpartlemma41}
| \widehat u - \widehat{\Pi}_{\widehat Q, p} \widehat u |_{1,\widehat K} \le \const 2^k e^k p^{-k} | \widehat u |_{k+1,\Qhat}
\end{equation}
with $\const$ a constant independent on $u$, $k$ and $p$.
\end{lem}
\begin{proof}
Let $\Qhat = [-1,1]^2$. Let $\{V_i\}_{i=1}^4$ be the set of vertices of $\Qhat$. Let $\uhat \in H^{k+1}(\Qhat)$.
Let $\mathbb Q _{p}(\Theta)$ be the set of polynomials of maximum degree $p$ in each variable over a domain $\Theta \in \mathbb R^2$.
As an easy consequence of Lemma 4.67 in \cite{SchwabpandhpFEM}, it is possible to show the existence of $\hatphi _p \in \mathbb Q_p (\Qhat)$ such that:
\begin{equation} \label{interpolationproperty}
\hatphi _p(V_i) = \uhat (V_i),\quad \forall i=1,\dots,4
\end{equation}
and
\begin{equation} \label{stimaSchwablemma467}
|\uhat - \hatphi _p|^2_{1,\Qhat} \le 2 \left \{  \frac{(p-k)!}{(p+k)!} + \frac{1}{p(p+1)} \cdot \frac{(p-k+1)!}{(p+k-1)!}    \right\} |\widehat u|^2_{k+1,\Qhat}.
\end{equation}
Since $p\ge k$, it is easy to show that \eqref{stimaSchwablemma467} leads to the following simpler bound:
\begin{equation} \label{stimeinstileSchwabsuquadratodiriferimento}
|\uhat - \hatphi _p|_{1,\Qhat} \le \const e^{k} p ^{-k} |\uhat|_{k+1,\Qhat}, \text{ with } \const=\sqrt{e}.
\end{equation}
In order to show this, we perform the computations only on the first term in the right hand side of \eqref{stimaSchwablemma467} since the treatment of the other one is analogous.
Using Stirling fomula:
\[
\frac{(p-k)!}{(p+k)!} = \frac{(p-k)^{(p-k)} \cdot e^{-(p-k)} \cdot \sqrt{2\pi(p-k)} \cdot e^{\theta_{p-k}} }{(p+k)^{(p+k)} \cdot e^{-(p+k)} \cdot \sqrt{2\pi(p+k)} \cdot e^{\theta_{p+k}} },\;
\text{with }\frac{1}{12n+1}\le \theta _{n} \le \frac{1}{12n},\, \forall n\in \mathbb N.
\]
Then:
\begin{equation} \label{stimefattoriali}
\frac{(p-k)!}{(p+k)!} \le p^{-2k} \cdot e^{2k} \cdot e^{\theta_{p-k}}\le \const e^{2k} p^{-2k},\;\text{with } \const=e.
\end{equation}
At this point, we observe that $\mathbb Q_p(\Qhat) \subseteq \mathbb P_{2p}(\Qhat)$.
This fact and \eqref{stimeinstileSchwabsuquadratodiriferimento} immediately imply that there exists $\hatphi _p \in \mathbb P_p(\Qhat)$ which interpolates $\widehat u$ at the vertices of $\Qhat$ as in \eqref{interpolationproperty} and which satisfies
\[
|\uhat - \hatphi _p|_{1,\Qhat} \le \const 2^k e^k p^{-k} |\uhat|_ {k+1,\Qhat},
\]
provided that $p\ge 2k$.
We note that, owing to the fact that $\widehat K \subseteq \Qhat$, it holds
\[
|\uhat - \hatphi _p|_{1,\widehat K} \le |\uhat - \hatphi _p|_{1, \Qhat} \le \const 2^k e^k p^{-k} |\widehat u|_{k+1, \Qhat}.
\]
In order to conclude, it suffices to define $\widehat \Pi _{\widehat Q, p} \uhat := \hatphi _p$.
\end{proof}
The counterpart of Lemma \ref{lemmaphBrambleHilbert} follows.
\begin{lem} \label{lemmacounterpartlemma42}
Let $K\in \tauh$. Let $Q=Q(K)$ be the smallest square containing $K$ and let $u\in H^{k+1}(Q)$. Let $p\ge 2k$.
Then, there exists a sequence of projection operators $\Pi_{Q,p}^h$, $p=1,2,\dots$ with $\Pi_{Q,p}^h: H^{2} (Q) \longrightarrow \mathbb P_p(Q)$ such that for any $k \in \mathbb N$
\[
|u - \Pi _{Q,p}^h u|_{1,K} \le \const M^k \frac{h_K^{\mu}}{p^{k}} || u ||_{k+1,Q},\quad \mu = \min (p,k),
\]
where $\const$ and $M$ are two constants independent on $k$, $h$, $p$ and $u$.
\end {lem}
\begin{proof}
It suffices to apply Lemma \ref{lemmacounterpartlemma41} and a classical scaling argument.
The mapping $F$ between $Q$ and $\Qhat$ is the composition of a rototraslation and a dilatation in $\mathbb R ^2$.
The polygon $\widehat K \in \Qhat$ and the operator $\Pi ^h _{Q,p} u$ will be simply given by $\widehat  K = F(K)$ and $\Pi ^h _{Q,p} u = (\Pi ^h _{Q,p} (u\circ F^{-1}))\circ F$ respectively.
\end{proof}
As done in Section \ref{subsubsectionfirstterm}, we define $\upi \in S_h^{p, -1}(\tauh)$, $S_h^{p, -1}(\tauh)$ being introduced at the beginning of Section \ref{subsectionsomeapproximationresults}, as
\[
\upi|_K= (\Pi _{Q,p}^h u)|_K,\,\text{with }Q=Q(K),\; \forall K\in \tauh.
\]
Owing to assumption (\textbf{D3}) and Lemma \ref{lemmacounterpartlemma42}, we are able to give the following global estimate:
\begin{equation} \label{stimaglobaledopolemma52}
|u-u_{\pi}|_{h,1,\Omega} \le \const A^k\frac{h^{\mu}}{p^k} || u ||_{k+1,\Omegaext},\; \mu =\min (p,k),
\end{equation}
where $\Omegaext$ is defined in \eqref{equationOmegaext} and $\const$ and $A$ are two constants independent on $h$, $p$, $k$, $\gamma$, $\widetilde \gamma$ and $u$ ($A$ is independent also on $N$).

% ---------------------------------------------------------------------------------------------------------------------
\subsubsection {A second local estimate} \label{subsubsectionasecondlocalestimate}
% ---------------------------------------------------------------------------------------------------------------------
In the present section, we give an explicit representation of the constant $\const$ in \eqref{stimaestensionecileni} in terms of $k$.
We point out that here the shape regularity assumption is needed; in fact, the usual scaling arguments used herein are based on affine mappings of shape regular triangles into the master triangle.
\begin{lem}     \label{lemmacounterpartlemma43}
Under assumptions (\textbf{D1}), (\textbf{D2}) and (\textbf{D3}), provided that $p\ge 2k$,
there exists $u_I \in V_h$ such that
\begin{equation} \label{stimalemma53}
|u-u_I|_{1,\Omega} \le \const \cdot B ^k \frac{h^{k}}{p^{k}} |u|_{k+1,\Omegaext},
\end{equation}
where $\const$ and $B$ are two constants independent on $k$, $p$, $h$ and $u$ ($B$ is independent also on $N$).
\end{lem}
\begin{proof}
The proof of this lemma is a combination of the arguments used in Lemma \ref{lemmacounterpartlemma41} and the construction of Lemma \ref{lemmastimahpsudominio}.
Therefore, we only give the sketch of the proof.
We start by considering a triangle $\widetilde K$ in the subtriangular decomposition $\tauhtilde$, we map it into the master triangle $\widehat T$ (i.e. the triangle obtained halving the square $[-1,1]^2$ through its diagonal),
we use a Legendre-type approximant in order to derive a estimate in $p$ as in Lemma \ref{lemmacounterpartlemma41}, we go back to the triangle $\widetilde K$.
Let $\widetilde Q$ be the parallelogram $\widetilde Q = \widetilde Q(\widetilde K)$ (see assumption (\textbf{D3})) and let $\{\widetilde V _ i\}_{i=1}^3$ be the set of the vertices of $\widetilde K$.
Therefore, it is possible to show the existence of a $\phihp \in \mathbb P_p (\widetilde K)$ such that $\phihp (\widetilde V_i)=u(\widetilde V_i)$, $\forall i=1,2,3$ and such that
\begin{equation}\label{stimeinstileSchwabsutriangolodiriferimento}
|u - \phihp|_{1,\widetilde K} \le \const \widetilde B^k \frac{h^{k}}{p^k} |u|_{k+1,\widetilde Q} ,
\end{equation}
where $\const$ and $\widetilde B$ are two constants independent on $p$, $h$, $k$ and $u$ ($\widetilde B$ is also independent on $N$, $\gamma$ and $\widetilde \gamma$).
We point out that this estimate holds for all the triangles in the triangular subdecomposition $\tauhtilde$.
We denote, with a little abuse of notation, by $\phihp : \Omega \rightarrow \mathbb R$ the global piecewise polynomial function whose restriction on each triangle $\widetilde K$ satisfies \eqref{stimeinstileSchwabsutriangolodiriferimento}.

So far, we have obtained discontinuous piecewise polynomials.
We set:
\[
E=E(\widetilde K) := \left( \bigcup _{\left\{ \widetilde{\widetilde K} \in \tauhtilde \mid  \widetilde{K} \cap \widetilde{\widetilde K} =e    \right\}} \widetilde Q \left( \widetilde {\widetilde K} \right) \right) \cup \widetilde Q(\widetilde K),
\quad e\in \mathcal E _{\widetilde K},
\]
where we recall that $\mathcal E _{\widetilde K}$ is the set of the edges of $\widetilde K$ and $\widetilde Q (\widetilde {\widetilde K})$ is defined in assumption (\textbf{D3}).
We need to modify $\phihp$ in order to get a continuous piecewise polynomial over $\tauhtilde$ without changing the approximation property \eqref{stimeinstileSchwabsutriangolodiriferimento}.
This can be done following the same approach as in \cite[Theorem 4.6, Lemma 4.7]{babuskasurihpversionFEMwithquasiuniformmesh},
i.e. by correcting $\phihp$ with suitable polynomial extensions of its edge jumps. It is easy to check that such step does not introduce constants depending on $k$.

With another little abuse of notation, we have obtained a $\phihp \in H_0^1(\Omega)$ piecewise continuous polynomial of degree $p$ over the subtriangular decomposition $\tauhtilde$,
such that an analogous of \eqref{stimeinstileSchwabsutriangolodiriferimento} holds for all $\widetilde K \in \tauhtilde$:
\[
|u - \phihp|_{1,\widetilde K} \le c(\gamma, \widetilde \gamma) \widetilde{\widetilde B}^k \frac{h^{k}}{p^k} |u|_{k+1,E}.
\]
Using assumption (\textbf{D3}) and the arguments described in Lemma \ref{lemmastimahpsudominio}, it is easy to conclude the proof.
\end{proof}
The counterpart of Lemma \ref{lemmastima-loading} follows easily from Lemma \ref{lemmastima-loading} and Lemma \ref{lemmacounterpartlemma42}. In particular the following holds.
\begin{lem} \label{lemmacounterpartlemmaloading}
Under assumptions (\textbf{D1}), (\textbf{D2}) and (\textbf{D3}), let $\Omegaext$ be defined in \eqref{equationOmegaext}, let the loading term $f \in H^{\widetilde{k}+1}(\Omegaext)$. Then it holds
\begin{equation} \label{stima-loading}
\bigfh \le \const D^k \frac{h^{\widetilde \mu}}{p ^{\widetilde k + 2}} ||f||_{\widetilde k+1, \Omegaext},
\quad \widetilde \mu = \min (p , \widetilde k +2),
\end{equation}
where $\const$ and $D$ are two constants independent on $k$, $h$, $p$, $\gamma$, $\widetilde \gamma$ and $u$ ($D$ is also independent on $N$).
\end{lem}
% ---------------------------------------------------------------------------------------------------------------------
\subsubsection{A global estimate result} \label{subsubsectionaglobalestimateresult}
% ---------------------------------------------------------------------------------------------------------------------
Combining bounds \eqref{stimaglobaledopolemma52}, \eqref{stimalemma53}, \eqref{stima-loading} and \eqref{stimevecchieVEMvolley} yields the following result.
\begin{thm} \label{theoremaglobalestimateresult}
Let $k\in \mathbb N$, $k>\frac{1}{2}$. Let the mesh assumptions (\textbf{D1}), (\textbf{D2}) and (\textbf{D3}) hold.
 Let $u$ and $u_h$ be respectively the solution of problems (\ref{continuousproblemweakformulation}) and (\ref{discreteproblem}), with $u\in H^k (\Omega)$.
Let $\Omegaext$ be defined as in \eqref{equationOmegaext}. Let $u\in H^{k+1}(\Omegaext)$.
Let $\gamma$, $\widetilde \gamma$ and $N$ be the constants introduced in assumptions (\textbf{D1}), (\textbf{D2}) and (\textbf{D3}).
Assume also $p\ge 2k$. Then, the following $hp$ estimate holds:
\begin{equation}  \label{stimeglobaliespliciteinkH1}
|u-u_h|_{1,\Omega} \le \const \widetilde A^k \frac{h^{k}}{p^{k}} |u|_{k+1,\Omegaext},
\end{equation}
where $\const$ and $\widetilde A$ are two constants independent on $h$, $p$, $k$ and $u$ ($\widetilde A$ is also independent on $N$).
\end{thm}
As done in Remark \ref{remarkstimehpL2}, we point out that if the domain $\Omega$ is convex
it is possible to derive easily, owing to the approximation properties of Legendre polynomials, $L^2$ estimates of the form:
\begin{equation} \label{stimeglobaliespliciteinkL2}
|u-u_h|_{0,\Omega} \le \const \widetilde A^k \frac{h^{k+1}}{p^{k+1}} |u|_{k+1,\Omegaext},
\end{equation}
where $\const$ and $\widetilde A$ are two constants independent on $h$, $p$, $k$ and $u$ ($\widetilde A$ is also independent on $N$).
\begin{remark} \label{remarkBabuvsSchwab}
We point out that in order to obtain the $hp$ estimates of Theorem
\ref{theoremstimehpfinali} and of Theorem
\ref{theoremaglobalestimateresult}, we used two different
approximant polynomials. Throughout Section
\ref{sectionapproximationresult}, we decided to follow the
\Babuska-Suri construction (see
\cite{babuskasurihpversionFEMwithquasiuniformmesh} and
\cite{BabuSurioptimalconvergenceestimatepmethods}) which is based
on a Fourier series expansion on a proper domain; this choice
is essentially a matter of taste but has the merit of avoiding to
use bi-polynomial functions.
%%%%%
%%%%%%
Nevertheless this construction obliges, also in the case of the
overlapping square technique introduced at the beginning of
Section \ref{sectionexponentialconvergenceforanalyticfunction}, to
use some extension operator  (for instance the one described in
\cite{steinsingularintegrals} for Lipschitz domains). Thus, to
give an explicit representation of the dependence of the involved
constant on the Sobolev regularity $k$ is a not trivial work.
On the other hand, throughout Section \ref{sectionexponentialconvergenceforanalyticfunction}, we made use of Legendre-type approximant (as done for instance in \cite{SchwabpandhpFEM}).
In this case, owing to Legendre polynomials properties, we are able to obtain exponential estimates (see Section \ref{subsectionexponentialconvergence}),
since the dependence in the constant with respect to the Sobolev regularity $k$ can be derived.
\end{remark}
% ---------------------------------------------------------------------------------------------------------------------
\subsection {Exponential convergence} \label{subsectionexponentialconvergence}
% ---------------------------------------------------------------------------------------------------------------------
We have the following exponential convergence result for analytic solutions $u$ over the extended domain $\Omegaext$ (see \eqref{equationOmegaext}).
\begin{thm} \label{theoremeponentialconvergence}
Let the mesh assumptions (\textbf{D1}), (\textbf{D2}) and (\textbf{D3}) hold. Let $u$ and $u_h$ be respectively the solution of problems (\ref{continuousproblemweakformulation}) and (\ref{discreteproblem}),
with $u\in \mathcal A (\overline \Omegaext)$, $\mathcal A (\overline \Omegaext)$ being the set of analytic function over the closure of $\Omegaext$ defined in \eqref{equationOmegaext}.
Then, the following exponential convergence estimate holds:
\begin{eqnarray} \label{stime-exponential}
||u-u_h||_{1,\Omega} \le \const e^{-bp},
\end{eqnarray}
for some positive constants $\const$ and $b$ independent on $p$.
\end{thm}
\begin{proof}
We recall (see for instance \cite{MonvelKreepseudodifferential}) that an analytic function in the closure of a domain $\Theta \in \mathbb R ^2$ is characterized by the following bound:
\begin{equation} \label{caratterizzazionefunzionianalitiche}
|| D^{\boldalpha} u ||_{\infty, \overline \Theta} \le \const A^{|\boldalpha|} \boldalpha !,\quad \boldalpha = (\alpha_1, \alpha_2) \in {\mathbb N_0}^2,
\end{equation}
where $\boldalpha ! = \alpha_1 ! \alpha_2!$ and where $\const$ and $A$ are constants independent on the multi-index $\boldalpha$; nevertheless, $\const$ and $A$ depends on $u$ and on $\overline \Theta$.
Recalling \eqref{stimeglobaliespliciteinkH1}, we have
\[
|u-\uh|_{1,\Omega} \le \const(\gamma, \widetilde \gamma, N) \widetilde A(\gamma, \widetilde \gamma) ^k\frac{h^{k }}{p^{k}} |u|_{k+1,\Omegaext},
\]
if $p\ge 2k$.
Using standard results from space interpolation theory \cite{Bergh-Lofstrom, Triebel}, from the above bound one can easily derive
\begin{equation}\label{stima-s-real}
|u-\uh|_{1,\Omega} \le \const(\gamma, \widetilde \gamma, N) \widetilde A(\gamma, \widetilde \gamma)^{s}\frac{h^{s}}{p^{s}} | u |_{s+1,\Omegaext}
\end{equation}
for all $s \in {\mathbb R}$ with $2\le 2s \le p$.
The combination of \eqref{stima-s-real} and \eqref{caratterizzazionefunzionianalitiche} yields
\[
|u-\uh|_{1,\Omega} \le \const \left( \widetilde A \frac{h}{p} \right)^s A^{s+1} (s+1)! .
\]
By means of Stirling formula, we obtain:
\[
|u-\uh|_{1,\Omega} \le \const \left( \frac{h A \widetilde A}{p} \right)^s  \left( \frac{s+1}{e}  \right)^{s+1} \sqrt{2\pi} (s+1)^{\frac{1}{2}} ,
\]
easily yielding
\[
|u-\uh|_{1,\Omega} \le \const \left( \frac{h A \widetilde A}{e p} s \right)^s  s^{\frac{3}{2}}.
\]
By denoting $\delta=\frac{h A \widetilde A}{e}$ we can write:
\[
|u-\uh|_{1,\Omega} \le \const \left( \frac{s}{p} \delta \right)^s  s^{\frac{3}{2}}.
\]
Since this last inequality holds true for all $s$ such that $2 \le 2s \le p$, we may choose $s=\frac{p}{2(\delta+1)}$. Hence:
\begin{equation} \label{abbiamoquasiconvergenzaesponenziale}
|u-\uh|_{1,\Omega} \le \const \left( \frac{\delta}{2(\delta+1)} \right)^{\frac{p}{2(\delta+1)}}  p^{\frac{3}{2}}=  \const e^ {- b p  }  p^{\frac{3}{2}} \: ,
\ \textrm{ with } b=\frac{\log(\frac{\delta}{2(\delta+1)})}{2(\delta+1)}.
\end{equation}
The multiplier $p^{\frac{3}{2}}$ can be absorbed by $e^{-bp}$ by making $b$ a little bit smaller and increasing $\const$; therefore, \eqref{abbiamoquasiconvergenzaesponenziale} immediately yields
\begin{equation} \label{exponentialconvergence}
|u-u_h|_{1,\Omega} \le \const e^{-bp},
\end{equation}
for some constants $\const$ and $b$ independent on $p$. The result
follows by the Poincar\'e inequality. 
\end{proof}
% -----------------------------------------------------------------
\section{Numerical results} \label{sectionnumericalresults}
% -----------------------------------------------------------------

In this section, we present numerical results experimentally validating the error estimates \eqref{stimehpfinalia}, \eqref{stimehpfinalib}, \eqref{stimeL2hpfinalia} and \eqref{exponentialconvergence}.
We consider four types of meshes (see Figure \ref{figurefourmeshes})
on the domain $\Omega= [0,1]^2$, namely an unstructured triangular mesh, a regular square mesh, a regular hexagonal mesh and a Voronoi-Lloyd mesh (see \cite{dufabergunzburgerVoronoi}).
\begin{figure}  [h]
\centering
\subfigure {\includegraphics [angle=0, width=0.24\textwidth]{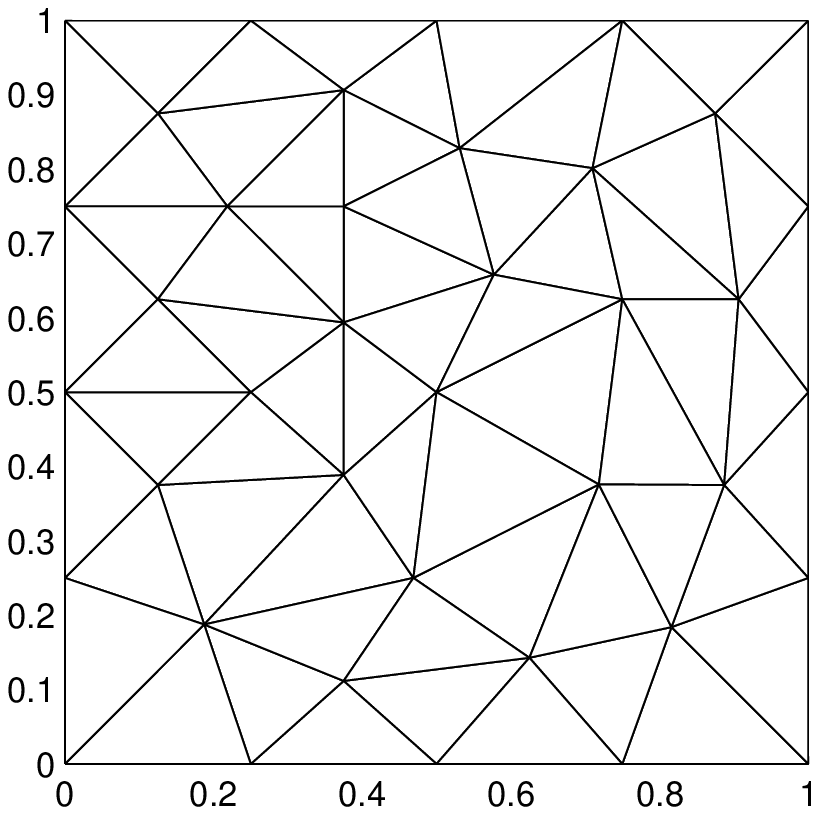}}
\subfigure {\includegraphics [angle=0, width=0.24\textwidth]{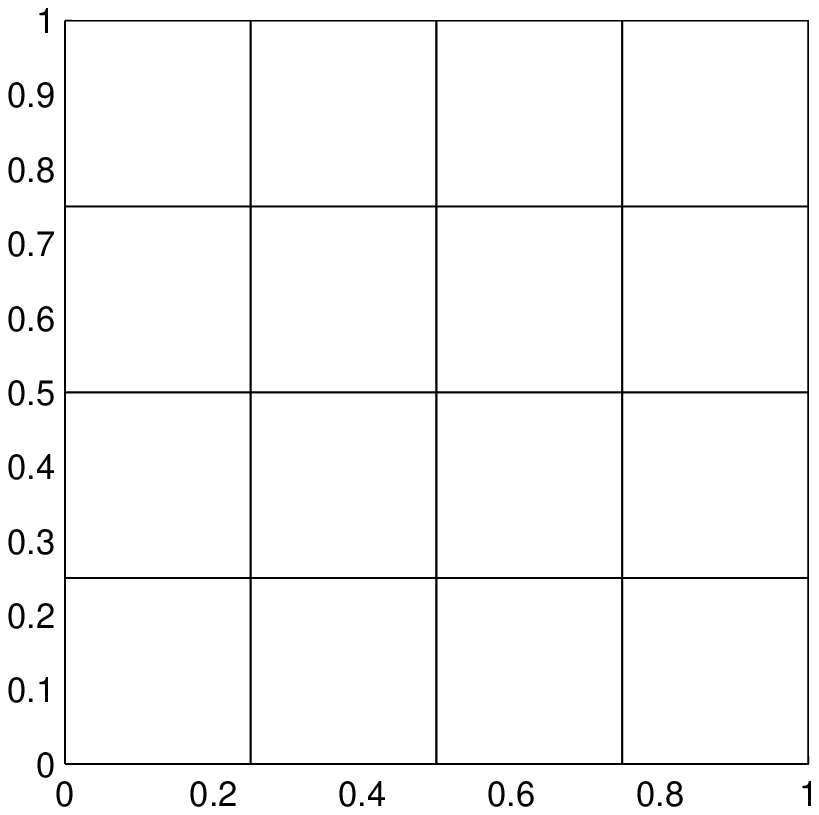}}
\subfigure {\includegraphics [angle=0, width=0.24\textwidth]{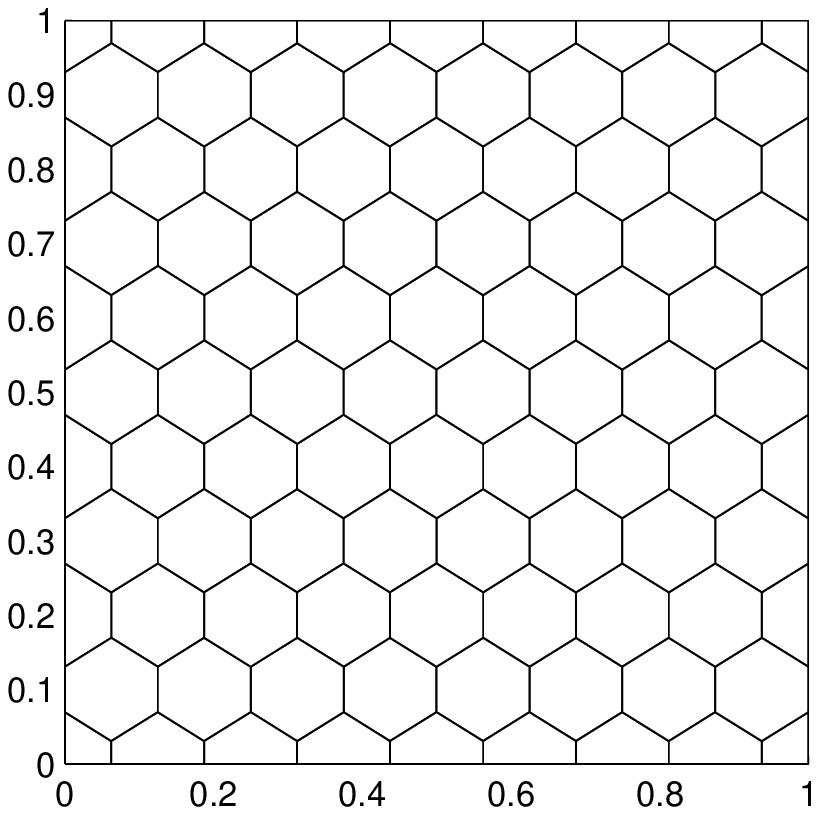}}
\subfigure {\includegraphics [angle=0, width=0.24\textwidth]{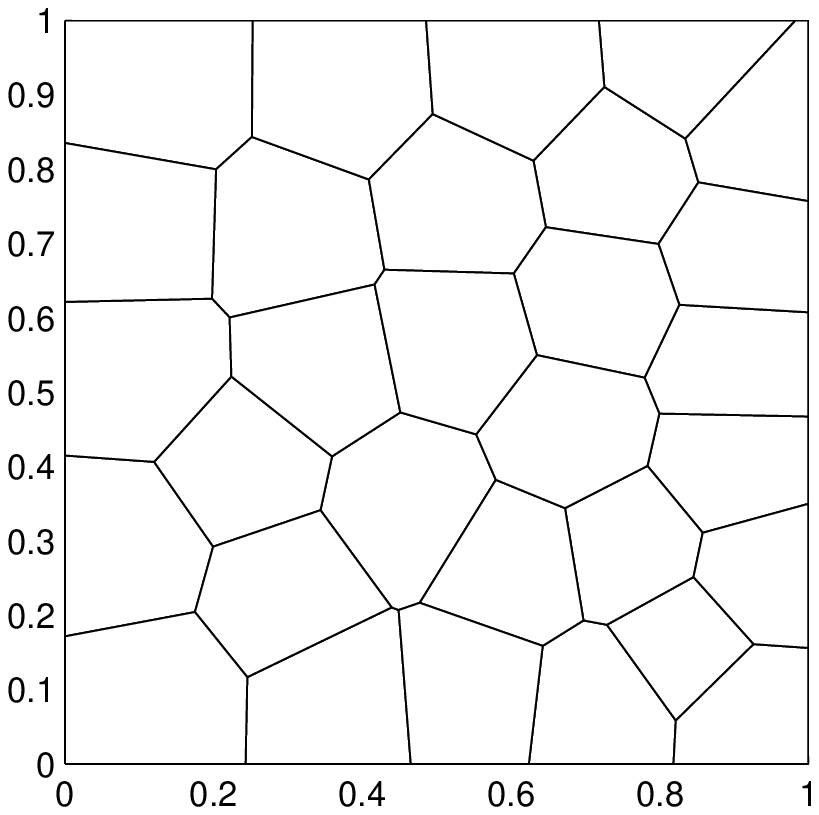}}
\caption{
From left to right: unstructured triangular mesh, regular square mesh, regular hexagonal mesh, Voronoi-Lloyd mesh} \label{figurefourmeshes}
\end{figure}
The basis $\{q_{\boldalpha}\}$ of the space $\mathbb P_{p-2}(K)$ introduced in Section \ref{sectionVEMforthePoissonproblem}, $\forall K\in \tauh$, is taken to be the same as that introduced for instance in \cite{VEMvolley} or \cite{hitchhikersguideVEM}.
Different choices are investigated in the Appendix.
Moreover, we fix a possible choice for the stabilizing term $S^K$ introduced in \eqref{stabilizingterm} (see for instance \cite{VEMvolley}) as:
\begin{equation} \label{choicestabilizingform}
S^K(u_h,v_h) = \sum _{r=1}^{\text{dim}(\VhK)} \chi_r (u_h) \chi_r (v_h) ,\; \forall u_h,\, v_h \in \VhK,\, K\in \tauh,
\end{equation}
where $\chi _r$, $\forall r=1,\dots,\text{dim}(\VhK)$ is the operator which associates to each function in the local space  $\VhK$ its $r$th local degree of freedom.\\
\begin{remark}
The stabilizing bilinear form \eqref{choicestabilizingform} does not guarantee the stability property \eqref{stabilizingterm} uniform in $p$.
Nevertheless, the numerical results seem to be robust with respect to this choice.
A theoretical study of the stabilization will be the object of further investigation.
\end{remark}
In order to estimate the error introduced by the MATLAB algebraic sparse solver, we have solved a problem whose exact solution is the polynomial $u(x,y)=x^2+y^2$.
Since the VEM passes the patch test, in this case, for $k\geq 2$, the approximate solution $u_h$ coincides with $u$ and the error that we measure is only due to the algebraic solver (it should be zero in exact arithmetic).
Hence, together with the standard error in the $H^1$-norm and in the $L^2$-norm with respect to the solutions \eqref{caso test sen} and \eqref{caso test power}, in our convergence figures we also plot this algebraic error.
When the error curve comes close to the algebraic error curve, the convergence error and the error introduced by the MATLAB algebraic solver are of the same order and the expected theoretical behaviour does not hold anymore.
% -----------------------------------------------------------------
\subsection{Convergence in $p$ for an analytic function} \label{subsectionumericaltestsnforaCinftyfunction}
% -----------------------------------------------------------------
We consider problem \eqref{continuousproblemweakformulation} with loading term $f(x,y)=2\pi^2 \sin(\pi x)\sin(\pi y)$. The exact solution is given by:
\begin{equation} \label{caso test sen}
u(x,y)=\sin(\pi x)\sin(\pi y).
\end{equation}
In this tests, the mesh is kept fixed (see Figure \ref{figurefourmeshes}) and the polynomial degree is raised.
In Figures \ref{figureerrorsfunzioneCinftypatchtest_stab1_a} and \ref{figureerrorsfunzioneCinftypatchtest_stab1_b} we report the errors among the discrete and exact solutions.
Since we are dealing with a virtual element solution $u_h$ (that is unknown inside elements) we cannot direcly compute the error $|| u - u_h ||{s,\Omega}$, $s=0,1$. Therefore, as is standard in VEM, we plot instead $|| u-\Pinablap u_h ||_{0,\Omega}$ and $| u-\Pinablap u_h |_{h,1,\Omega}$, that are good representatives of the above errors (see \eqref{definitionPinabla} for the definition of the operator $\Pinablap$ and \eqref{H1brokenSobolevseminorm} for the definition of the $H^1$ broken Sobolev seminorm).
\begin{figure}  [h] \label{figureerrorsfunzioneCinftypatchtest_stab1_a}
\begin{overpic}[angle=0, width=0.45\textwidth]{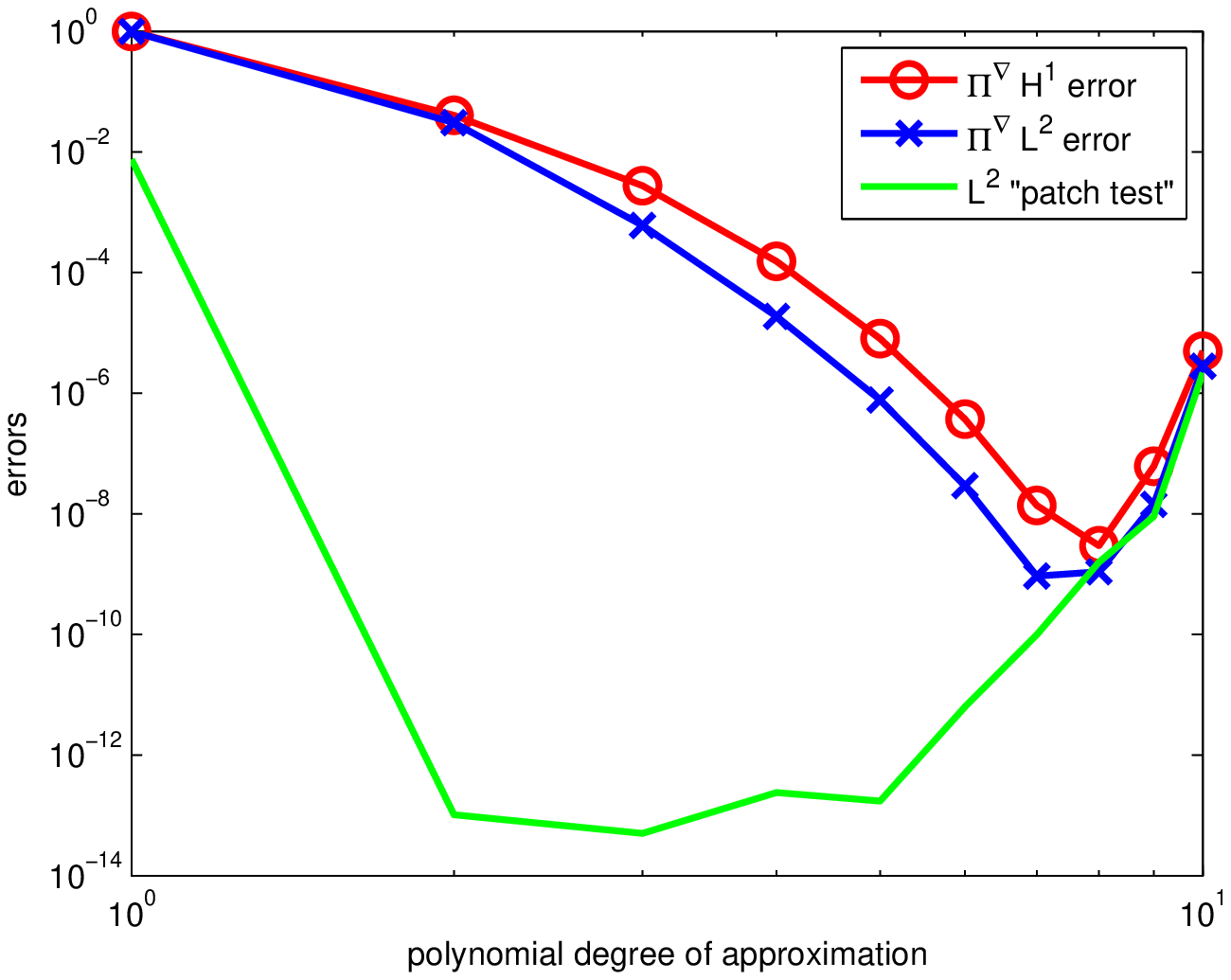}
\put(20,9.5){\colorbox{white}{\tiny 1}} \put(68,9.5){\tiny{2}} \put(94.5,9.5){\tiny{3}} \put(113,9.5){\tiny{4}} \put(127.5,9.5){\tiny{5}}
\put(139.5,9.5){\tiny{6}} \put(149,9.5){\tiny{7}} \put(158,9.5){\tiny{8}} \put(166,9.5){\tiny{9}} \put(170,9.5){\colorbox{white}{\tiny 10}}
\end{overpic}
\begin{overpic}[angle=0, width=0.45\textwidth]{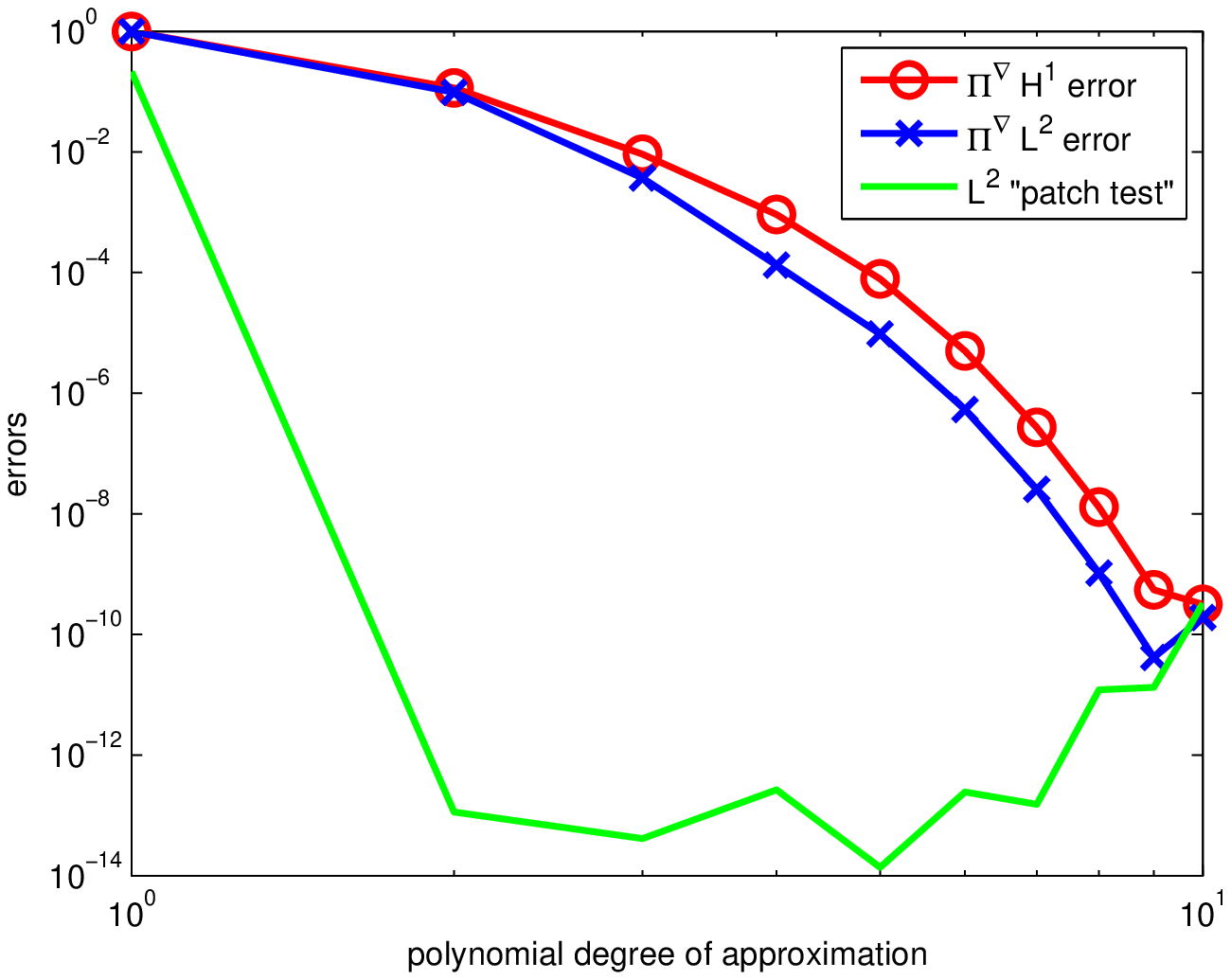}
\put(20,9.5){\colorbox{white}{\tiny 1}} \put(68,9.5){\tiny{2}} \put(94.5,9.5){\tiny{3}} \put(113,9.5){\tiny{4}} \put(127.5,9.5){\tiny{5}}
\put(139.5,9.5){\tiny{6}} \put(149,9.5){\tiny{7}} \put(158,9.5){\tiny{8}} \put(166,9.5){\tiny{9}} \put(170,9.5){\colorbox{white}{\tiny 10}}
\end{overpic}
\caption{$u(x,y)=\sin(\pi x)\sin(\pi y)$; unstructured triangle mesh (left); regular square mesh (right)} \label{figureerrorsfunzioneCinftypatchtest_stab1_a}
\end{figure}
\begin{figure}  [h]
\begin{overpic}[angle=0, width=0.45\textwidth]{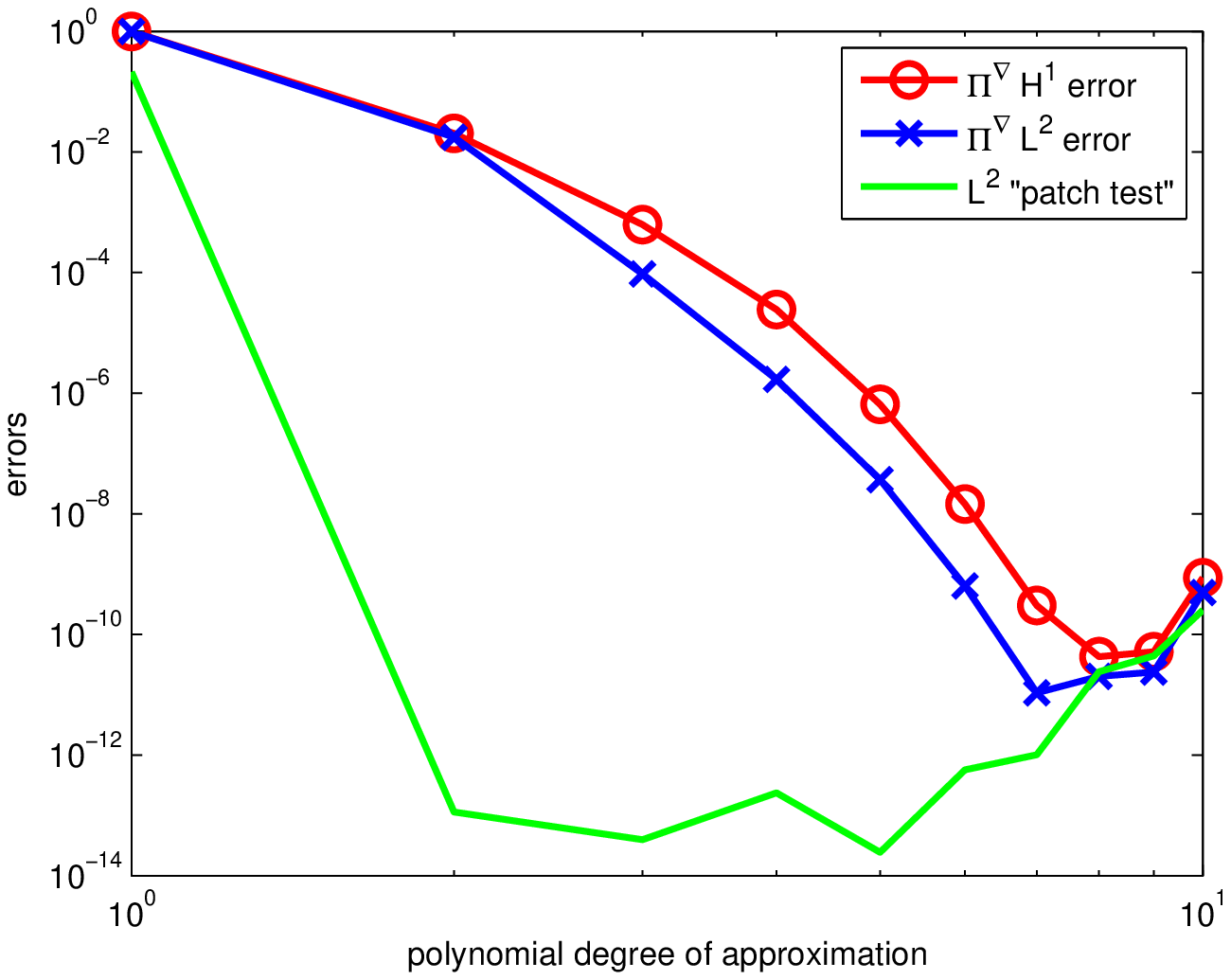}
\put(20,9.5){\colorbox{white}{\tiny 1}} \put(68,9.5){\tiny{2}} \put(94.5,9.5){\tiny{3}} \put(113,9.5){\tiny{4}} \put(127.5,9.5){\tiny{5}}
\put(139.5,9.5){\tiny{6}} \put(149,9.5){\tiny{7}} \put(158,9.5){\tiny{8}} \put(166,9.5){\tiny{9}} \put(170,9.5){\colorbox{white}{\tiny 10}}
\end{overpic}
\begin{overpic}[angle=0, width=0.45\textwidth]{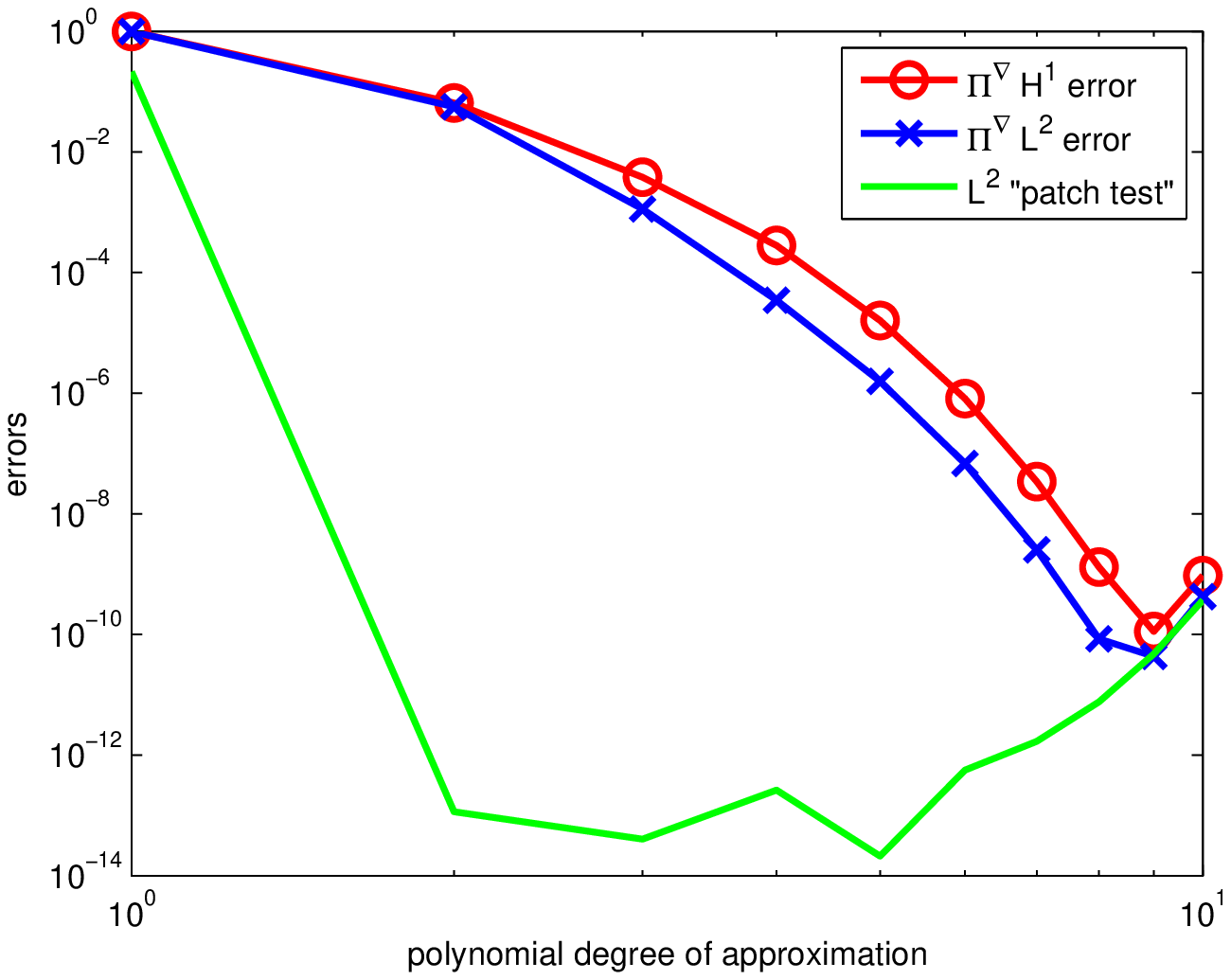}
\put(20,9.5){\colorbox{white}{\tiny 1}} \put(68,9.5){\tiny{2}} \put(94.5,9.5){\tiny{3}} \put(113,9.5){\tiny{4}} \put(127.5,9.5){\tiny{5}}
\put(139.5,9.5){\tiny{6}} \put(149,9.5){\tiny{7}} \put(158,9.5){\tiny{8}} \put(166,9.5){\tiny{9}} \put(170,9.5){\colorbox{white}{\tiny 10}}
\end{overpic}
\caption{$u(x,y)=\sin(\pi x)\sin(\pi y)$; regular hexagonal mesh (left); 4: Voronoi-Lloyd mesh (right)} \label{figureerrorsfunzioneCinftypatchtest_stab1_b}
\end{figure}

In accordance with Theorem \ref{theoremeponentialconvergence}, the exponential
convergence is evident from the decreasing slope in the error
graphs. Moreover, from Figures
\ref{figureerrorsfunzioneCinftypatchtest_stab1_a} and
\ref{figureerrorsfunzioneCinftypatchtest_stab1_b} we can observe
that the lower line is a good marker for the indication of the
machine algebra error.
% -----------------------------------------------------------------
\subsection{Convergence in $p$ for a function with finite Sobolev regularity} \label{subsectionnumericaltestsforafunctionwithfixedregularity}
% -----------------------------------------------------------------
Secondly, we present a similar behaviour test for the case of a problem with solution
\begin{equation} \label{caso test power}
u(r,\theta)=r^{2.5}\sin(2.5 \theta),
\end{equation}
where $(r,\theta)$ are the polar coordinates with respect to the origin.
Since the function is harmonic, the loading term $f=0$ and the Dirichlet boundary conditions are set in accordance with $u|_{\partial \Omega}$.
We note that $u\in H^{3.5- \varepsilon}(\Omega)$, $\forall \varepsilon >0$.
In Figures \ref{figureerrorsfunzioneSobolevpatchtest_stab1_a} and \ref{figureerrorsfunzioneSobolevpatchtest_stab1_b}, the segmented line represents a line of slope $5=2\cdot 2.5$.
Owing to Theorem \ref{theoremstimehpfinali}, we should have an estimate in $p$ of the type $p^{-a}$, $a=2.5$.
Anyhow, for this type of corner singularity, one could extend the technical result of \cite{babuskasurihpversionFEMwithquasiuniformmesh} and \cite{BabuSurioptimalconvergenceestimatepmethods}
obtaining error estimate in $p$ of the type $p^{-2a}$ also in our VEM framework.
Thus, we expect a slope for the $H^1$ error of the type $p^{-5}$, which is represented with the dashed line in Figures \ref{figureerrorsfunzioneSobolevpatchtest_stab1_a} and \ref{figureerrorsfunzioneSobolevpatchtest_stab1_b}.
\begin{figure}  [h]
\begin{overpic}[angle=0, width=0.45\textwidth]{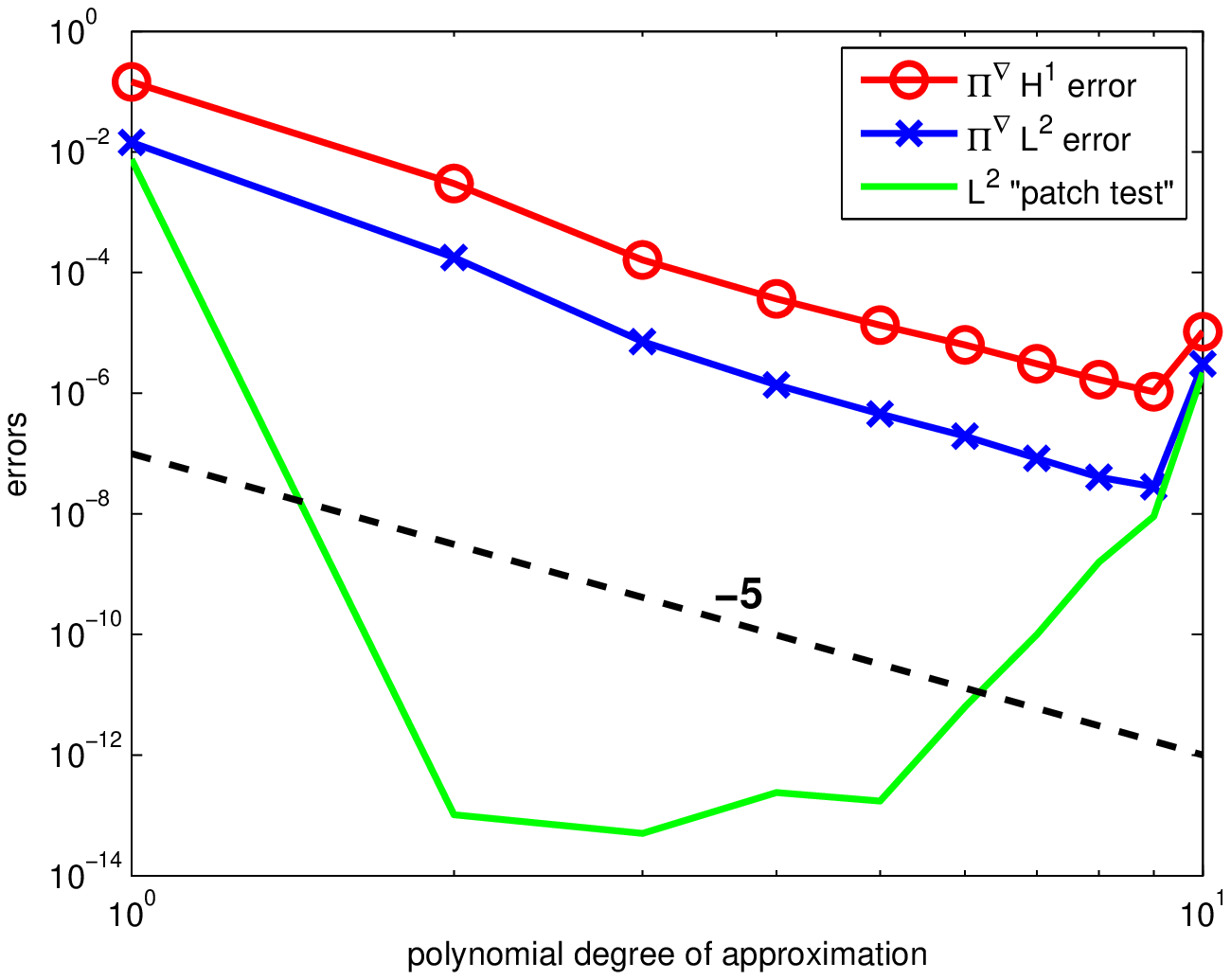}
\put(20,9.5){\colorbox{white}{\tiny 1}} \put(68,9.5){\tiny{2}} \put(94.5,9.5){\tiny{3}} \put(113,9.5){\tiny{4}} \put(127.5,9.5){\tiny{5}}
\put(139.5,9.5){\tiny{6}} \put(149,9.5){\tiny{7}} \put(158,9.5){\tiny{8}} \put(166,9.5){\tiny{9}} \put(170,9.5){\colorbox{white}{\tiny 10}}
\end{overpic}
\begin{overpic}[angle=0, width=0.45\textwidth]{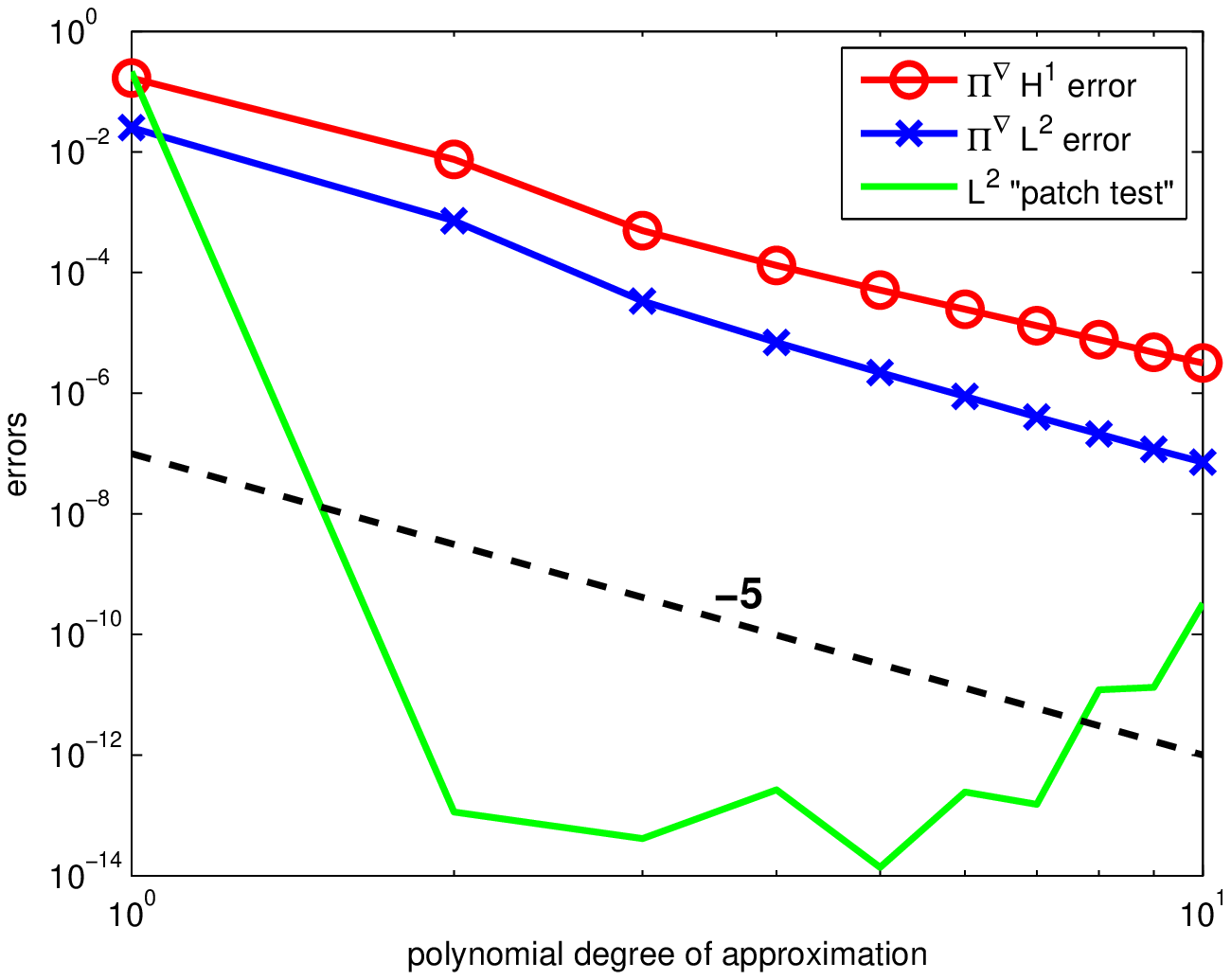}
\put(20,9.5){\colorbox{white}{\tiny 1}} \put(68,9.5){\tiny{2}} \put(94.5,9.5){\tiny{3}} \put(113,9.5){\tiny{4}} \put(127.5,9.5){\tiny{5}}
\put(139.5,9.5){\tiny{6}} \put(149,9.5){\tiny{7}} \put(158,9.5){\tiny{8}} \put(166,9.5){\tiny{9}} \put(170,9.5){\colorbox{white}{\tiny 10}}
\end{overpic}
\caption{$u=r^{2.5}\sin(2.5 \theta)$; unstructured triangle mesh (left); 2: regular square mesh (right)} \label{figureerrorsfunzioneSobolevpatchtest_stab1_a}
\end{figure}
\begin{figure} [h]
\begin{overpic}[angle=0, width=0.45\textwidth]{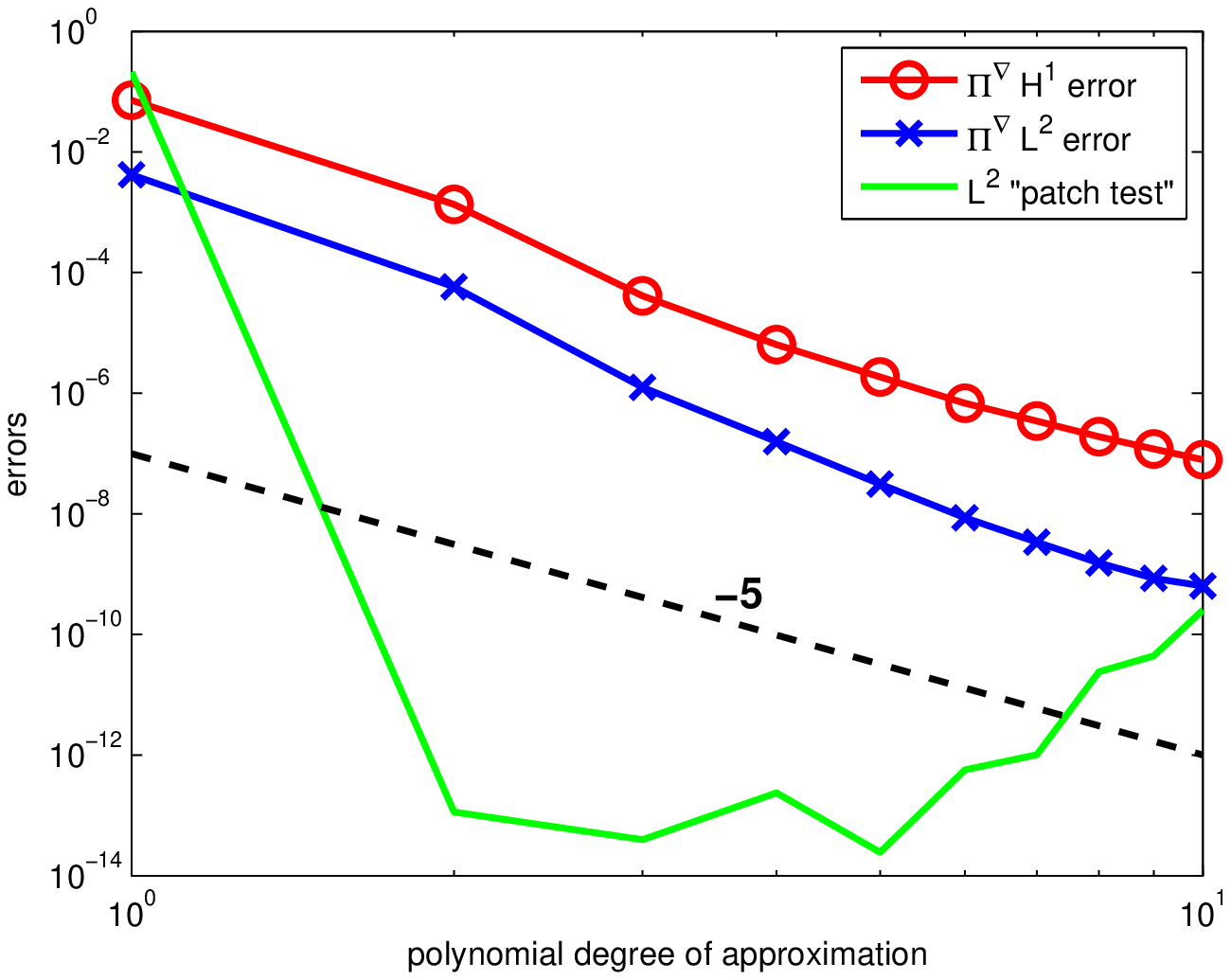}
\put(20,9.5){\colorbox{white}{\tiny 1}} \put(68,9.5){\tiny{2}} \put(94.5,9.5){\tiny{3}} \put(113,9.5){\tiny{4}} \put(127.5,9.5){\tiny{5}}
\put(139.5,9.5){\tiny{6}} \put(149,9.5){\tiny{7}} \put(158,9.5){\tiny{8}} \put(166,9.5){\tiny{9}} \put(170,9.5){\colorbox{white}{\tiny 10}}
\end{overpic}
\begin{overpic}[angle=0, width=0.45\textwidth]{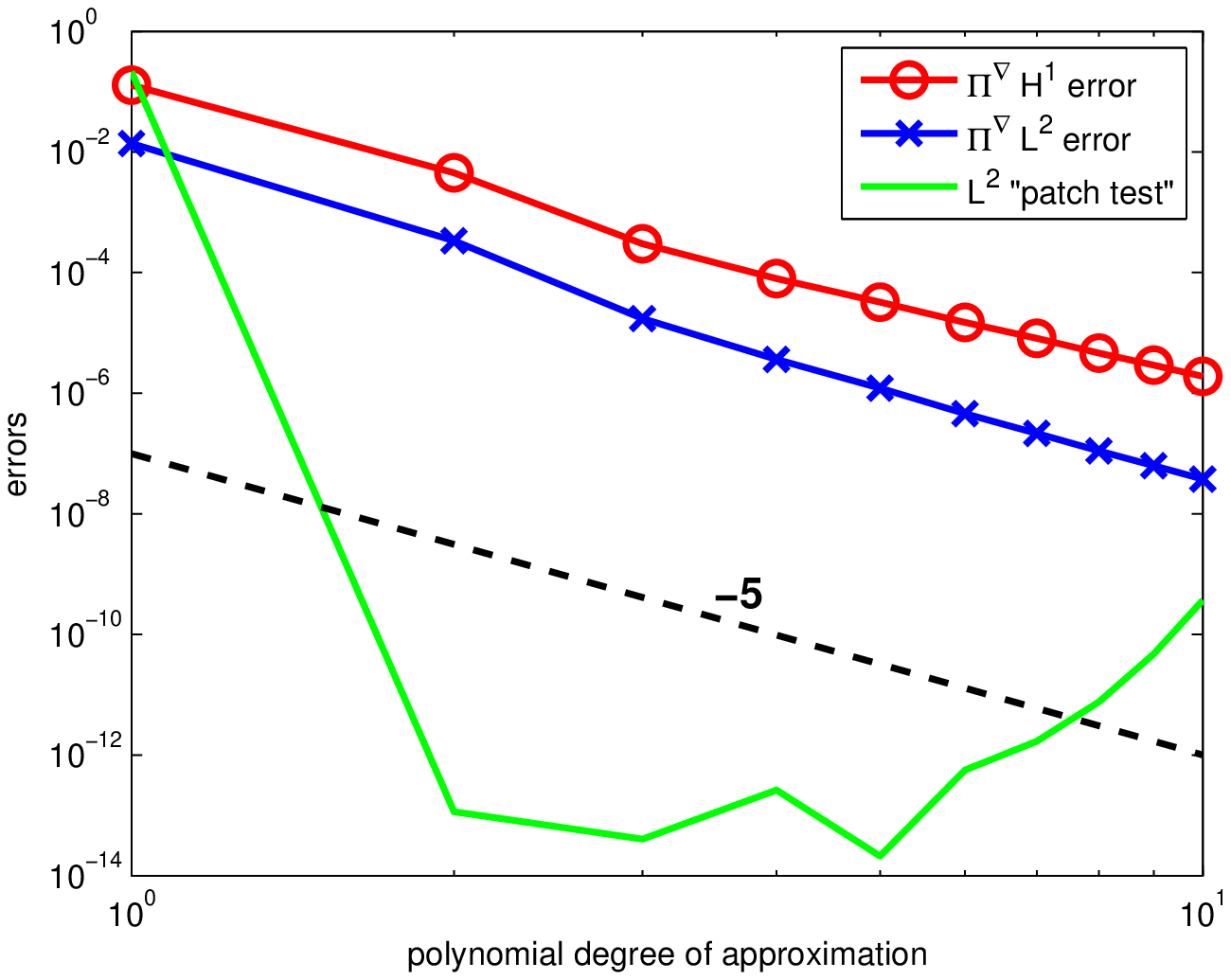}
\put(20,9.5){\colorbox{white}{\tiny 1}} \put(68,9.5){\tiny{2}} \put(94.5,9.5){\tiny{3}} \put(113,9.5){\tiny{4}} \put(127.5,9.5){\tiny{5}}
\put(139.5,9.5){\tiny{6}} \put(149,9.5){\tiny{7}} \put(158,9.5){\tiny{8}} \put(166,9.5){\tiny{9}} \put(170,9.5){\colorbox{white}{\tiny 10}}
\end{overpic}
\caption{$u=r^{2.5}\sin(2.5 \theta)$; regular hexagonal mesh (left); 4: Voronoi-Lloyd mesh (right)} \label{figureerrorsfunzioneSobolevpatchtest_stab1_b}
\end{figure}
Figures \ref{figureerrorsfunzioneSobolevpatchtest_stab1_a} and \ref{figureerrorsfunzioneSobolevpatchtest_stab1_b}, are in agreement with such observation.

% -----------------------------------------------------------------
\subsection{Convergence in $h$} \label{subsectionhconvergence}
% -----------------------------------------------------------------
In this subsection, we show the convergence rate when the polynoimial degree is kept fixed and the meshsize goes to zero.
We consider a sequence of hexagonal and Voronoi-Lloyd meshes and we study the same harmonic test case as in Section \ref{subsectionumericaltestsnforaCinftyfunction}. In particular, we examine the case $p=3$ and $p=5$.
\begin{figure}  [h]
\centering
\subfigure {\includegraphics [angle=0, width=0.45\textwidth]{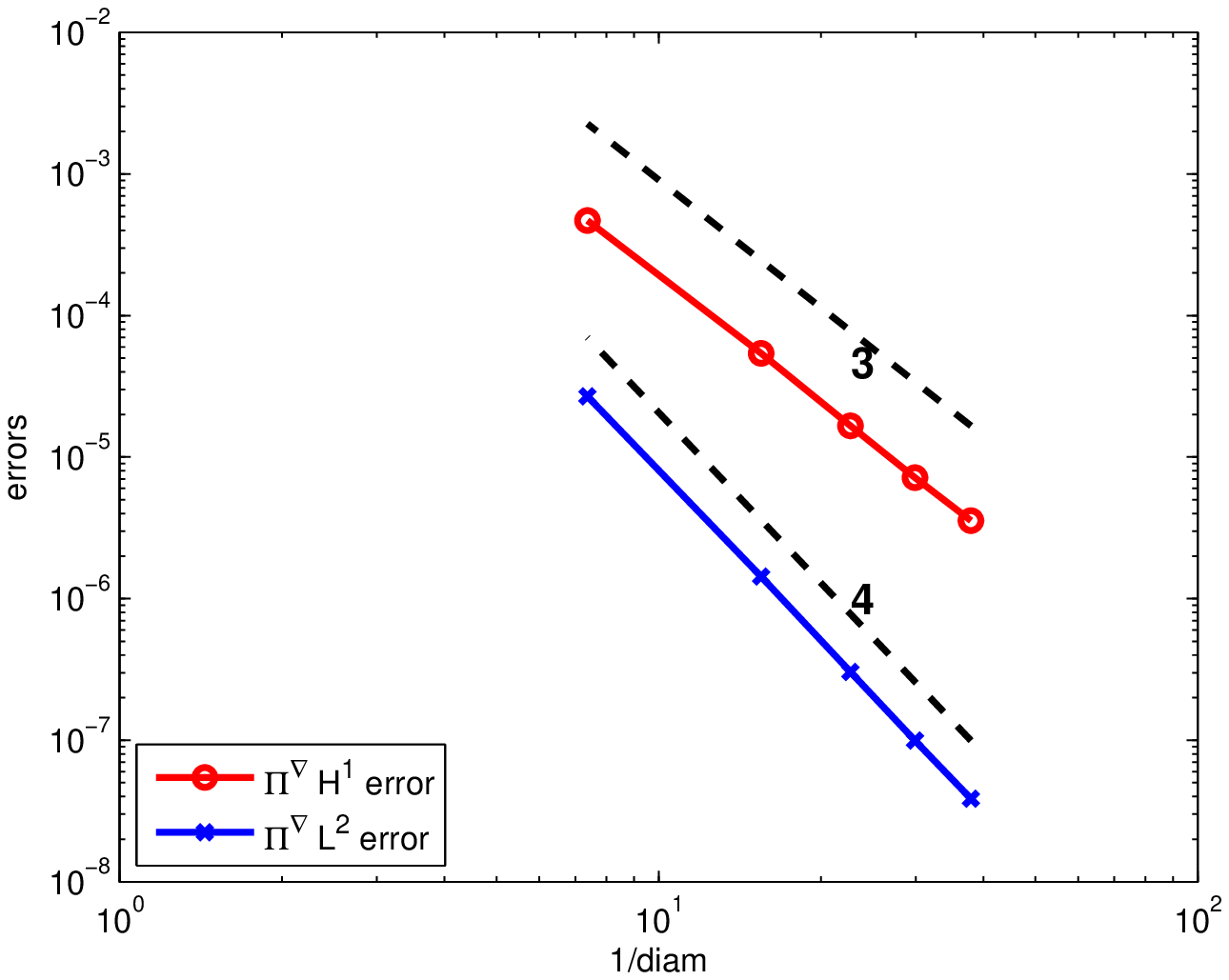}}
\subfigure {\includegraphics [angle=0, width=0.45\textwidth]{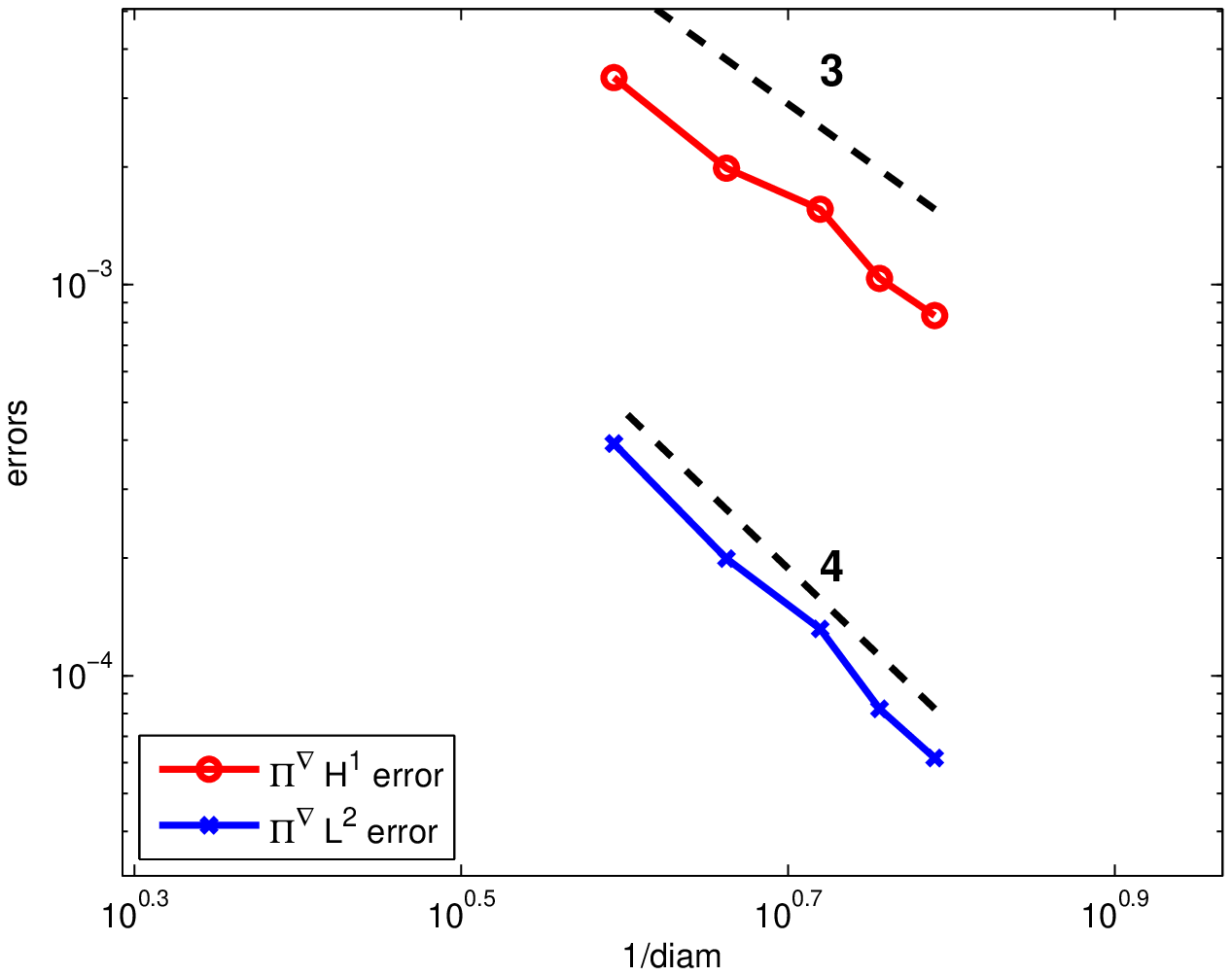}}
\caption{regular hexagonal mesh (left); Voronoi-Lloyd mesh (right); p=3} \label{h_errors_p3}
\end{figure}
We observe that the slope of the errors are in accordance with  Theorem \ref{theoremstimehpfinali} and with estimate \eqref{stimeL2hpfinalia}.
\begin{figure}  [h]
\centering
\subfigure {\includegraphics [angle=0, width=0.45\textwidth]{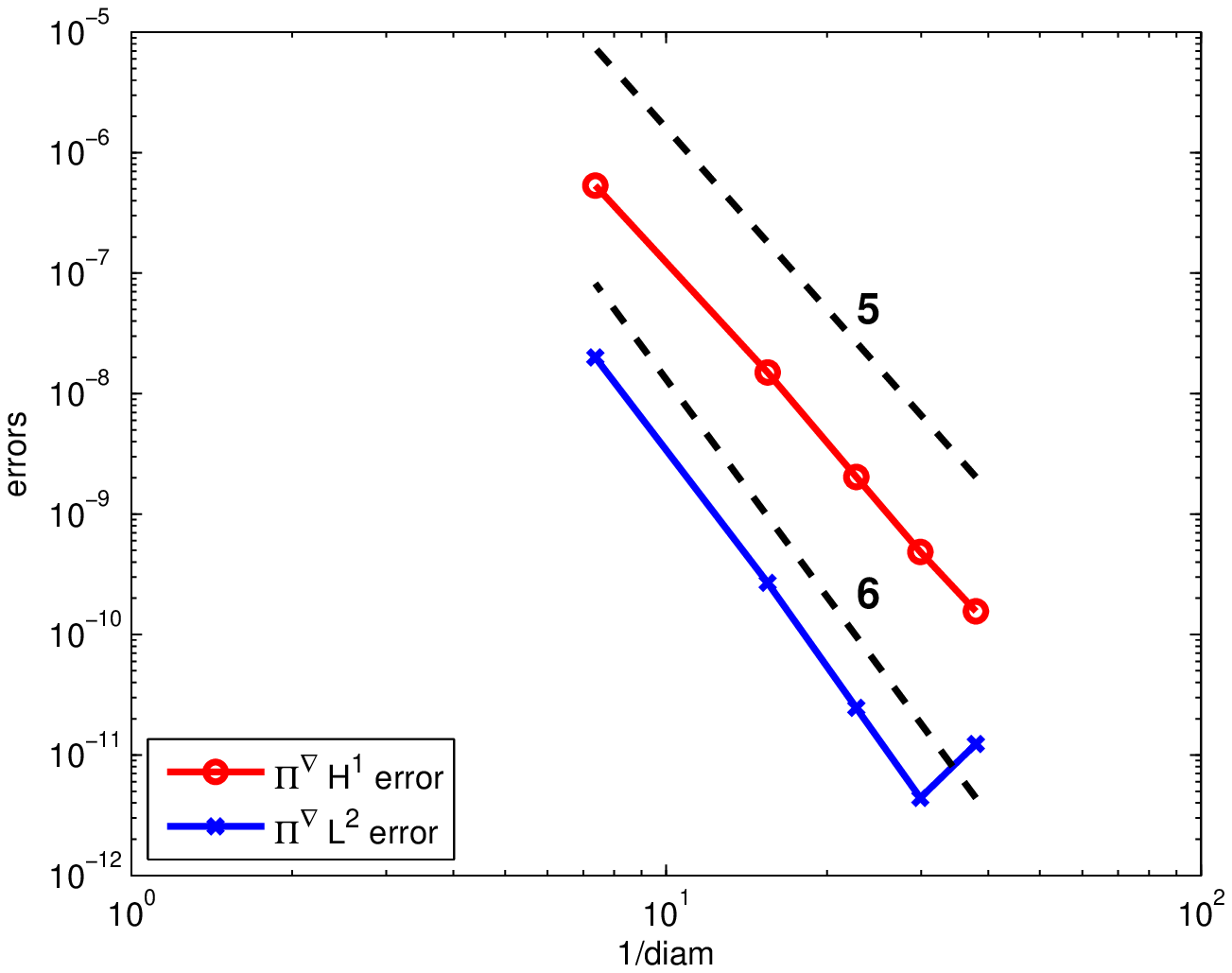}}
\subfigure {\includegraphics [angle=0, width=0.45\textwidth]{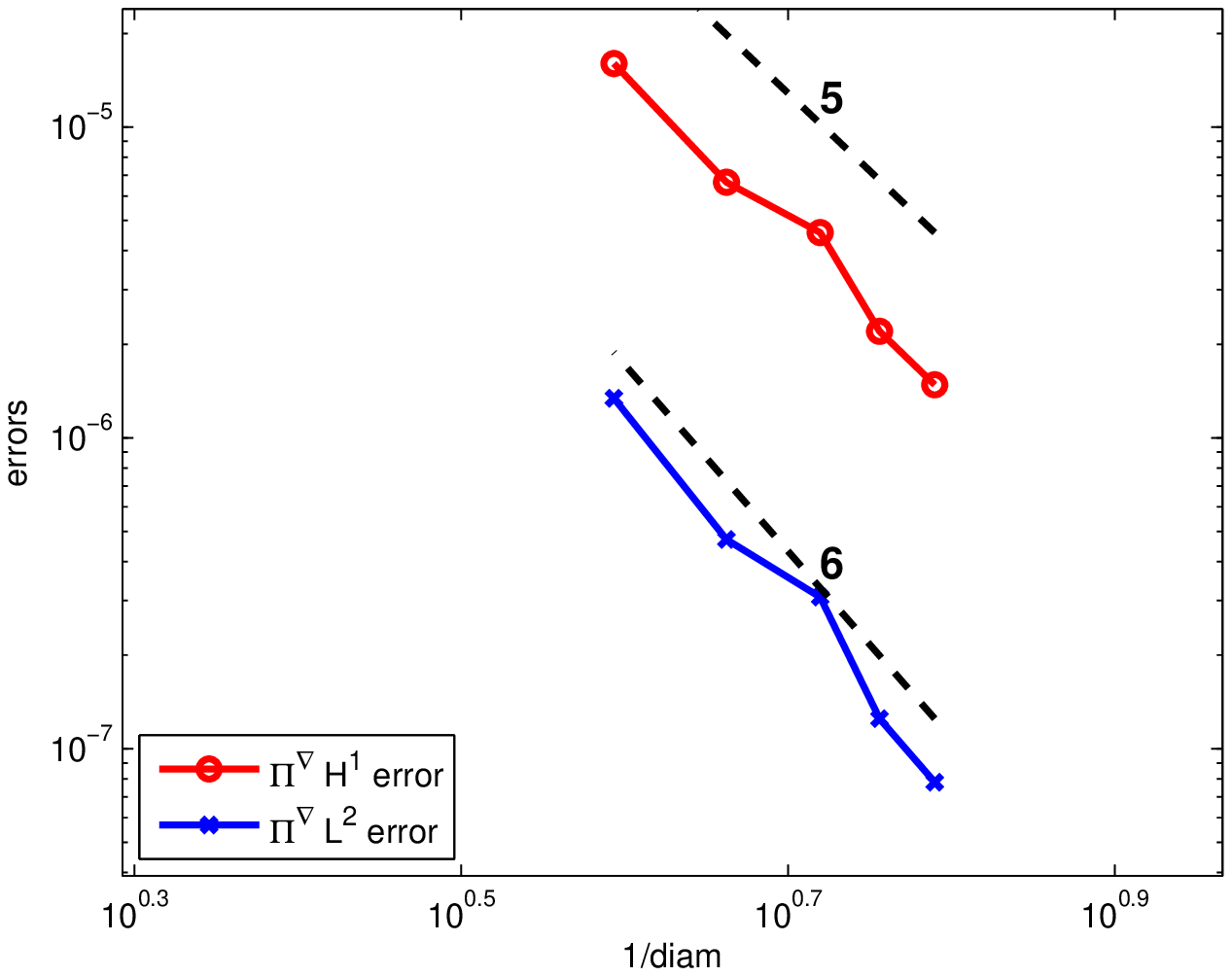}}
\caption{regular hexagonal mesh (left); VoronoiLloyd mesh (right); p=5} \label{h_errors_p5}
\end{figure}
The same considerations about Figure \ref{h_errors_p3} are still valid for Figure \ref{h_errors_p5}.
We only point out that the strange $L^2$ error behaviour for the final step is due to the machine precision error.
%%%%%%%%%%%%%%%%%%%%%%%%%%%%%%%%%%%%%%%%%%%%%%%%%%%%%%%%%%%%%%%%%%%%%%%%%%%

% -----------------------------------------------------------------
\section{Appendix} \label{sectionappendix}
% -----------------------------------------------------------------

In this appendix, we study the behaviour of the condition number of the global stiffness matrix of problem \eqref{discreteproblem} and we explore some alternatives in the choice of the local VEM basis.
We note that, as it happens in Finite Elements (see \cite{AdjeridAiffaFlahertyhierarchicalfiniteelementbasis} and \cite{BabuskaGriebelPitkarantaselectingshapefunctions}),
the main responsible for the growth of the condition number (when $p$ increases) are the internal ``bubble'' basis functions.
In Section \ref{subsectionthreeexplicitbasis}, we numerically investigate the behaviour of the condition number by changing the polynomial basis $\{q_{\boldalpha}\}$ of $\mathbb P_{p-2}(K)$ (for all $K\in \tauh$),
introduced in Section \ref{sectionVEMforthePoissonproblem} for the definition of the local internal degrees of freedom.
In Section \ref{subsectionavirtualgramschmidtprocess}, we discuss an orthogonalization technique that strongly reduces the condition number,
but is unstable with respect to machine precision.
% -----------------------------------------------------------------
\subsection{Three explicit bases} \label{subsectionthreeexplicitbasis}
% -----------------------------------------------------------------
In this subsection, we consider three explicit basis:
\begin{itemize}
\item $\{q_{\boldalpha}^1\}$, the same basis introduced for instance in \cite{VEMvolley} and \cite{hitchhikersguideVEM}, that is to say:
\begin{equation} \label{basisVEMvolley}
q_{\boldalpha}^1= \left( \frac{\mathbf x - \mathbf x_K}{h_K}  \right)^{\boldalpha}, \quad \forall \boldalpha\in \mathbb N_0^2,\, |\boldalpha|\le p-2,
\end{equation}
where $\mathbf{x_K}$ and $h_K$ are respectively the barycenter and the diameter of the polygon $K$.
\item $\{q_{\boldalpha}^2\}$,which is defined by
\begin{equation} \label{basisVEMvolleyscaled}
q_{\boldalpha}^2=\frac{q_{\boldalpha}^1}{||q_{\boldalpha}^1||_{0,K}}, \quad \forall \boldalpha\in \mathbb N_0^2,\, |\boldalpha|\le p-2.
\end{equation}
\item $\{q_{\boldalpha}^3\}$, which is a Legendre-type basis. In order to define it, we recall that $\NVK$ is the number of vertices of $K$ and $\{V_i\}_{i=1}^{\NVK}$ is the set of vertices of K;
moreover, we set
\[
\begin{split}
 &\widetilde {\mathbf{x_K}} = (\widetilde x_K , \widetilde y_K) = \left( \frac{ x_{V,\text{max}} - x_{V,\text{min}} }{2} ;  \frac{ y_{V,\text{max}} - y_{V,\text{min}} }{2} \right),\,\\
&h_K^x=| x_{V,\text{max}} + x_{V,\text{min}} |,\, h_K^y=| y_{V,\text{max}} + y_{V,\text{min}} |,\\
\end{split}
\]
where $ x_{V,\text{max}} =\max_{i=1}^{\NVK} x_i$, $ x_{V,\text{min}}= \min_{i=1}^{\NVK} x_i $, $ y_{V,\text{max}} =\max_{i=1}^{\NVK} y_i$, $ y_{V,\text{min}}= \min_{i=1}^{\NVK} y_i$.
Besides, let $L_s(\cdot)$ be the Legendre polynomial of degree $s$ on the segment $[-1,1]$ (see e.g. \cite{GhizzettiOssiciniquadratureformulae} for the properties of Legendre polynomials).
Then, we are able to define the basis:
\begin{equation} \label{basisLegendre}
q_{\boldalpha}^3=L_{\alpha_1}\left( 2\frac{x-\widetilde {x_K}}{h_K^x} \right) L_{\alpha_2}\left( 2\frac{y-\widetilde {y_K}}{h_K^y} \right),\quad \forall \boldalpha\in \mathbb N_0^2,\, |\boldalpha|\le p-2.
\end{equation}
\end{itemize}
We observe that the third choice should be, at least at a first glance, better than the other two, thanks to orthogonality properties of Legendre polynomials. We will see that instead this is not the case in general.
Indeed, the orthogonality properties of the Legendre basis are quickly lost when the considered domain is not rectangular.
In our tests, we consider the same meshes already used in Section \ref{sectionnumericalresults}; see Figure \ref{figurefourmeshes}.
In Figures \ref{figurecond3mesh_a} and \ref{figurecond3mesh_b}, we compare the behaviour of the condition number,
given by the Matlab command \emph{cond}, for the three choice of the bases mentioned  above and the four meshes.
\begin{figure}  [h]
\centering
\subfigure {\includegraphics [angle=0, width=0.45\textwidth]{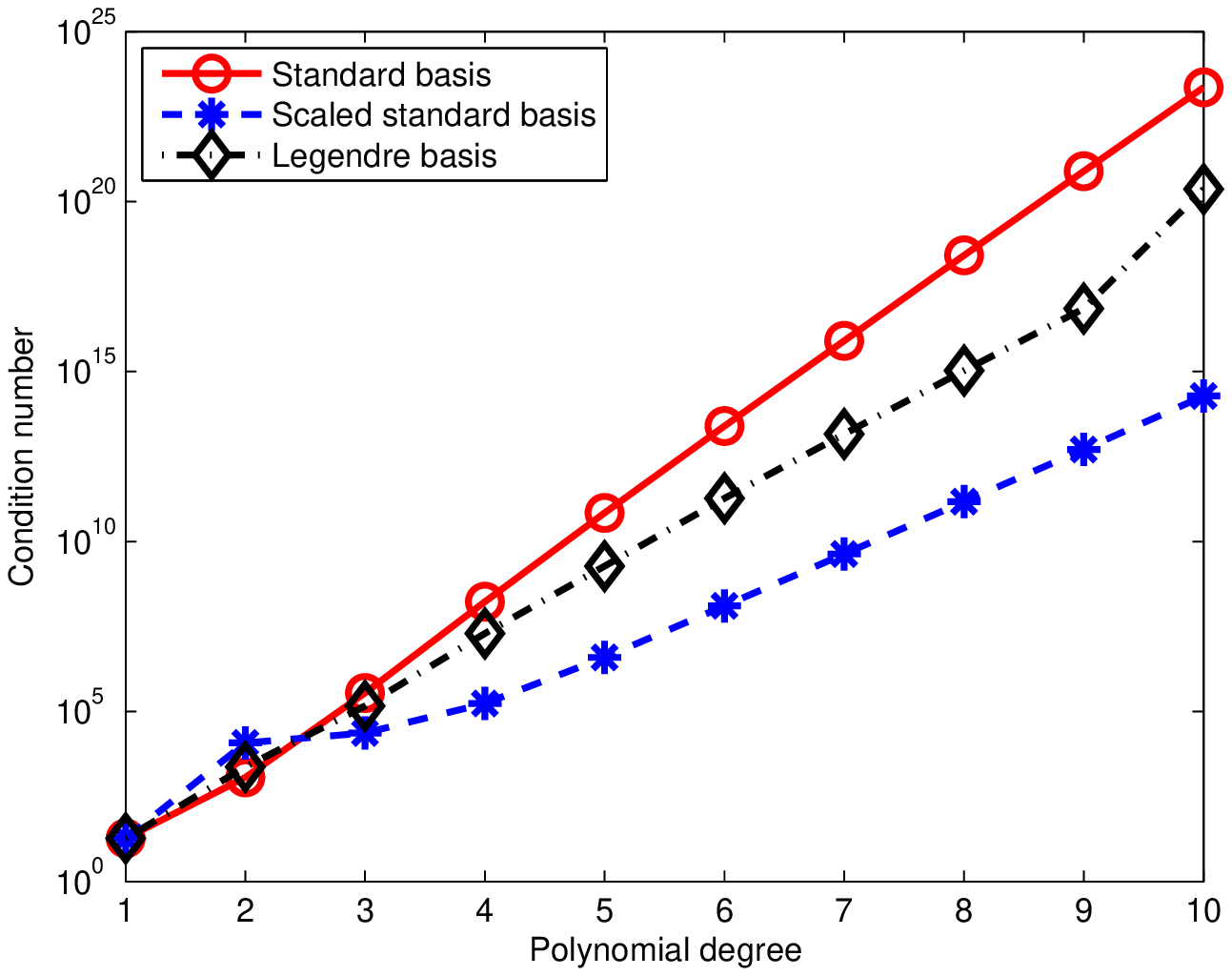}}
\subfigure {\includegraphics [angle=0, width=0.45\textwidth]{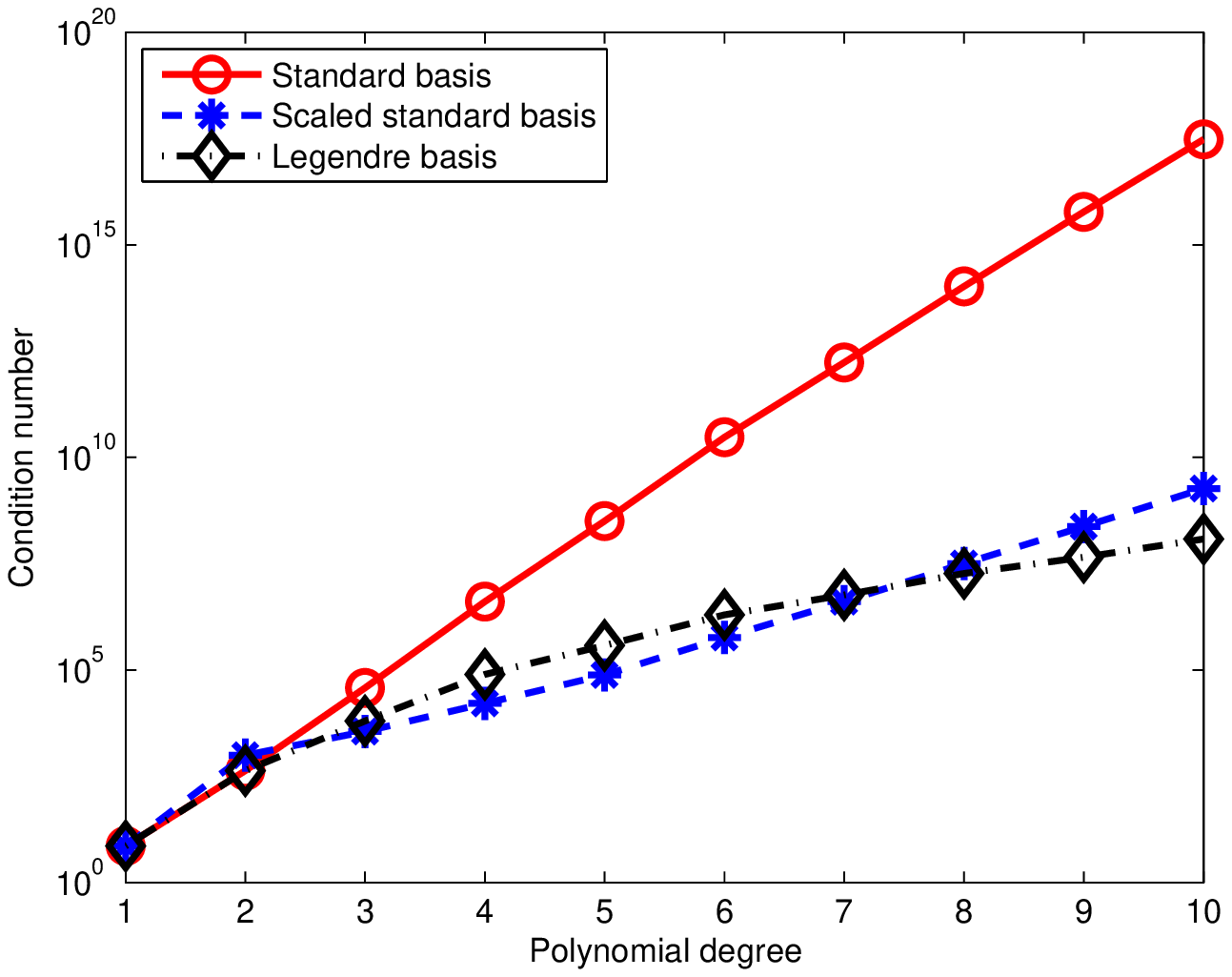}}
\caption{1: unstructured triangle mesh; 2: regular square mesh} \label{figurecond3mesh_a}
\end{figure}
\begin{figure}  [h]
\centering
\subfigure {\includegraphics [angle=0, width=0.45\textwidth]{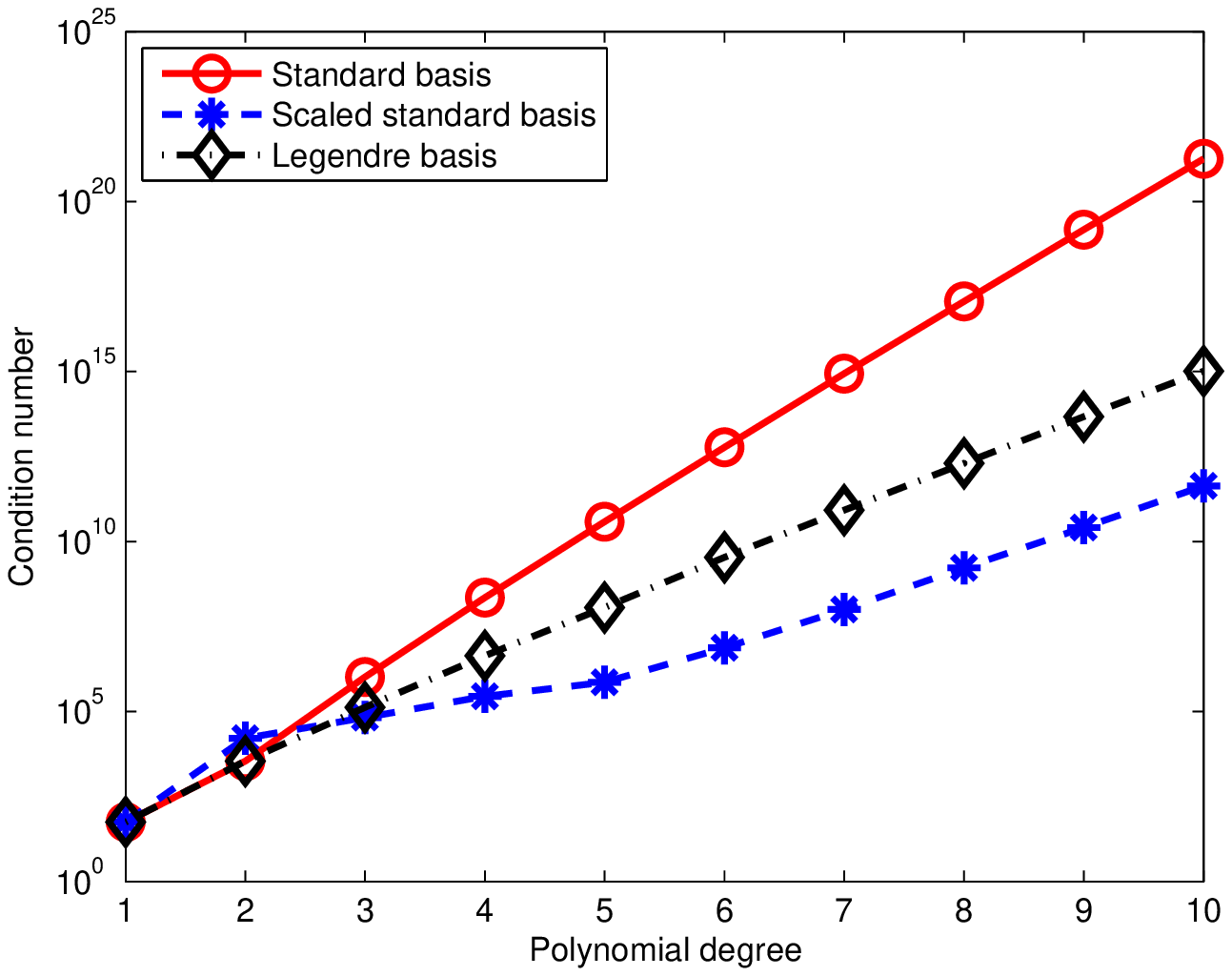}}
\subfigure {\includegraphics [angle=0, width=0.45\textwidth]{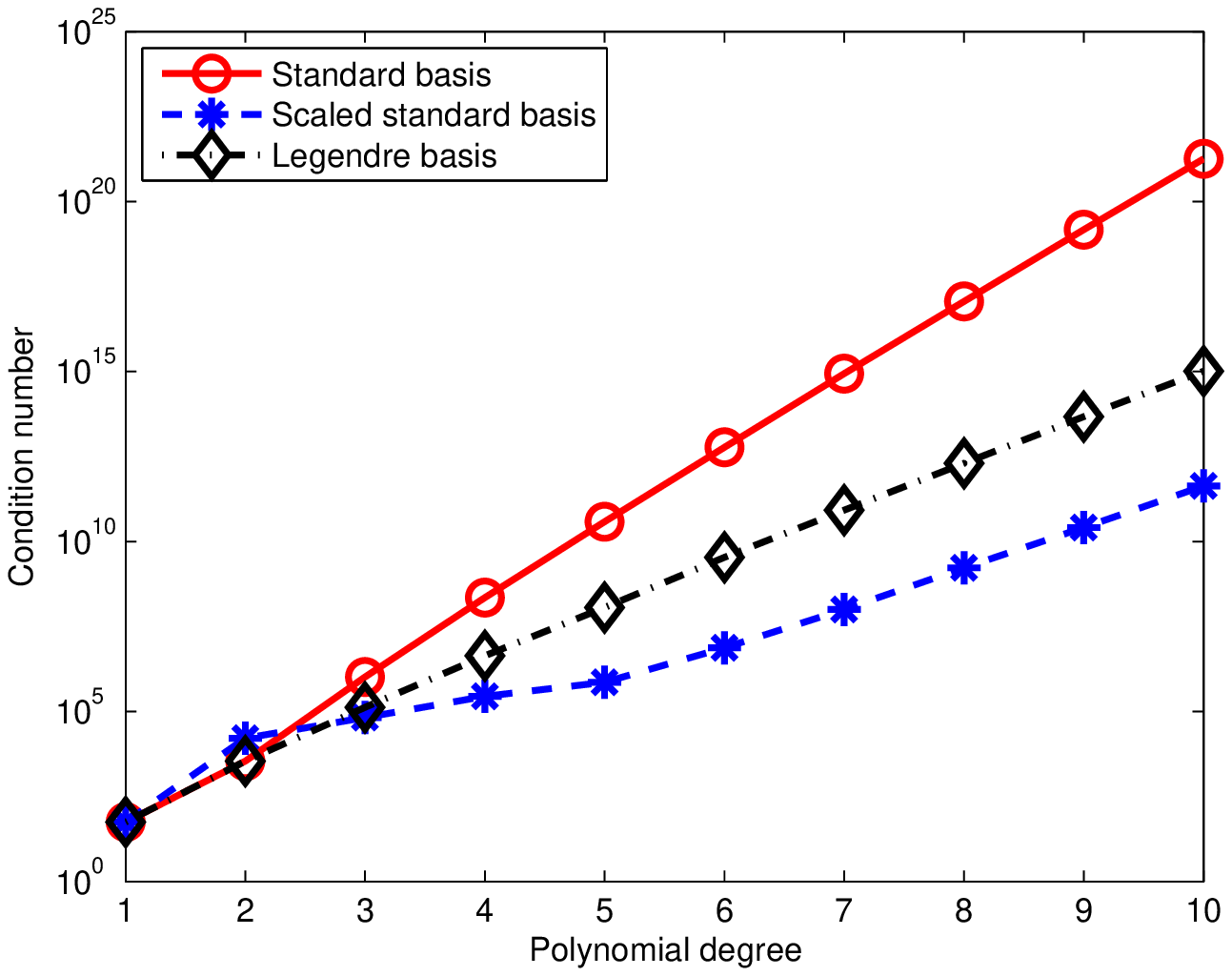}}
\caption{3: regular hexagonal mesh; 4: VoronoiLloyd mesh} \label{figurecond3mesh_b}
\end{figure}
We stress the fact that the Legendre-type basis performs better in the case of the square mesh; this is believable thanks to the orthogonality properties of the Legendre polynomials.
On the contrary, one can see that more general meshes, such as the hexagonal, unstructured triangular and Voronoi-Lloyd ones, the best result is obtained with the $L^2$-scaled basis.
Finally, we present in Figures \ref{figurethreeerrorssinsin_a}, \ref{figurethreeerrorssinsin_b}, \ref{figurethreeerrors2punto5_a} and \ref{figurethreeerrors2punto5_b} some ``comparison tests''
in which we report the error $|u-\Pinablap u_h|_{1,\Omega}$ (see \eqref{definitionPinabla}), by using the four different meshes in Figure \ref{figurefourmeshes} and the three different bases.
We consider the same two test cases of Section \ref{sectionnumericalresults}.
\begin{figure} [h]
\begin{overpic}[angle=0, width=0.45\textwidth]{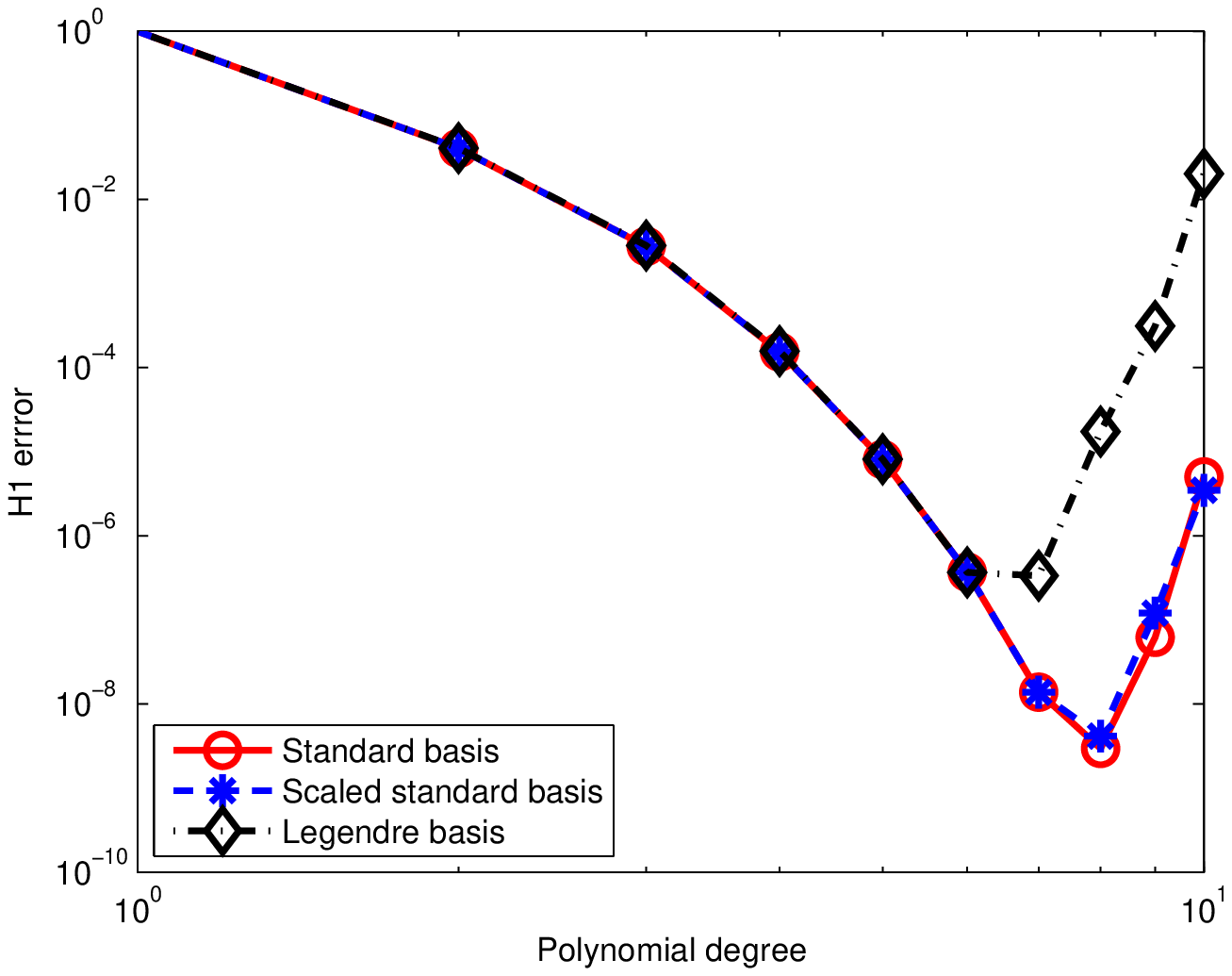}
\put(20,9.5){\colorbox{white}{\tiny 1}} \put(68,9.5){\tiny{2}} \put(94.5,9.5){\tiny{3}} \put(113,9.5){\tiny{4}} \put(127.5,9.5){\tiny{5}}
\put(139.5,9.5){\tiny{6}} \put(149,9.5){\tiny{7}} \put(158,9.5){\tiny{8}} \put(166,9.5){\tiny{9}} \put(170,9.5){\colorbox{white}{\tiny 10}}
\end{overpic}
\begin{overpic}[angle=0, width=0.45\textwidth]{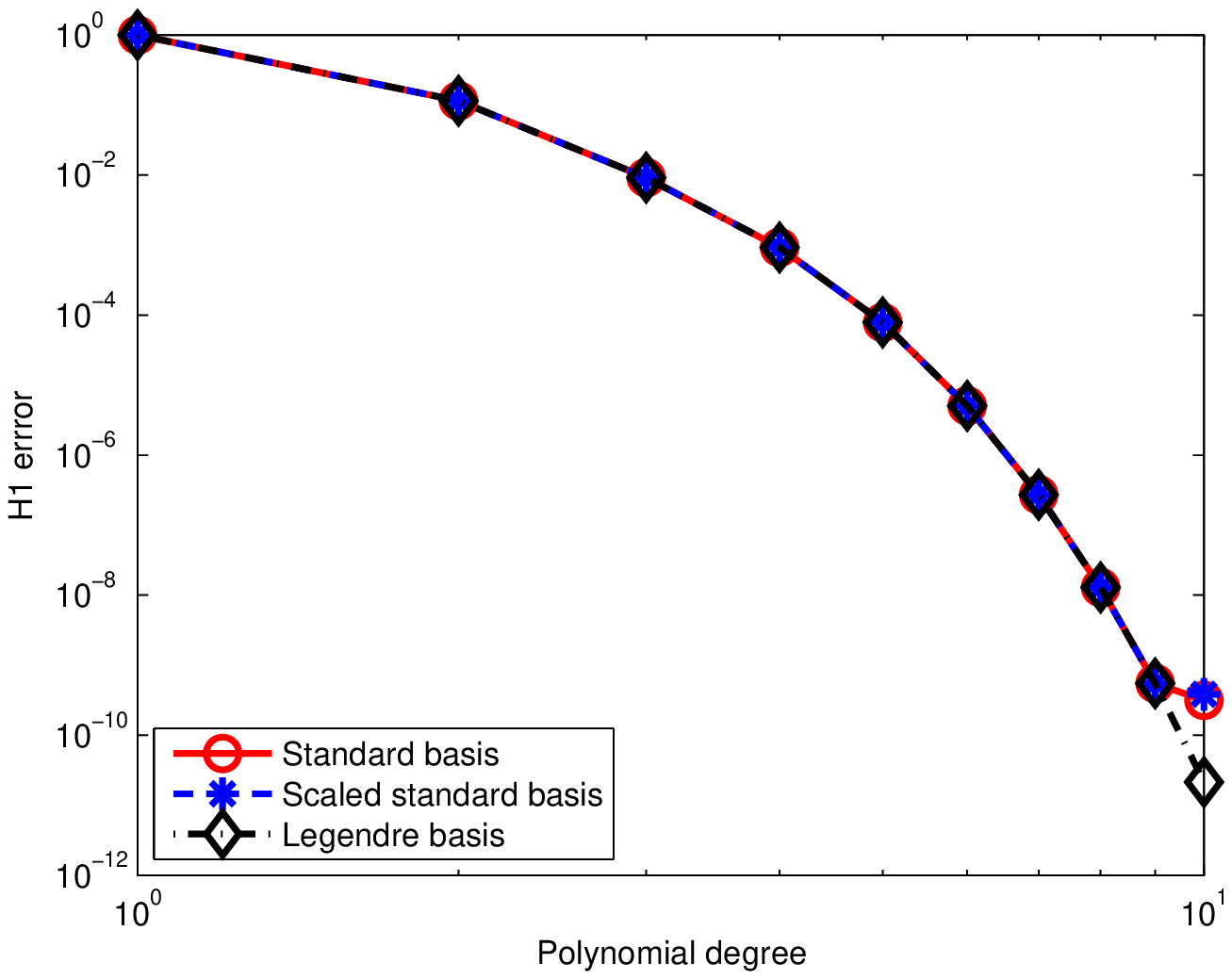}
\put(20,9.5){\colorbox{white}{\tiny 1}} \put(68,9.5){\tiny{2}} \put(94.5,9.5){\tiny{3}} \put(113,9.5){\tiny{4}} \put(127.5,9.5){\tiny{5}}
\put(139.5,9.5){\tiny{6}} \put(149,9.5){\tiny{7}} \put(158,9.5){\tiny{8}} \put(166,9.5){\tiny{9}} \put(170,9.5){\colorbox{white}{\tiny 10}}
\end{overpic}
\caption{$u(x,y)=\sin(\pi x)\sin(\pi y)$; 1: unstructured triangle mesh; 2: regular square mesh} \label{figurethreeerrorssinsin_a}
\end{figure}
\begin{figure} [h]
\begin{overpic}[angle=0, width=0.45\textwidth]{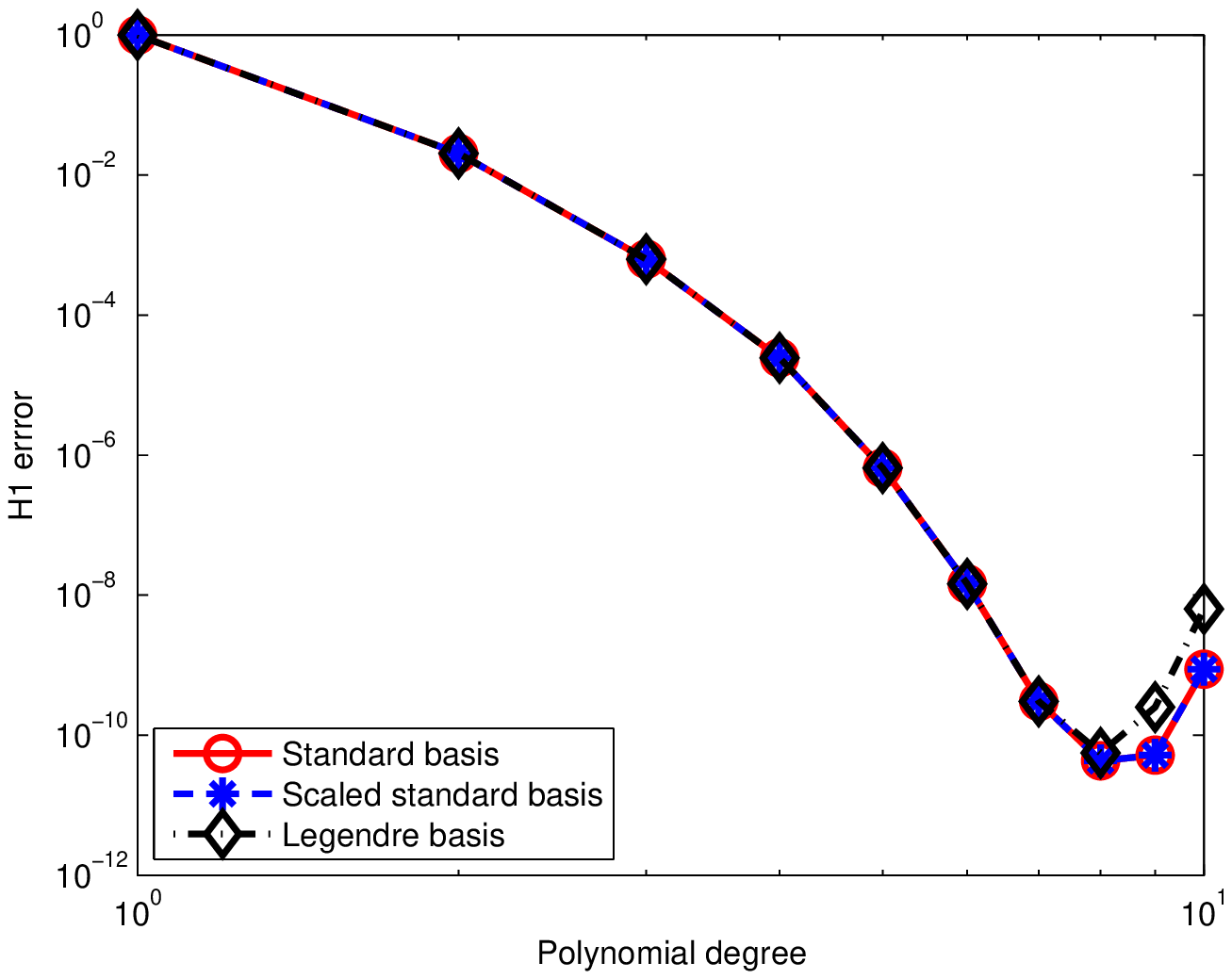}
\put(20,9.5){\colorbox{white}{\tiny 1}} \put(68,9.5){\tiny{2}} \put(94.5,9.5){\tiny{3}} \put(113,9.5){\tiny{4}} \put(127.5,9.5){\tiny{5}}
\put(139.5,9.5){\tiny{6}} \put(149,9.5){\tiny{7}} \put(158,9.5){\tiny{8}} \put(166,9.5){\tiny{9}} \put(170,9.5){\colorbox{white}{\tiny 10}}
\end{overpic}
\begin{overpic}[angle=0, width=0.45\textwidth]{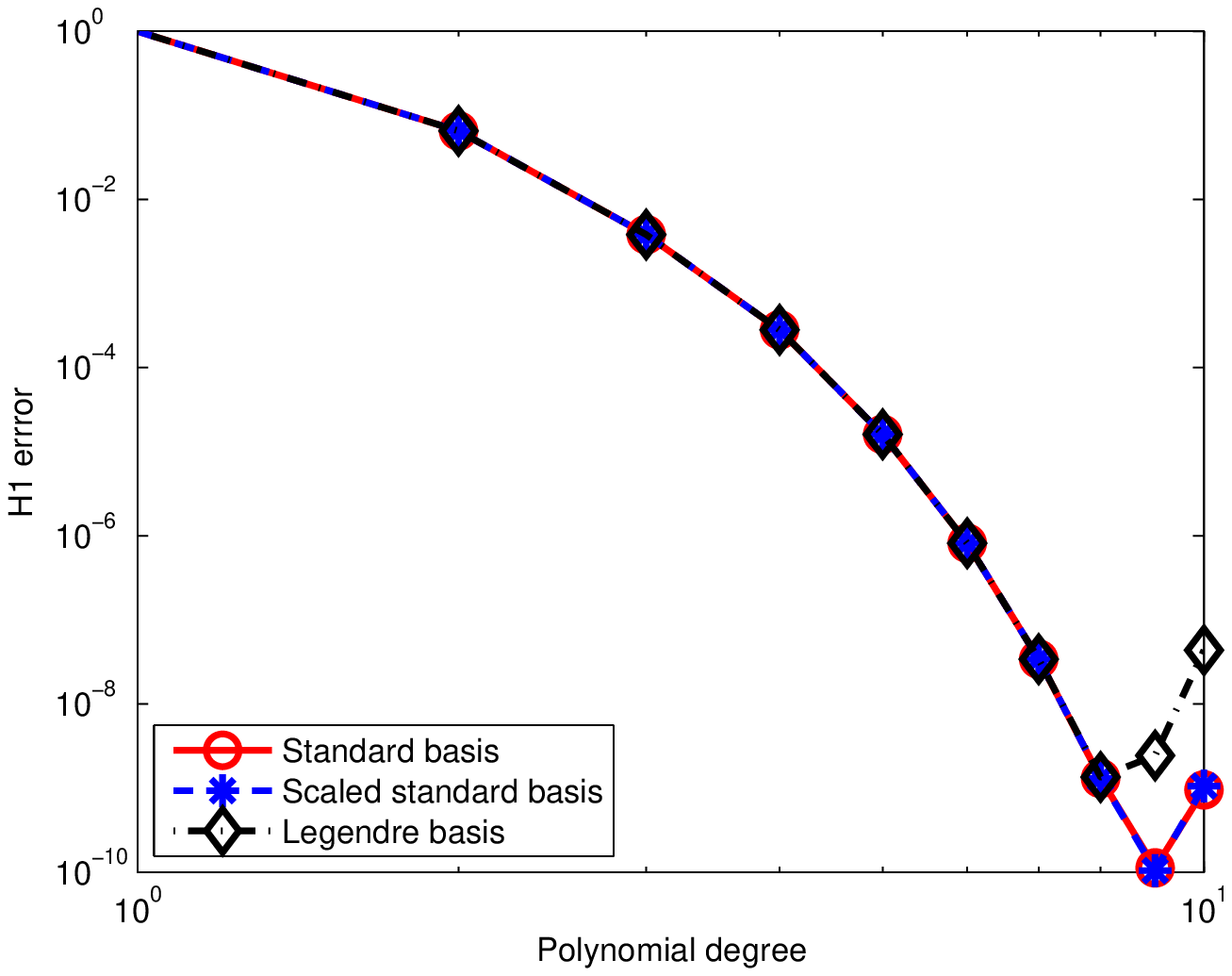}
\put(20,9.5){\colorbox{white}{\tiny 1}} \put(68,9.5){\tiny{2}} \put(94.5,9.5){\tiny{3}} \put(113,9.5){\tiny{4}} \put(127.5,9.5){\tiny{5}}
\put(139.5,9.5){\tiny{6}} \put(149,9.5){\tiny{7}} \put(158,9.5){\tiny{8}} \put(166,9.5){\tiny{9}} \put(170,9.5){\colorbox{white}{\tiny 10}}
\end{overpic}
\caption{$u(x,y)=\sin(\pi x)\sin(\pi y)$; 3: regular hexagonal mesh; 4: VoronoiLloyd mesh} \label{figurethreeerrorssinsin_b}
\end{figure}
From Figures \ref{figurethreeerrorssinsin_a} and \ref{figurethreeerrorssinsin_b}, one can see that the Legendre-type basis is a good choice for the square case, while on triangles it is very unstable;
besides, for general meshes, it seems that the slope of the error with the other two bases is almost the same and performs better than the Legendre-type basis.
\begin{figure} [h]
\begin{overpic}[angle=0, width=0.45\textwidth]{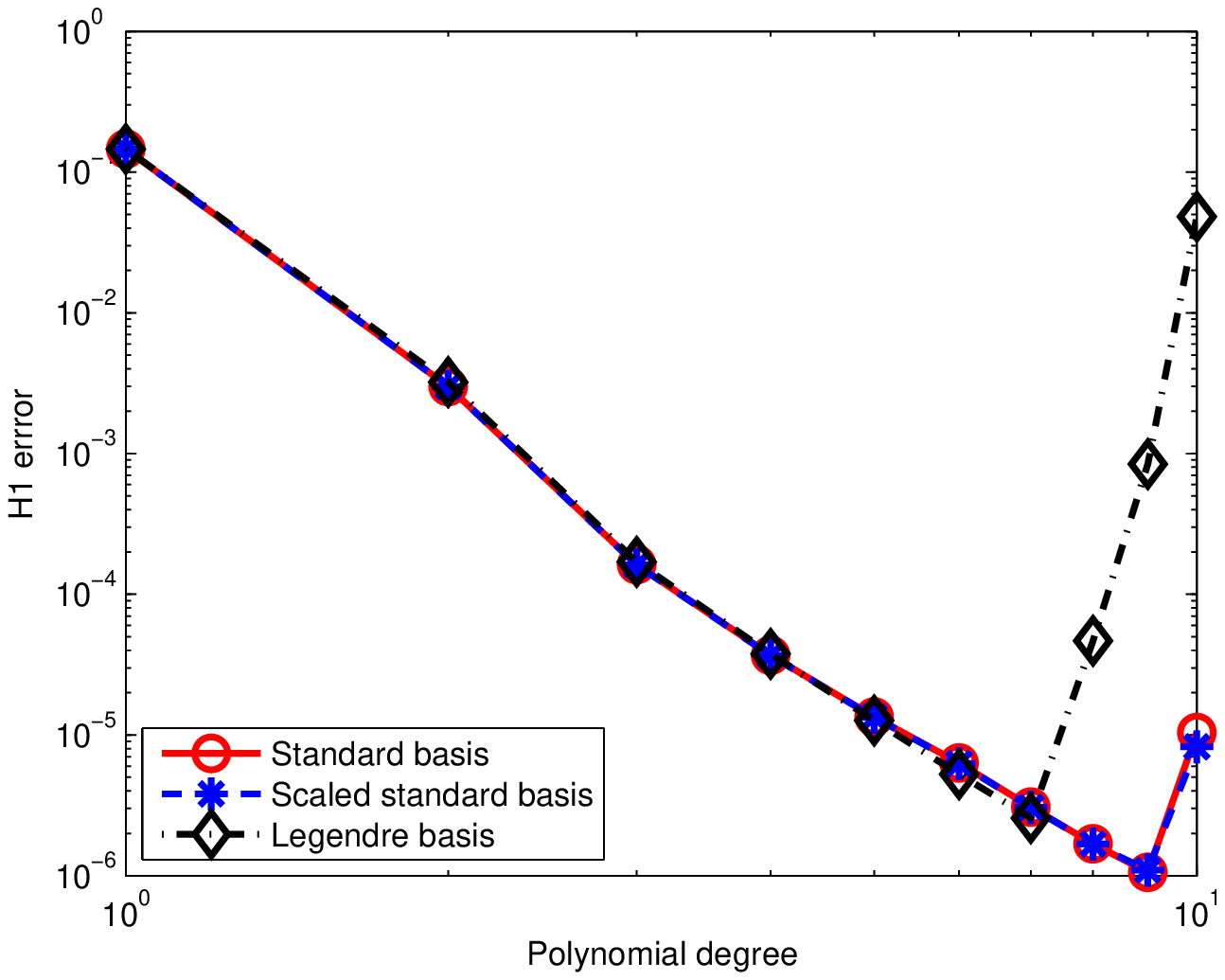}
\put(20,9.5){\colorbox{white}{\tiny 1}} \put(68,9.5){\tiny{2}} \put(94.5,9.5){\tiny{3}} \put(113,9.5){\tiny{4}} \put(127.5,9.5){\tiny{5}}
\put(139.5,9.5){\tiny{6}} \put(149,9.5){\tiny{7}} \put(158,9.5){\tiny{8}} \put(166,9.5){\tiny{9}} \put(170,9.5){\colorbox{white}{\tiny 10}}
\end{overpic}
\begin{overpic}[angle=0, width=0.45\textwidth]{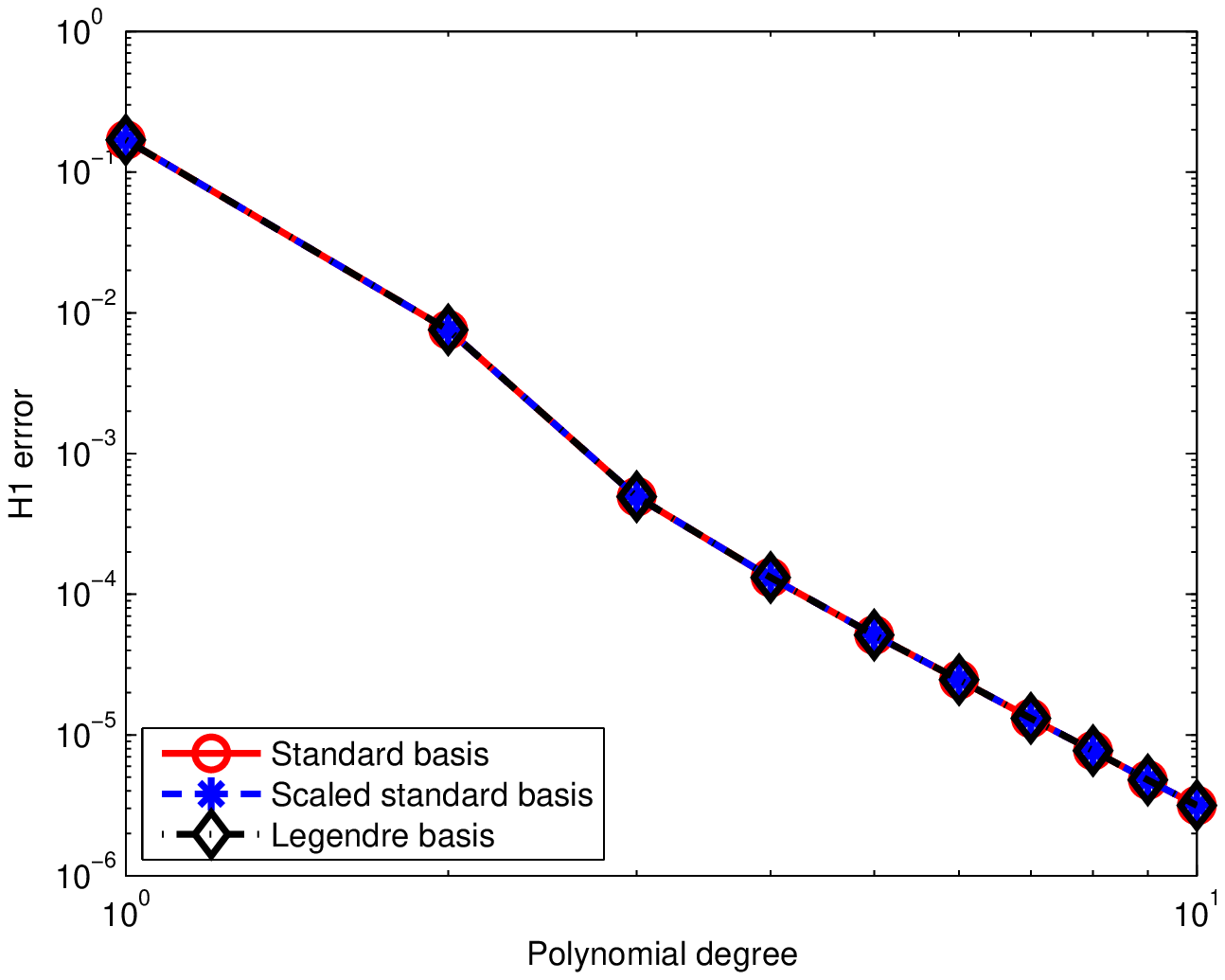}
\put(20,9.5){\colorbox{white}{\tiny 1}} \put(68,9.5){\tiny{2}} \put(94.5,9.5){\tiny{3}} \put(113,9.5){\tiny{4}} \put(127.5,9.5){\tiny{5}}
\put(139.5,9.5){\tiny{6}} \put(149,9.5){\tiny{7}} \put(158,9.5){\tiny{8}} \put(166,9.5){\tiny{9}} \put(170,9.5){\colorbox{white}{\tiny 10}}
\end{overpic}
\caption{$u(r,\theta) = r^{2.5}\sin(2.5\theta)$; 1:unstructured triangle mesh; 2: regular square mesh} \label{figurethreeerrors2punto5_a}
\end{figure}
\begin{figure}  [h]
\begin{overpic}[angle=0, width=0.45\textwidth]{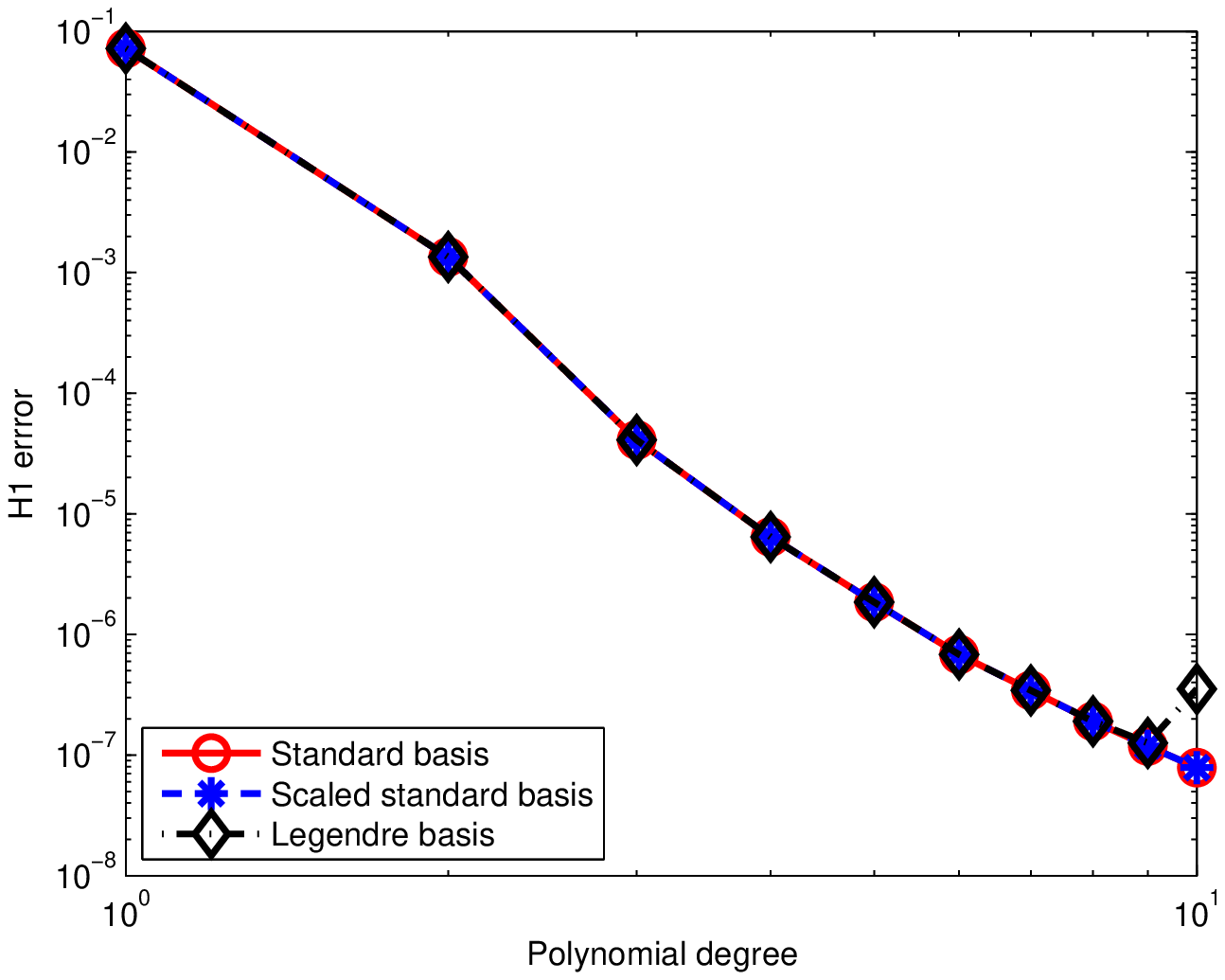}
\put(20,9.5){\colorbox{white}{\tiny 1}} \put(68,9.5){\tiny{2}} \put(94.5,9.5){\tiny{3}} \put(113,9.5){\tiny{4}} \put(127.5,9.5){\tiny{5}}
\put(139.5,9.5){\tiny{6}} \put(149,9.5){\tiny{7}} \put(158,9.5){\tiny{8}} \put(166,9.5){\tiny{9}} \put(170,9.5){\colorbox{white}{\tiny 10}}
\end{overpic}
\begin{overpic}[angle=0, width=0.45\textwidth]{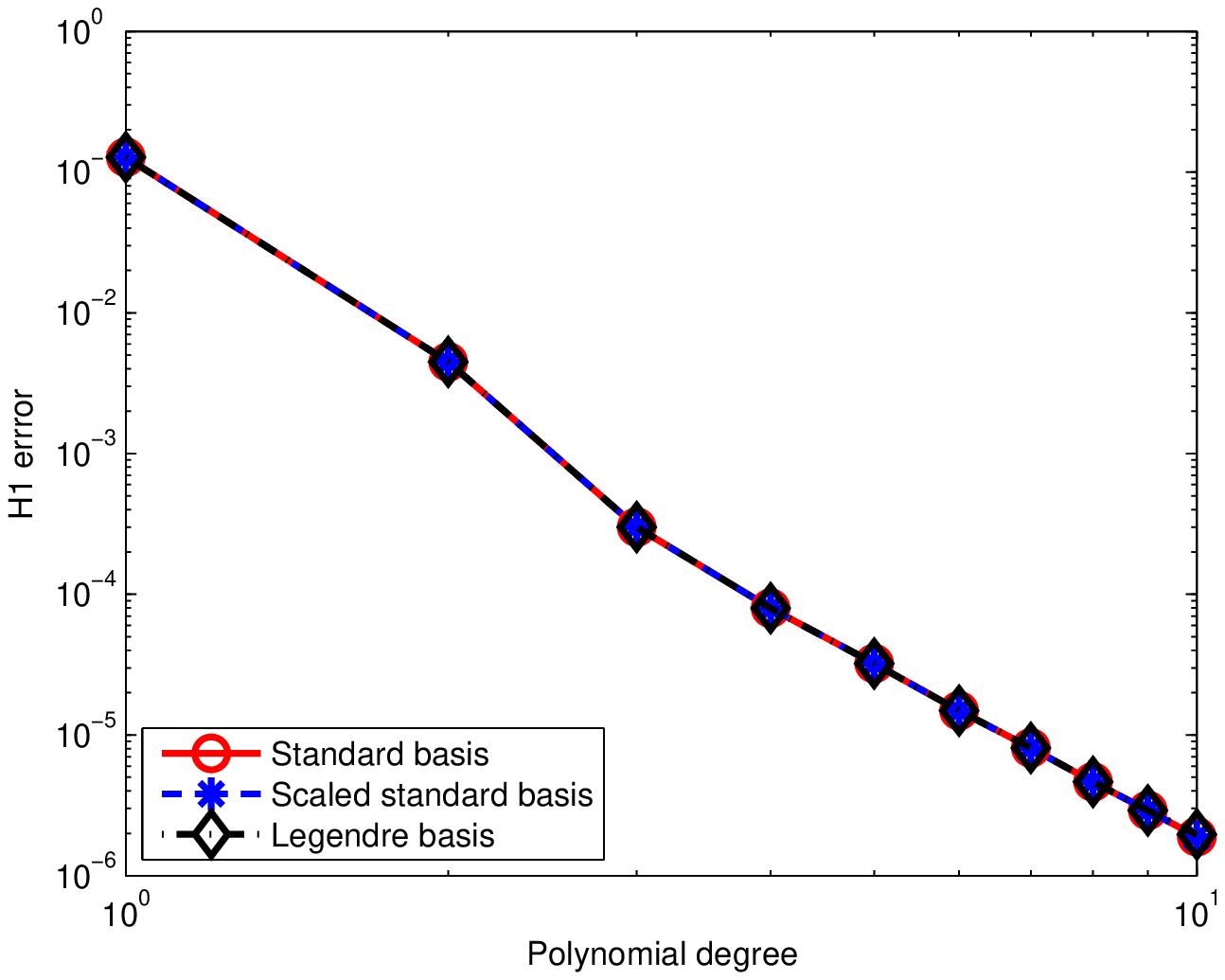}
\put(20,9.5){\colorbox{white}{\tiny 1}} \put(68,9.5){\tiny{2}} \put(94.5,9.5){\tiny{3}} \put(113,9.5){\tiny{4}} \put(127.5,9.5){\tiny{5}}
\put(139.5,9.5){\tiny{6}} \put(149,9.5){\tiny{7}} \put(158,9.5){\tiny{8}} \put(166,9.5){\tiny{9}} \put(170,9.5){\colorbox{white}{\tiny 10}}
\end{overpic}
\caption{$u(r,\theta) = r^{2.5}\sin(2.5\theta)$; 3: regular hexagonal mesh; 4: VoronoiLloyd mesh} \label{figurethreeerrors2punto5_b}
\end{figure}
The same considerations for Figures \ref{figurethreeerrorssinsin_a} and \ref{figurethreeerrorssinsin_b} hold for Figures \ref{figurethreeerrors2punto5_a} and \ref{figurethreeerrors2punto5_b}.\\
We point out that in our numerical tests we have used a direct solver in order to work out the global system arising from the discrete problem \eqref{discreteproblem}.
A consequence of this fact is that the condition number of the global matrix does not affect the resolution of the linear system as it would do if we used an iterative solver.
This explains why the behavior of the errors with the choice of the classical basis $\{q_{\boldalpha}^1\}$ and the scaled basis $\{q_{\boldalpha}^2\}$ is almost the same,
notwithstanding the large difference in the condition number of the global matrix as shown in Figure \ref{figurecond3mesh_a} and in Figure \ref{figurecond3mesh_b}.

%This fact explains why the behaviour of the errors with the choice of the classical basis $\{q_{\boldalpha}^1\}$ and the scaled basis $\{q_{\boldalpha}^2\}$ is almost the same.
%We stress that if one wants to use an iterative method, he should use the basis $\{q_{\boldalpha}^3\}$ as suggested.
% -----------------------------------------------------------------
\subsection{A ``Virtual'' Gram Schmidt process} \label{subsectionavirtualgramschmidtprocess}
% -----------------------------------------------------------------
From the previous subsection, it is clear that by suitably changing the basis of the space one obtains better condition numbers for the global stiffness matrix.
Despite this, from Figures \ref{figurecond3mesh_a} and \ref{figurecond3mesh_b} we note that the condition numbers are still large and, although in our codes we use a direct solver, it would be preferable to reduce such numbers.
Therefore, in this subsection, we consider a technique which considerably reduces the condition number of the global stiffness matrix,
but at the price of an unstable propagation of the machine error precision (as better discussed later).
In particular, we begin working locally and we fix an element $K \in \tauh$; we consider the local stiffness matrix $\AhK$.
We recall that the $(i,j)$-th entry of $\AhK$ has the following form (see \cite{VEMvolley}):
\[
[  \AhK  ] _{i,j} := [\mathbf{A}]_{i,j} =\ahK (\varphi _i, \varphi _j) = a^K ( \Pinabla \varphi _i , \Pinabla \varphi _j) + S^K(\varphi _i -\Pinabla \varphi _i , \varphi _j - \Pinabla \varphi _j).
\]
Moreover, we recall that 
\[
\AhK \in := \mathbb R ^{B_d+I_d, B_d, + I_d},
\]
where $I_d$ and $B_d$ are respectively the internal and the boundary degrees of freedom.
As already mentioned, the internal ``bubble'' basis functions are the main responsible for the growth of the condition number when $p$ increases.\\
In Section \ref{subsubsectionavirtualGramSchmidtmethod}, we introduce a technique which allows us to orthonormalize the internal elements of the local basis with respect to the discrete bilinear form $\ahK$.
In Section \ref{subsubsectionnumericalresultsaboutthevirtualgramschmidtmethod}, we present numerical experiments in which we compare the condition numbers and $H^1$ type errors of both standard VEM 
(i.e. those using the basis $\{q_{\boldalpha}^1\}$) and this new choice.
% -----------------------------------------------------------------
\subsubsection{A ``virtual'' Gram Schmidt method} \label{subsubsectionavirtualGramSchmidtmethod}
% -----------------------------------------------------------------
Let $G_d$, $B_d$ and $I_d$ be the number of total, boundary-vertex and internal degrees of freedom respectively  over $K$ of the space $\VhK$. Let 
\begin{equation} \label{oldbasis}
\{ \varphi _i \}_{i=1}^{G_d} = \{ \varphi ^B _j \}_{j=1}^{B_d} \cup \{ \varphi ^I _l \}_{l=1}^{I_d}
\end{equation}
be the basis of $\VhK$ introduced in \cite{hitchhikersguideVEM}, where $\{ \varphi ^B _j \}_{j=1}^{B_d}$ represents the ``boundary'' basis elements and $\{ \varphi ^I _l \}_{l=1}^{I_d}$ represents the ``internal'' basis elements.
We want here to introduce a new basis which shares the same ``boundary'' elements of  (\ref{oldbasis}) and which has orthonormalized ``internal'' elements with respect to the discrete local bilinear form $\ahK$.
The new basis reads
\begin{equation} \label{newbasis}
\{ \tildephi _i \}_{i=1}^{G_d} = \{ \varphi ^B _j \}_{j=1}^{B_d} \cup \{ \tildephi ^I _l \}_{l=1}^{I_d},
\end{equation}
with the $\tildephi ^I _l$ are to be defined.
Since we want orthonormality with respect to $\ahK$, we simply apply the Gram Schmidt method to the internal elements of (\ref{oldbasis}).
We observe that each $\tildephi ^I _l$ can be split as:
\[
\tildephi ^I _l = \sum _{k=1}^{l} \lambda _{l,k} \varphi ^I _k,
\]
where the $\lambda _{l,k}$ are the Gram Schmidt coefficients.\\
Before giving an explicit representation of the $\lambda_{k,l}$, we want to observe the following fact. Let $\widetilde {\mathbf A}  _{II}$ be the matrix having $(i,j)$-th entry given by $\ahK (\tildephi _i ^I, \tildephi _j ^I)$.
Then, thanks to the choice of using the Gram Schmidt method, we have that $\widetilde {\mathbf A} _{II}$, is a diagonal matrix.
We can explicitly represent such matrix by means of:
\[
\widetilde {\mathbf A} _{II} = \mathbf{\Lambda} \cdot \mathbf{A _{II}} \cdot \mathbf{\Lambda}^T,
\]
where $\mathbf{\Lambda}$ is the left lower triangular matrix having in the $l$-th row the coefficients $\lambda_{l,k}$, $k=1,\dots,l$ of $\tildephi _l ^I$.\\
Out task, now, is to find the matrix $\mathbf{\Lambda}$. We stress the fact that in order to compute such matrix we can not use explicitly the ``internal'' basis elements since they are virtual;
notwithstanding, we can only use the scalar products $\ahK (\varphi _i^I , \varphi_j^I)$.
We start by orthogonalizing the basis. We know that:
\[
\tildephi _1^I = \varphi_1^I,\quad \tildephi_l^I = \varphi_l^I - \sum_{k=1}^{l-1} \frac{\ahK(\varphi_l^I,\tildephi^I_k)}{\ahK(\tildephi^I_k,\tildephi^I_k)} \tildephi^I_k,\quad \forall l=2,\dots,I_d.
\]
%In order to find out some recursive formula on the coefficients, we give the full representation for, let's say, the first three elements:
%\begin{equation} \label{representationonthefirst3elements}
%\begin{split}
%& \tildephi_1^I = \varphi _1^I := \lambda_{1,1} \varphi_1^I,\\
%& \tildephi_2^I = \varphi _2^I - \frac{(\varphi _2^I,\tildephi _1^I)}{(\tildephi _1^I,\tildephi _1^I)} \tildephi _1^I = \varphi_2^I - \lambda_{1,1}\frac{(\varphi _2^I,\tildephi _1^I)}{(\tildephi _1^I,\tildephi _1^I)} \varphi_1^I:=
%                             \lambda_{2,1}\varphi_1^I + \lambda_{2,2} \varphi_2^I,\\
%& \tildephi_3^I = \varphi_3^I - \frac{(\varphi _3^I,\tildephi _2^I)}{(\tildephi _2^I,\tildephi _2^I)}\tildephi_2^I - \frac{(\varphi _3^I,\tildephi _1^I)}{(\tildephi _1^I,\tildephi _1^I)} \tildephi_1^I\\
%& = \varphi_3^I -\left[ \lambda_{2,2} \frac{(\varphi _3^I,\tildephi _2^I)}{(\tildephi _2^I,\tildephi _2^I)}  \right] \varphi_2^I 
%       -\left[ \lambda_{2,1}\frac{(\varphi _3^I,\tildephi _2^I)}{(\tildephi _2^I,\tildephi _2^I)} + \lambda_{1,1} \frac{(\varphi _3^I,\tildephi _1^I)}{(\tildephi _1^I,\tildephi _1^I)}  \right] \varphi_1^I
%      := \lambda_{3,1} \varphi^I_1 +  \lambda_{3,2}\varphi^I_2  + \lambda_{3,3}\varphi^I_3.
%\end{split}
%\end{equation}
By induction, one can show that the coefficients of the function $\tildephi _l^I$ reads:
\begin{equation} \label{coefficientslthelement}
\lambda_{l,l}=1 ,\quad \lambda_{l,k} = - \sum_{j=k}^{l-1} \lambda_{j,k} \frac{\ahK(\varphi_l^I, \tildephi _j^I)}{\ahK(\tildephi _j^I,\tildephi _j^I)},\quad \forall k=1,\dots,l-1.
\end{equation}
We observe that the coefficients $\lambda_{j,k}$ in (\ref{coefficientslthelement}) are already known at the l-th step.
On the other hand, we need a routine which permits to compute $\ahK(\varphi_l^I, \tildephi _j^I)$ and $\ahK(\tildephi _j^I,\tildephi _j^I)$, without knowing explicitly the basis and only knowing the scalar products $\ahK(\varphi _i^I, \varphi_j^I)$. One has:
\begin{equation} \label{contisudenominatore}
\ahK(\tildephi _j^I,\tildephi _j^I) = \ahK(\sum _{i=1}^{j} \lambda_{j,i} \varphi_i^I , \sum _{h=1}^{j} \lambda_{j,h} \varphi_h^I)=\mathbf{\Lambda}(j,:) \cdot \mathbf{A_{II}} \cdot \mathbf{\Lambda}^T(:,j)
\end{equation}
and
\begin{equation} \label{contisunumeratore}
\ahK(\varphi_l^I, \tildephi _j^I) = \ahK(\varphi_l^I, \sum_{i=1}^{j}\lambda_{j,i}\varphi_i^I) = \mathbf{A_{II}} (l,:) \cdot \mathbf{\Lambda} ^T(:,j).
\end{equation}
Thus, we have the tools for finding out the coefficients of $\mathbf{\Lambda}$.
We remark that we have obtained a new local virtual basis and that also the internal moments in the definition of the degrees of freedom are changed.
In particular, the basis of scaled monomials $\{m_{\boldalpha}\}$ has changed to a new basis $\{\widetilde m_{\boldalpha}\}$, but for the computation of the new local stiffness matrix, we do not need to know explicitly such a new basis.
In fact, let \textbf{I} be the identity matrix having size equal to $B_d$. Then, we define:
\[
\widetilde{\mathbf{\Lambda}} = \begin{bmatrix}
\mathbf{I} & \mathbf{0} \\ \mathbf{0}^T & \mathbf{\Lambda}\\
\end{bmatrix}
\]
and the new local stiffness matrix is obtained by:
\[
\widetilde {\mathbf{A}} := \begin{bmatrix} \widetilde{\mathbf{A}}_{BB} & \widetilde{\mathbf{A}}_{BI} \\ \widetilde{\mathbf{A}}_{IB} & \widetilde{\mathbf{A}}_{II}  \\    \end{bmatrix}
=\widetilde{\mathbf{\Lambda}}  \cdot \begin{bmatrix} \mathbf{A}_{BB} & \mathbf{A}_{BI} \\ \mathbf{A}_{IB} & \mathbf{A}_{II}  \\    \end{bmatrix}  \cdot \widetilde{\mathbf{\Lambda}}^T
= \begin{bmatrix} \mathbf{A}_{BB} & \widetilde{\mathbf{A}}_{BI} \\ \widetilde{\mathbf{A}}_{IB} & \mathbf{D}  \\    \end{bmatrix}.
\]
It remains the normalization of the internal basis elements.
We observe that the entries of the diagonal matrix $\widetilde {\mathbf A} _{II}$ are $[\widetilde {\mathbf A} _{II}]_{i,i} = \ahK(\tildephi _i, \tildephi_i)$.
Therefore, we just need to multiply on the left and on the right the matrix $\widetilde {\mathbf{\Lambda}}$ by
\[
\widetilde{\mathbf{D}} = \begin{bmatrix}    \mathbf{I} & \mathbf{0} \\ \mathbf{0}^T & |\widetilde {\mathbf A} _{II}|^{-\frac{1}{2}}     \end{bmatrix},
\]
where $|\widetilde {\mathbf A} _{II}|^{-\frac{1}{2}}$ is the diagonal matrix such that $[\widetilde {\mathbf A} _{II}|^{-\frac{1}{2}}]_{i,i} = \sqrt{|\widetilde {\mathbf A} _{II}]_{i,i}|}$, obtaining
\begin{equation} \label{lambdagrandetildetilde}
\widetilde{\widetilde{\mathbf{\Lambda}}} = \begin{bmatrix}   \mathbf{I} & \mathbf{0} \\ \mathbf{0}^T & \mathbf{\Lambda} |\widetilde {\mathbf A} _{II}|^{-\frac{1}{2}}   \end{bmatrix}.
\end{equation}
Finally, the definitive new local stiffness matrix reads:
\begin{equation} \label{Atildetilde}
\widetilde{\widetilde{\mathbf{A}}} = \begin{bmatrix} \widetilde{\widetilde{\mathbf{A}}}_{BB} & \widetilde{\widetilde{\mathbf{A}}}_{BI} \\ \widetilde{\widetilde{\mathbf{A}}}_{IB} & \widetilde{\widetilde{\mathbf{A}}}_{II}  \\    \end{bmatrix}
= \widetilde{\widetilde{\mathbf{\Lambda}}}  \cdot \begin{bmatrix} \mathbf{A}_{BB} & \mathbf{A}_{BI} \\ \mathbf{A}_{IB} & \mathbf{A}_{II}  \\    \end{bmatrix}  \cdot \widetilde{\widetilde{\mathbf{\Lambda}}}^T
= \begin{bmatrix} \mathbf{A}_{BB} & \widetilde{\widetilde{\mathbf{A}}}_{BI} \\ \widetilde{\widetilde{\mathbf{A}}}_{IB} & \mathbf{I}_{I_d, I_d}  \\    \end{bmatrix}.
\end{equation}
It is clear that the global virtual basis is obtained with a classical conforming coupling of the boundary degrees of freedom.
Finally, we stress the fact that in order to compute the projection $\Pinabla$, one may use the procedure described above; by means of straightforward computations, it is easy to recover analogous results to those presented in \cite{hitchhikersguideVEM}.
% -----------------------------------------------------------------
\subsubsection{Numerical results using the ``virtual'' Gram Schmidt method} \label{subsubsectionnumericalresultsaboutthevirtualgramschmidtmethod}
% -----------------------------------------------------------------
In the present section, we present some numerical experiments on the behaviour of the condition number.
 We consider the four different meshes in Figure \ref{figurefourmeshes} and we compare the condition number of the $L^2$ scaled basis introduced in Section \ref{subsectionthreeexplicitbasis}
and the new Gram-Schmidt basis, described in Section \ref{subsubsectionavirtualGramSchmidtmethod}.
We remind that, from the previous numerical experiments, the $L^2$ scaled basis seems to be the most well conditioned among the choices of Section \ref{subsubsectionavirtualGramSchmidtmethod}.
\begin{figure} [h]
\centering
\subfigure {\includegraphics [angle=0, width=0.45\textwidth]{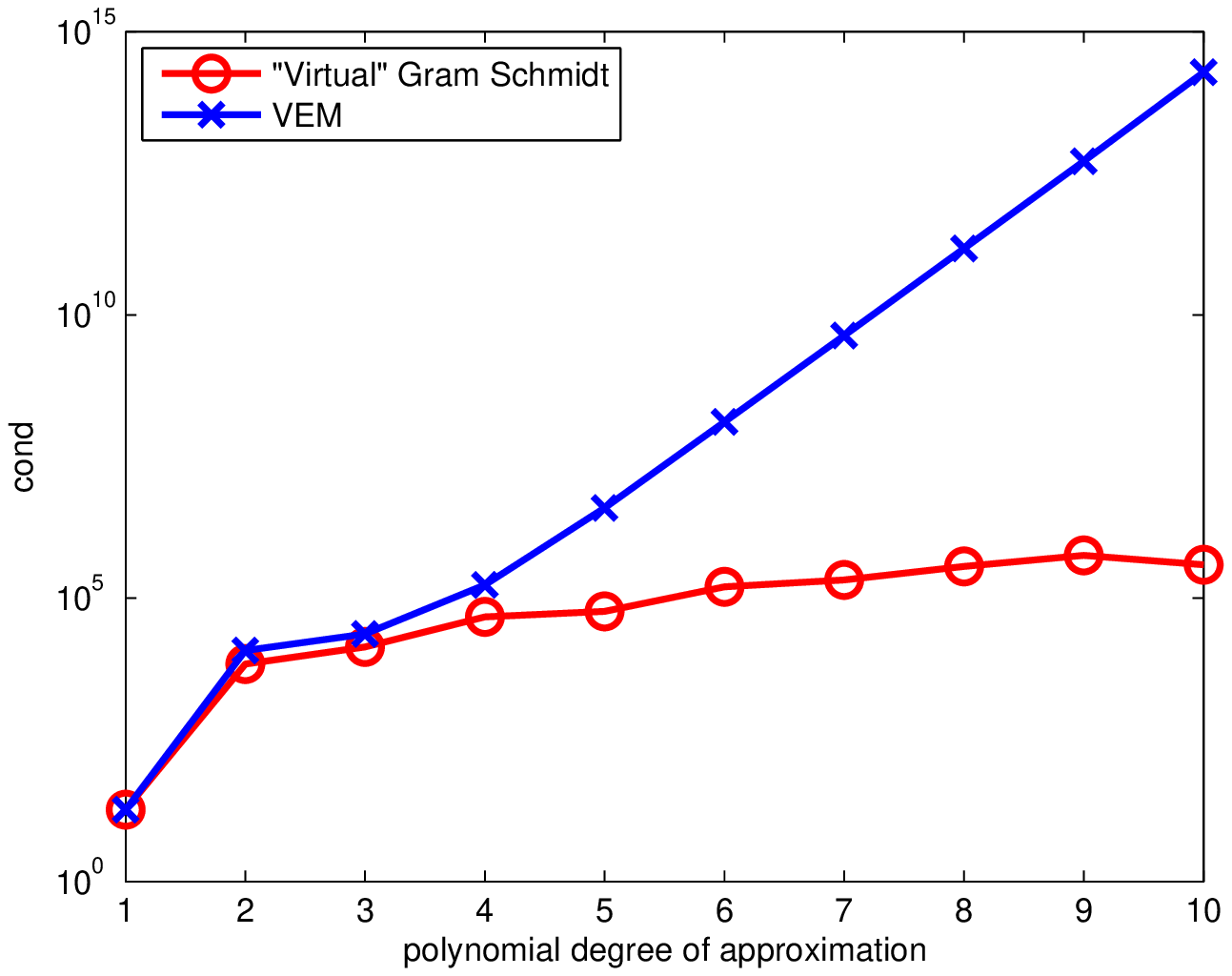}}
\subfigure {\includegraphics [angle=0, width=0.45\textwidth]{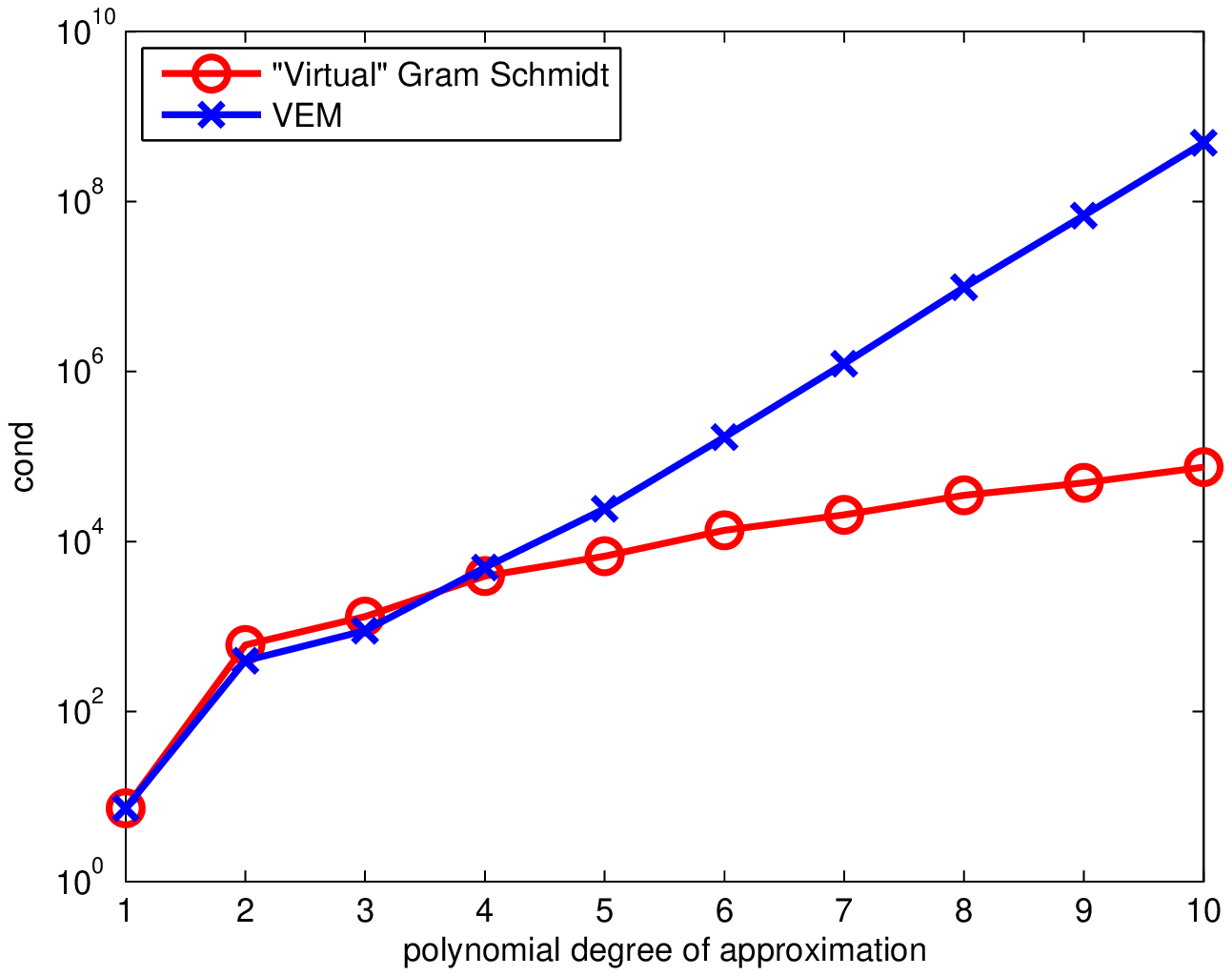}}
\caption{1:unstructured triangle mesh; 2: regular square mesh} \label{confrontoVEMeVGS_a}
\end{figure}
\begin{figure} [h]
\centering
\subfigure {\includegraphics[angle=0, width=0.45\textwidth]{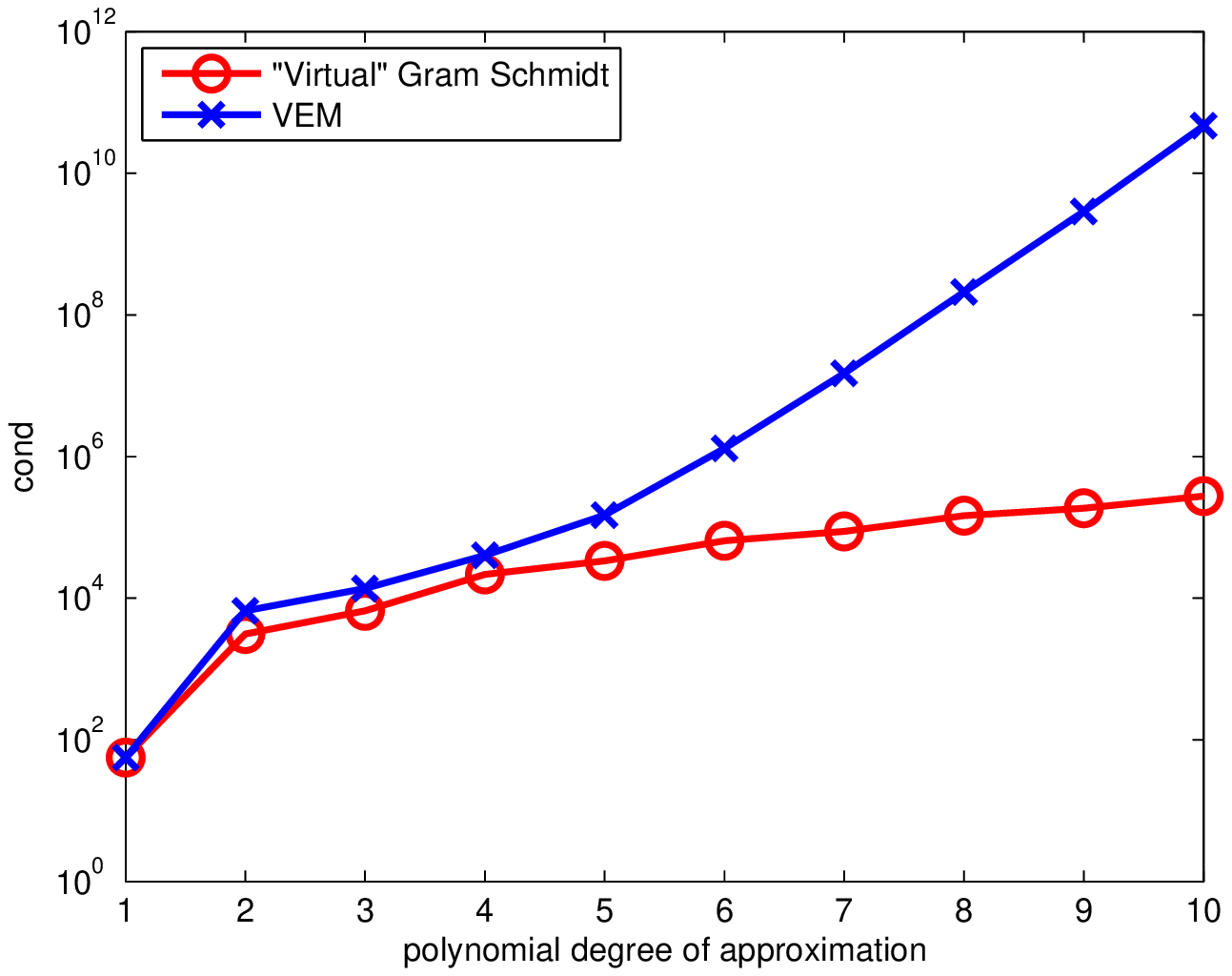}}
\subfigure {\includegraphics[angle=0, width=0.45\textwidth]{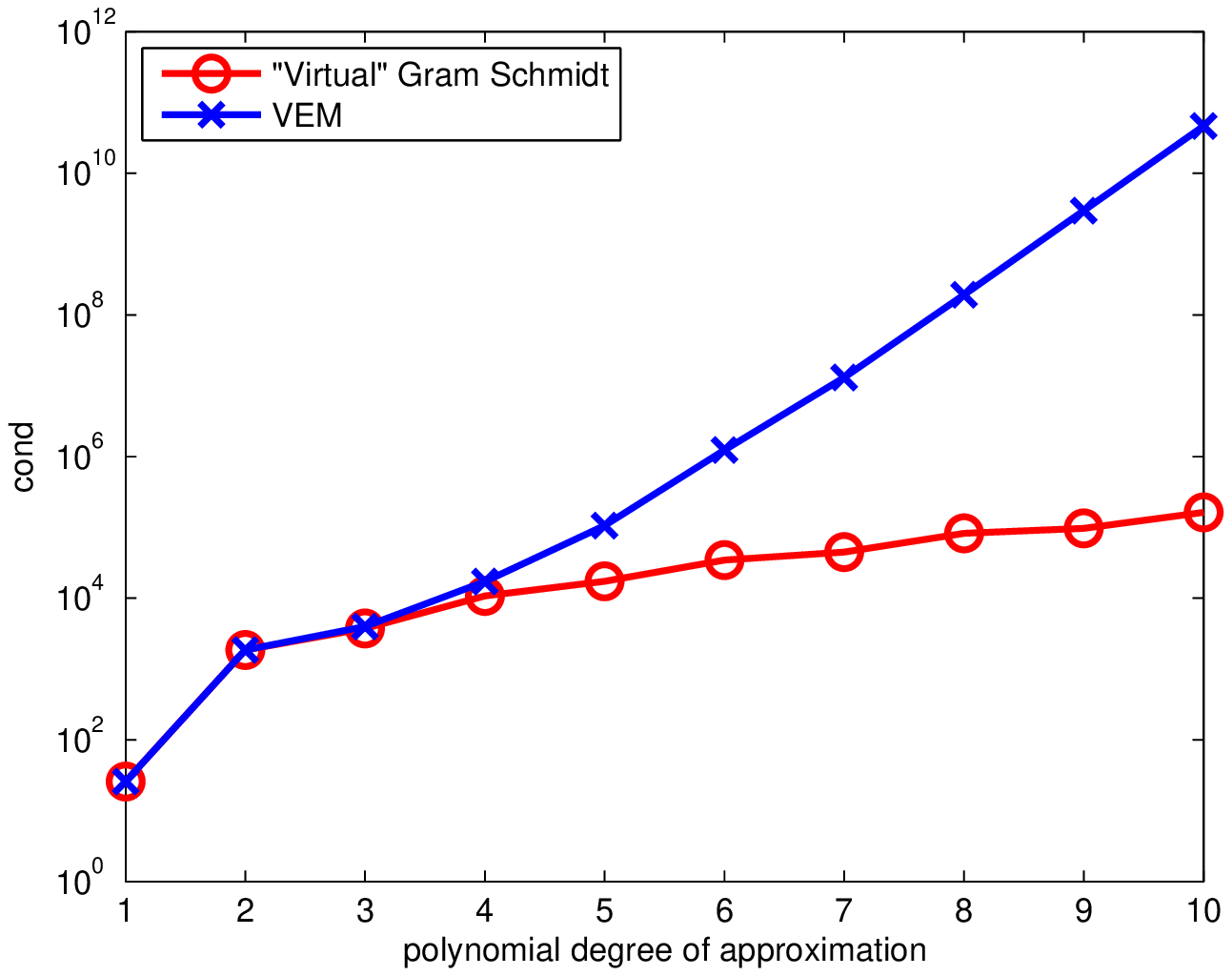}}
\caption{3: regular hexagonal mesh; 4: VoronoiLloyd mesh} \label{confrontoVEMeVGS_b}
\end{figure}
From the results in Figures \ref{confrontoVEMeVGS_a} and \ref{confrontoVEMeVGS_b}, it follows that the Gram Schmidt basis performs much better, at least for what concerns the condition number.\\
In Figures \ref{confrontoVEMeVGS_err_a} and \ref{confrontoVEMeVGS_err_b}, we compare the behaviour of the error $|u-\Pinablap u_h|_{1,\Omega}$ using the two bases above on the usual test case $u(x,y)=\sin(\pi x) \sin(\pi y)$.
\begin{figure} [h]
\begin{overpic}[angle=0, width=0.45\textwidth]{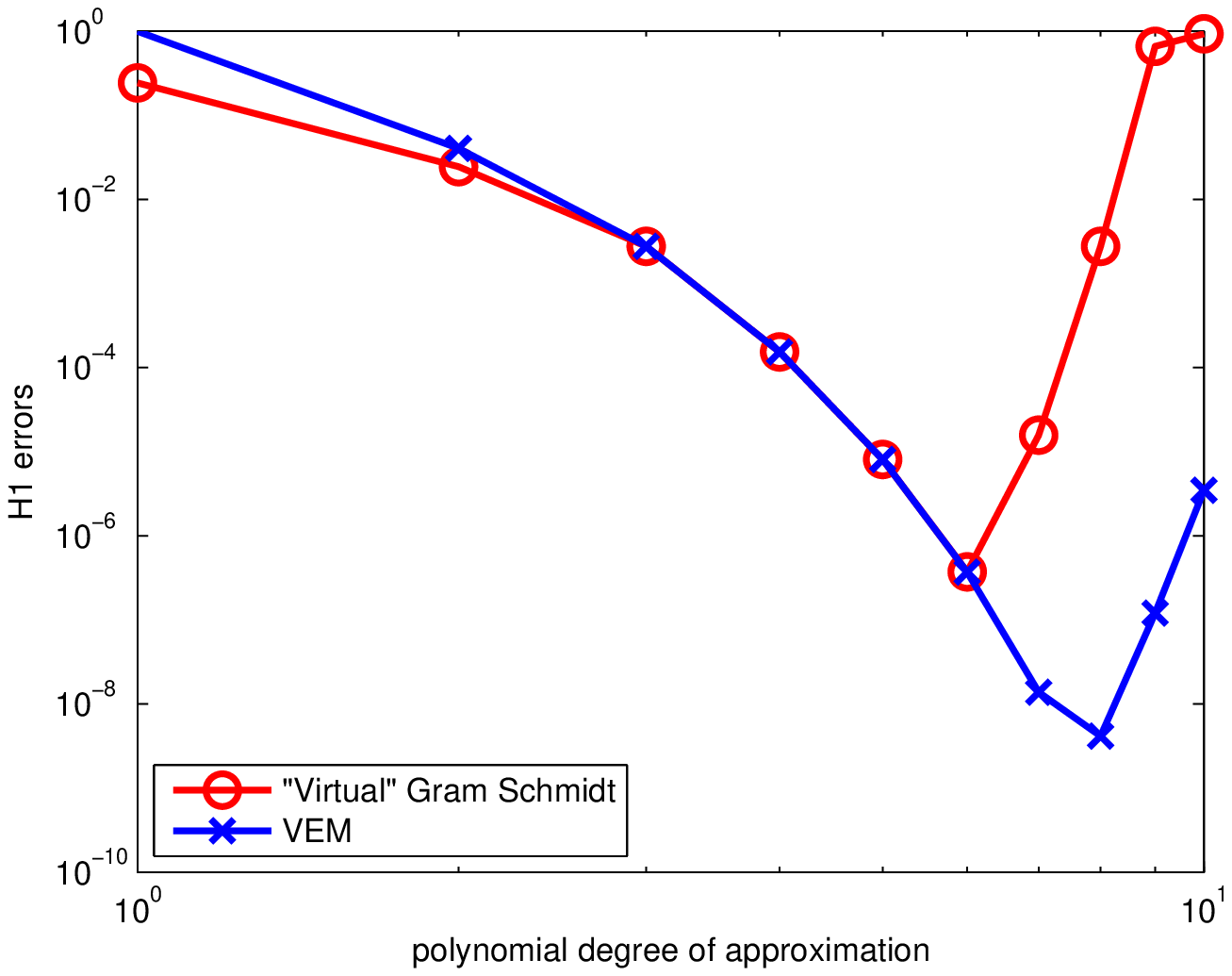}
\put(20,9.5){\colorbox{white}{\tiny 1}} \put(68,9.5){\tiny{2}} \put(94.5,9.5){\tiny{3}} \put(113,9.5){\tiny{4}} \put(127.5,9.5){\tiny{5}}
\put(139.5,9.5){\tiny{6}} \put(149,9.5){\tiny{7}} \put(158,9.5){\tiny{8}} \put(166,9.5){\tiny{9}} \put(170,9.5){\colorbox{white}{\tiny 10}}
\end{overpic}
\begin{overpic}[angle=0, width=0.45\textwidth]{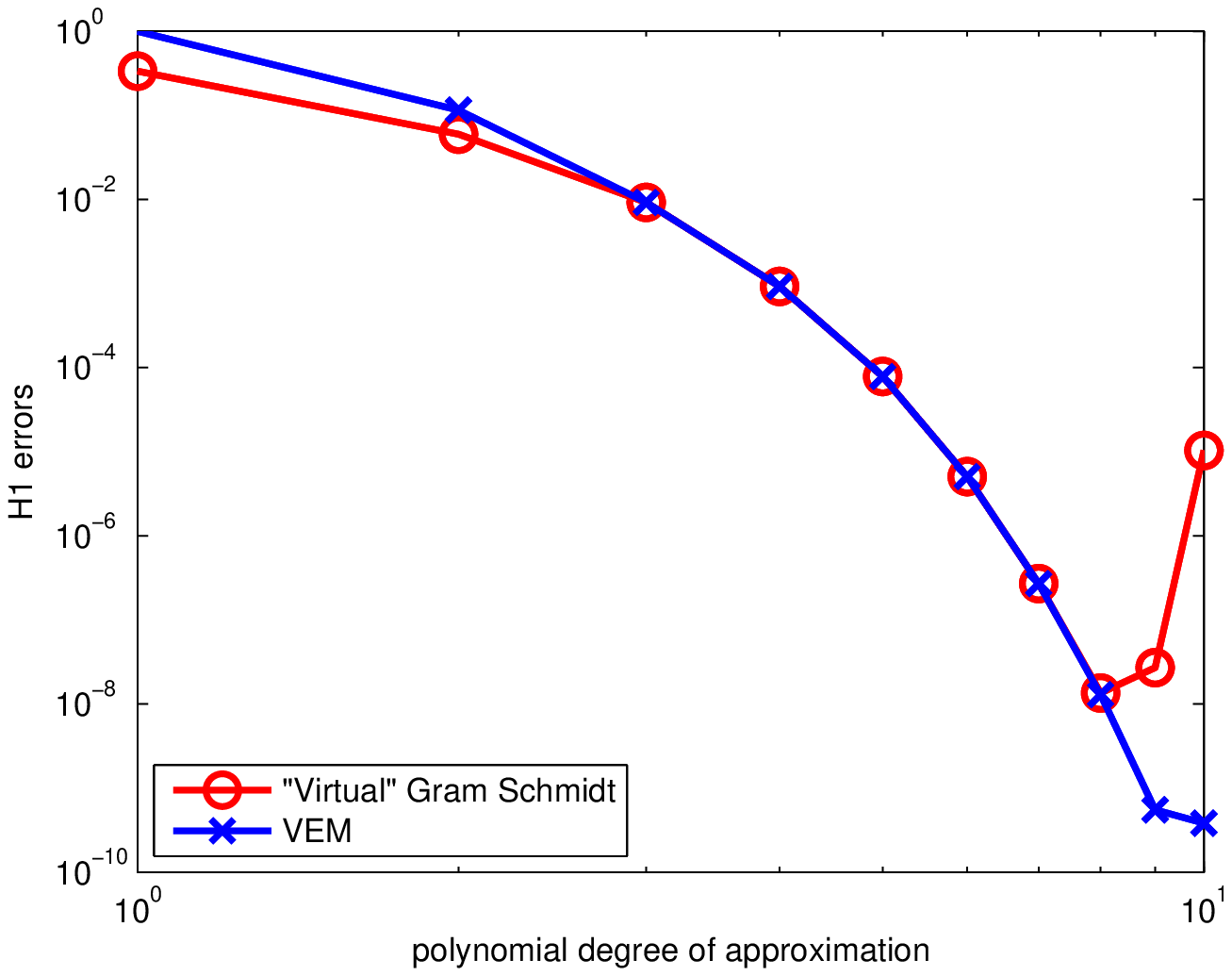}
\put(20,9.5){\colorbox{white}{\tiny 1}} \put(68,9.5){\tiny{2}} \put(94.5,9.5){\tiny{3}} \put(113,9.5){\tiny{4}} \put(127.5,9.5){\tiny{5}}
\put(139.5,9.5){\tiny{6}} \put(149,9.5){\tiny{7}} \put(158,9.5){\tiny{8}} \put(166,9.5){\tiny{9}} \put(170,9.5){\colorbox{white}{\tiny 10}}
\end{overpic}
\caption{1:unstructured triangle mesh; 2: regular square mesh} \label{confrontoVEMeVGS_err_a}
\end{figure}
\begin{figure} [h]
\begin{overpic}[angle=0, width=0.45\textwidth]{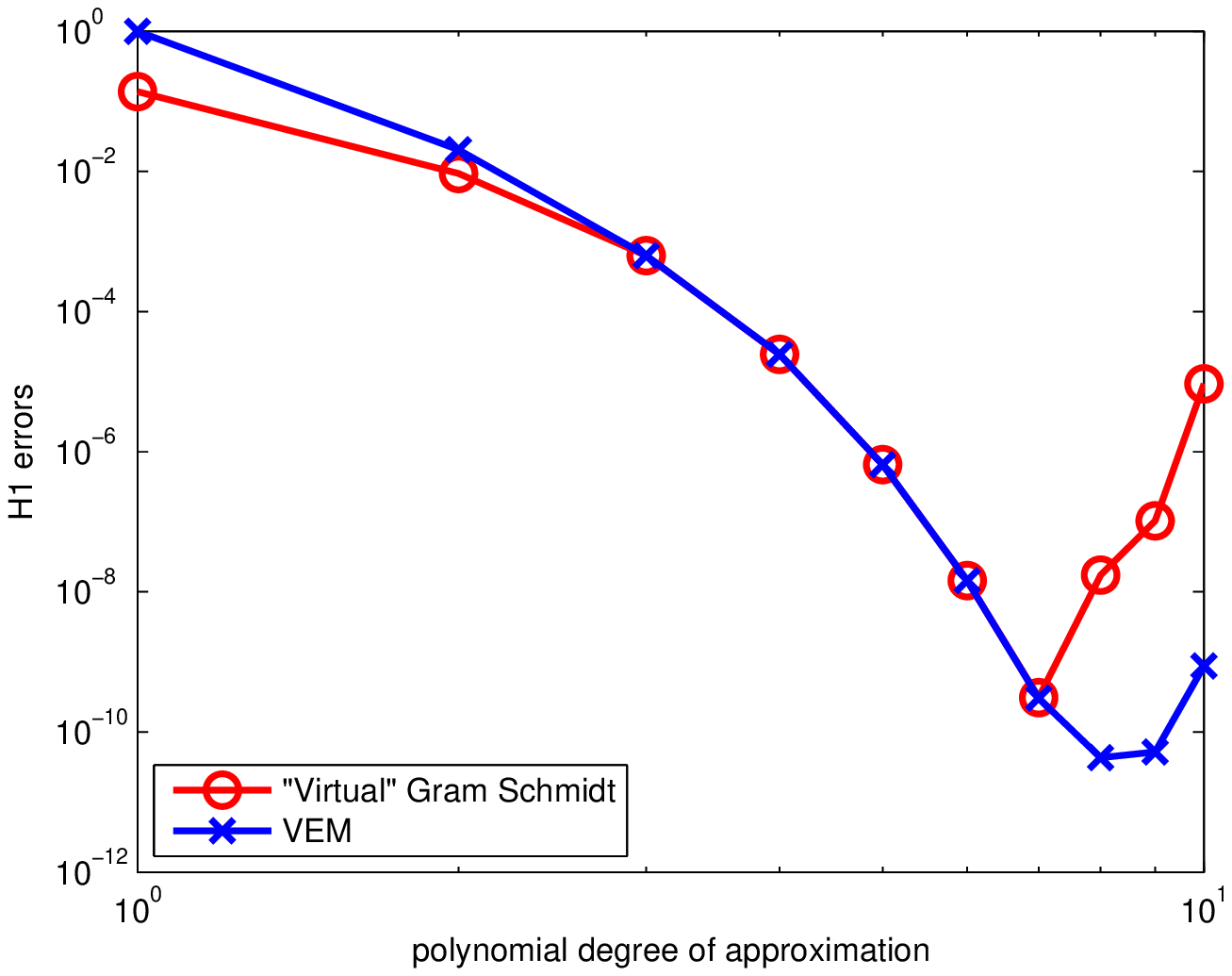}
\put(20,9.5){\colorbox{white}{\tiny 1}} \put(68,9.5){\tiny{2}} \put(94.5,9.5){\tiny{3}} \put(113,9.5){\tiny{4}} \put(127.5,9.5){\tiny{5}}
\put(139.5,9.5){\tiny{6}} \put(149,9.5){\tiny{7}} \put(158,9.5){\tiny{8}} \put(166,9.5){\tiny{9}} \put(170,9.5){\colorbox{white}{\tiny 10}}
\end{overpic}
\begin{overpic}[angle=0, width=0.45\textwidth]{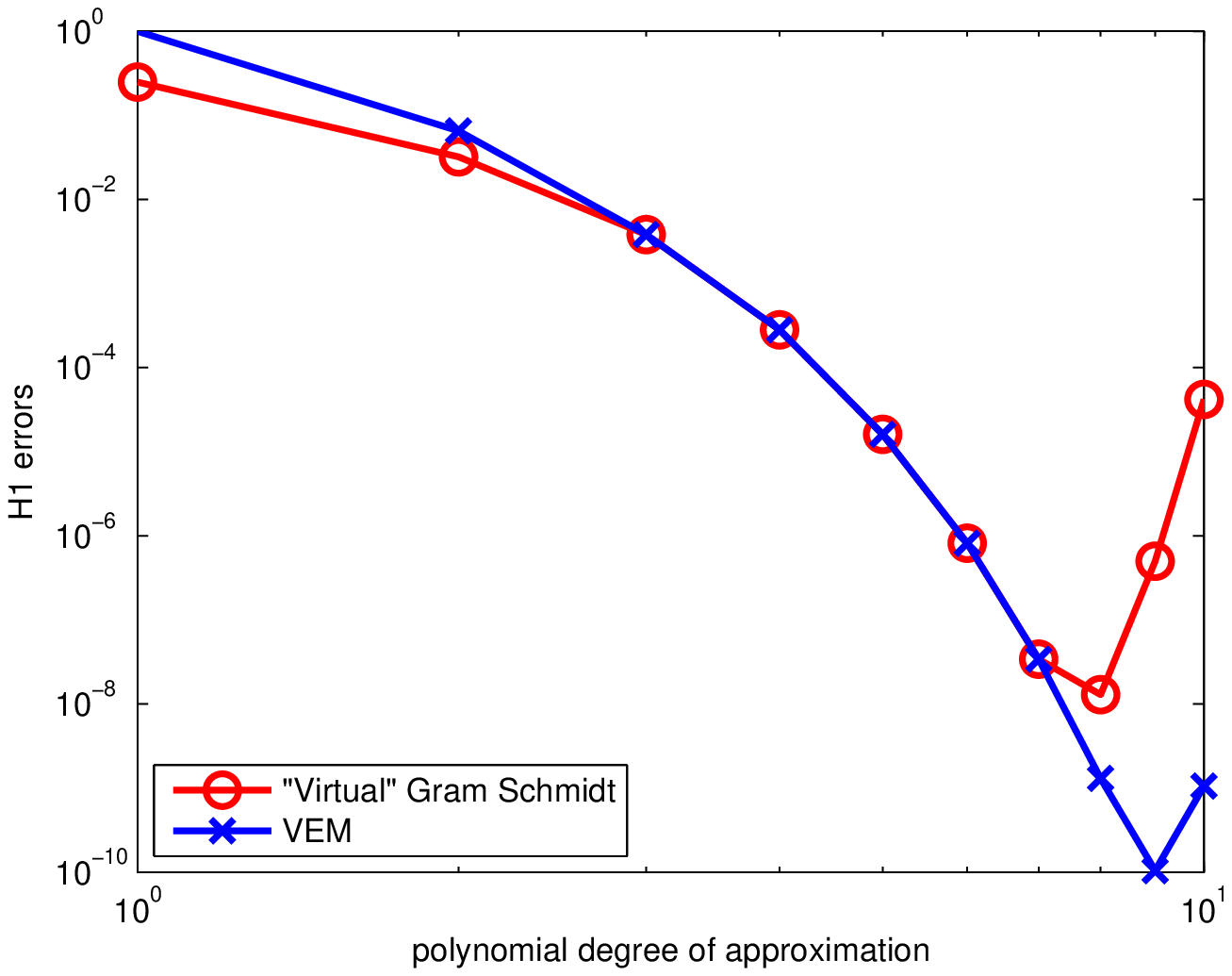}
\put(20,9.5){\colorbox{white}{\tiny 1}} \put(68,9.5){\tiny{2}} \put(94.5,9.5){\tiny{3}} \put(113,9.5){\tiny{4}} \put(127.5,9.5){\tiny{5}}
\put(139.5,9.5){\tiny{6}} \put(149,9.5){\tiny{7}} \put(158,9.5){\tiny{8}} \put(166,9.5){\tiny{9}} \put(170,9.5){\colorbox{white}{\tiny 10}}
\end{overpic}
\caption{3: regular hexagonal mesh; 4: VoronoiLloyd mesh} \label{confrontoVEMeVGS_err_b}
\end{figure}
We observe that, although the method described in this subsection improves the condition number of the global stiffness matrix, it is numerically unstable.
Therefore, in practice, the proposed Gram-Schmidt method may be preferable to the simple basis choice in \eqref{basisVEMvolleyscaled} only for mid-low values of $p$.

%%%%%%%%%%%%%%%%%%%%%%%%%%%%%%%%%%%%%%%%%%%%%%%%%%%%%%%%%%%%%%%%%%%%%%%%%%%
{\footnotesize
\bibliography{bibliogr}
}
\bibliographystyle{plain}
%%%%%%%%%%%%%%%%%%%%%%%%%%%%%%%%%%%%%%%%%%%%%%%%%%%%%%%%%%%%%%%%%%%%%%%%%%%

\end{document}